\documentclass[10pt,a4paper]{article}

\usepackage[utf8]{inputenc}
\usepackage[english]{babel}

\usepackage{graphicx}
\usepackage{subfig}

\usepackage{geometry}

\usepackage{tikz,tkz-tab}
\usepackage{array}

\usepackage{stmaryrd}
\usepackage{amsmath}
\usepackage{amssymb}
\usepackage{mathrsfs}
\usepackage{amsthm}
\usepackage{mathtools}
\usepackage[all]{xy} 

\usepackage{listings} 

\usepackage{enumitem} 

\usepackage{color}

\usepackage{pdfpages}

\usepackage{hyperref}

\newtheorem{theorem}{Theorem}[section]

\newtheorem{proposition}[theorem]{Proposition}
\newtheorem{lemma}[theorem]{Lemma}

\theoremstyle{definition}
\newtheorem{definition}[theorem]{Definition}

\theoremstyle{remark}
\newtheorem{remark}[theorem]{Remark}

\newcommand*{\C}{\mathcal{C}}
\newcommand*{\R}{\mathbb R}
\newcommand*{\N}{\mathbb N}

\renewcommand{\S}{\mathfrak{S}}
\newcommand{\A}{\mathcal{A}}
\newcommand{\D}{\mathcal{D}}

\newcommand{\e}{\mathrm e}

\newcommand*{\diff}{\mathop{}\!\mathrm{d}}

\newcommand{\hhat}{\hat{h}}
\newcommand{\h}{\bar{h}}
\newcommand{\xhat}{\hat{x}}
\newcommand{\zhat}{\hat{z}}
\newcommand{\sre}{\mathcal{E}}
\renewcommand{\L}{\Lambda}

\newcommand{\x}{\mathrm{x}}

\newcommand{\Jbar}{\bar{J}}

\newcommand{\eval}[2]{\left. #1\right|_{\begin{smallmatrix*}[l] #2\end{smallmatrix*}}}
\renewcommand{\ll}{\llbracket}
\newcommand{\rr}{\rrbracket}

\newcommand{\Span}{\mathrm{Span}}
\newcommand{\im}{ \mathrm{im}}
\newcommand{\Jac}{\mathrm{Jac}}

\title{Short geodesics losing optimality in contact sub-Riemannian manifolds and stability of the 5-dimensional caustic
}

\author{
Ludovic Sacchelli \thanks{Aix Marseille Univ, Université de Toulon, CNRS, LIS, Marseille, France ({\tt sacchelli@univ-tln.fr}).
This research has been supported by the ANR SRGI (reference ANR-15-CE40-0018).
The author would like to thank Grégoire Charlot and Luca Rizzi for the many fruitful discussions that lead to the present paper.}}

\date{\today}

\begin{document}

\maketitle
\begin{abstract}
We study the sub-Riemannian exponential for contact distributions on manifolds of dimension greater or equal to 5. We compute an approximation of the sub-Riemannian Hamiltonian flow and show that the conjugate time can have multiplicity 2 in this case. We obtain an approximation of the first conjugate locus for small radii and introduce a geometric invariant to show that the metric for contact distributions typically exhibits an original behavior, different from the classical 3-dimensional case.
We apply these methods to the case of $5$-dimensional contact manifolds. We provide a stability analysis of the sub-Riemannian caustic 
from the Lagrangian point of view and classify the singular points of the exponential map.

\end{abstract}


\section{Introduction}
\label{S:intro_short_geo}

Let $M$ be a smooth ($C^\infty$) manifold of dimension $2n+1$, with $n\geq 1$ integer. A \emph{contact distribution} is a   $2n$-dimensional vector sub-bundle $\Delta\subset TM$ that locally coincides with the kernel of a smooth $1$-form $\omega$ on $M$ such that $\omega\wedge (\diff \omega)^n\neq 0$.
 The sub-Riemannian structure on $M$ is given by a smooth scalar product $g$ on $\Delta$, and we call $(M,\Delta,g)$ a \emph{contact sub-Riemannian  manifold} (see, for instance, \cite{ABB_2018,agrachev_barilari_rizzi_contact_2016}).

The small scale geometry of general 3-dimensional contact sub-Riemannian manifolds is well understood but not much can be said for dimension $5$ and beyond, apart from the particular case of Carnot groups.  We are interested in giving a qualitative description of the local geometry of contact sub-Riemannian manifolds by describing the family of short locally-length-minimizing curves (or geodesics) originating from a given point. In the case of contact sub-Riemannian manifolds, all length-minimizing curves are projections of integral curves of an intrinsic Hamiltonian vector field on $T^*M$, and as such, geodesics are  characterized by their initial point and initial covector. 

By analogy with the Riemannian case, for all $q\in M$, we denote by $\sre_q$ the \emph{sub-Riemannian exponential}, that maps 
a covector $p\in T_{q}^*M$ to the evaluation at time $1$ of the geodesic curve starting at $q$ with initial covector $p$.
An essential observation on 
length minimizing curves in sub-Riemannian geometry is that there exist locally-length-minimizing curves that lose local optimality arbitrarily close to their starting point \cite{Diniz_veloso2009,HughenThesis,montgomery2006tour}.
 Hence the geometry of sub-Riemannian balls of small radii is inherently linked with the geometry of the 
conjugate locus, that is, at $q$, the set of points $\sre_q(p)$ such that $p$ is a critical point of $p\mapsto \sre_q(p)$, \cite{Barilari2013,barilari2018VolumeOfBalls,barilari2016heat}.

The sub-Riemannian exponential has a natural structure of Lagrangian map, since it is the projection of a Hamiltonian flow over $T^*M$,  and its conjugate locus is a Lagrangian caustic. In small dimension, this observation allows  the study of the stability of the caustic and the classification of singular points of the exponential from the point of view of Lagrangian singularities (see, for instance, \cite{Arnold_singularities_diff_maps1}). 

In the $3$-dimensional case, this analysis has been conducted,  with different approaches, in \cite{agrachev_1996_exponential} and \cite{Gauthier_1996_small_SR_balls}. These works describe asymptotics of the sub-Riemannian exponential, the conjugate and cut loci near the starting point (see also \cite{AgrachevCharlotGauthierZakalyukin}). The aim of the present work is to extend this study to higher dimensional contact sub-Riemannian manifolds, following the methodology developed in \cite{Gauthier_1996_small_SR_balls} and \cite{charlot_2002_quasi_contact} (the latter focusing on a similar study of quasi-contact sub-Riemannian manifolds). More precisely, we use a perturbative approach to compute approximations of the Hamiltonian flow. This is made possible by using a general well-suited normal form for contact sub-Riemannian structures. 
The normal form we use has been obtained in \cite{Gauthier_2001_SR_metrics_and_isoperimetric_problems}.
(We recall its properties in Appendix~\ref{A:Gauthier}.)

Finally, it can be noted that classical behaviors displayed by $3$-dimensional contact sub-Riemannian structures may not be typical in larger dimension. The 3-dimensional case is very rigid in the class of sub-Riemannian manifolds and appears to be so even in regard of contact sub-Riemannian manifolds of arbitrary dimension. Therefore, part of our focus is dedicated to highlighting the differences between this classical case and those of larger dimension.

\subsection{Approximation of short geodesics}

\paragraph*{Notation}In the following, for any two integers $m,n\in \N$, $m\leq n$, we denote by $\ll  m,n\rr $ the set of integers $k\in \N$ such that $m\leq k\leq n$.

Let $(M,\Delta,g)$ be a contact sub-Riemannian manifold of dimension $2n+1$, $n\geq 1$ integer.

\paragraph*{Invariants of the nilpotent approximation} 
Consider a $1$-form $\omega$ such that $\ker \omega=\Delta$ and  $\omega\wedge (\diff \omega)^n\neq 0$ ($\omega$ is not unique, this property holds for any $f \omega$ where $f$ is a non-vanishing smooth function). For all $q\in M$, there exists 
a linear map $A(q):\Delta_q\rightarrow \Delta_q$, skew-symmetric with respect to $g_q$, such that for all $X,Y\in \Delta$, $\diff \omega (X,Y)(q)=g_q (A(q) X(q),Y(q))$. Then the eigenvalues of  $A(q)$, $\{\pm i b_1,\dots, \pm i b_n\}$, are invariants of the sub-Riemannian structure at $q$ (up to a multiplicative constant). In the following, we will assume that the invariants $\{b_1,\dots, b_n\}\in \R^+$ are rescaled so that $b_1  \cdots   b_n=\frac{1}{n!}$.

These invariants are parameters of the metric tangent to the sub-Riemannian structure at $q$, or nilpotent approximation at $q$ (see \cite{bellaiche1996TagentSpace}), which admits a structure of Carnot group. 
Notice in particular that the nilpotent approximations of a contact sub-Riemannian structure at two points $q_1,q_2\in M$ may not be isometric if the dimension $2n+1$ is larger than $3$.

For a given $q\in M$, there always exists a set of coordinates $(x_1,\dots, x_{2n},z):\R^{2n+1}\to \R^{2n+1}$ such that a frame $\left(\widehat{X}_1,\dots,\widehat{X}_{2n}\right) $ of the nilpotent approximation at $q$ can be written in the normal form
$$
\widehat{X}_{2i-1}=\partial_{x_{2i-1}} +\frac{ b_i }{2}x_{2i}\partial_z,
\quad
\widehat{X}_{2i}=\partial_{x_{2i}}-\frac{ b_i }{2}x_{2i-1}\partial_z,
\qquad \forall i\in \ll1,n\rr.
$$
Geodesics of such contact Carnot groups can be computed explicitly, and their features have been extensively studied (see, for instance, \cite{IHP-Vol,gaveau1977principe,lerario2017many}).
The central idea we follow is that the sub-Riemannian structure at a point $q\in M$ can be expressed as a small perturbation of the nilpotent structure at $q_0$ for points $q$ close to $q_0$. An important tool we use is the Agrachev--Gauthier normal form, introduced in \cite{Gauthier_2001_SR_metrics_and_isoperimetric_problems}, which asserts, for any given $q_0\in M$, the existence of 
coordinates at $q_0$, $(x_1,\dots x_{2n},z):M\to \R^{2n+1}$, and a frame of $(\Delta,g)$, $(X_1,\dots ,X_{2n})$, such that 
$$
X_i(x,z) =\widehat{X}_i(x,z)+O\left(|x|^2\right).
$$

\paragraph{Asymptotics and covectors}
Let 
$H(p,q)= \frac{1}{2} \sup _{v\in \Delta_q\setminus \{0\}}\frac{{\langle p,v\rangle}^2}{g_q(v,v)}$
be the sub-Rie\-man\-nian Hamiltonian. 
For all $q\in M$, $H(\cdot,q)$ is a positive quadratic form on $T^*_{q}M$ of rank $2n$. Then for all $r>0$, the set $\{H(p,q)=r\mid p\in T_q^*M\}$ is an unbounded subset of $T_q^*M$ with the topology of the cylinder $\mathbb{S}^{2n-1}\times \R$ (see for instance \cite{ABB_2018,agrachev_barilari_rizzi_contact_2016}). 
In the following, for all $q\in M$ and $r\geq 0$, we denote this set by
$$
\mathcal{C}_{q}(r)=\{H(p,q)=r\mid p\in T_q^*M\}.
$$
Abusing notations, for  $V\subset \R^+$, we denote $\mathcal{C}_{q}(V)=\cup_{r\in V}\C_{q}(r)$.
We choose coordinates $p=(h,h_0)$ on $T_q^*M$ where for a given $r>0$,  $h_0$ denotes the unbounded component of $p\in \C_{q}(r)$. 

An important observation is that in the nilpotent case, geodesics losing optimality near their starting point correspond to  initial covectors in $\C_{q}(r)$ such that $|h_0|/r$ is very large (see, for instance, \cite{barilari2012_small_time_heat_kernel,Biggs2016}). The expansions obtained in this paper rely on the same type of  asymptotics.

Section~\ref{S:Normal_extremals} is dedicated to the computation of an approximation of the flow of the Hamiltonian vector field $\vec{H}$ as $h_0\to \infty$. Since $\vec{H}$ is a quadratic Hamiltonian vector field, its integral curves satisfy the symmetry
$$
\e^{t\vec{H}}(p_0,q_0)=\e^{\vec{H}}(t p_0,q_0) ,\qquad \forall q_0\in M, p_0\in T_{q_0}^*M, t\in \R.
$$
Hence it is useful for us to consider the time-dependent exponential that maps the pair $(t,p)\in \R\times \C_{q}(1/2)$ to the geodesic of initial covector $p$ evaluated at time $t$.
Using the approximation of the Hamiltonian flow as $h_0\to \infty$, Section~\ref{S:conjugate_time} is dedicated to the computation of the conjugate time. For a given $q\in M$, the conjugate time $t_c(p)$ is the smallest positive time such that $\sre_q(t_c(p),\cdot)$ is critical at $p$. The computation of the conjugate locus follows once the conjugate time is known.

Notice in particular that for a given initial covector $p\in  \C_{q}(1/2)$, $t_c(p)$ is then an upper bound of the sub-Riemannian distance between $q$ and $\sre_q(t_c(p),p)$ (and we have equality if $\sre_q(t_c(p),p)$ is also in the cut locus).

In the 3D case, it is proven in \cite{agrachev_1996_exponential,Gauthier_1996_small_SR_balls} that for an initial covector $(\cos \theta , \sin \theta ,h_0)\in \C_{q}(1/2)$, the conjugate time at $q$ satisfies as $h_0\to \pm\infty$
\begin{equation}\label{E:tc_expansion_3D}
t_c(\cos \theta , \sin \theta ,h_0)=\frac{2\pi}{|h_0|}-\frac{\pi \kappa }{|h_0|^3}+O\left(\frac{1}{|h_0|^4}\right)
\end{equation}
and the first conjugate point satisfies (in well chosen adapted coordinates at $q$)
$$
\sre_q(t_c(\cos \theta , \sin \theta ,h_0),(\cos \theta , \sin \theta ,h_0))
=
\pm\frac{\pi }{|h_0|^2}(0,0,1)
\pm 
\frac{2\pi \chi}{|h_0|^3}(-\sin^{3}\theta,  \cos^3\theta,  0)+O\left(\frac{1}{|h_0|^4}\right).
$$

The analysis we carry in Sections~\ref{S:Normal_extremals} and \ref{S:conjugate_time} aims at generalizing such expansions.
(Notice that we  focus only on the case $h_0\to +\infty$ but the case $h_0\to -\infty$ is the same.)
 Our results, however,  provide an important distinction between the classical 3D contact case and higher dimensional ones. Indeed, a very useful fact in the analysis of the geometry of the 3D case is that 
a 3D sub-Riemannian contact structure is very well approximated by its nilpotent approximation (as exemplified in \cite{Barilari2013}, for instance). 

This can be illustrated by using the 3D version of the Agrachev--Gauthier normal form, as introduced in \cite{Gauthier_1996_small_SR_balls}.
Let us denote by $\widehat{\sre}_q$ the exponential of the nilpotent approximation of the sub-Riemannian structure at $q_0$ in normal form.
Then as $h_0\rightarrow +\infty$, we have the expansion
\begin{equation}\label{E:sre_exp_3D}
\sre_q(\tau/h_0,(h_1,h_2,h_0))=\widehat{\sre}_q(\tau/h_0,(h_1,h_2,h_0))+O\left(\frac{1}{h_0^3}\right).
\end{equation}
As a result, one immediately obtains a rudimentary version of expansion~\eqref{E:tc_expansion_3D},
\begin{equation}\label{E:tc_expansion_3D_prim}
t_c(\cos \theta , \sin \theta ,h_0)=\frac{2\pi}{|h_0|}+O\left(\frac{1}{|h_0|^3}\right).
\end{equation}

However, expansion~\eqref{E:tc_expansion_3D_prim} is not true in general when we consider contact manifolds of dimension larger than $3$ (that is, the conjugate time  is not a third order perturbation of the nilpotent conjugate time $2\pi/|h_0|$). As an application of Theorem~\ref{T:big_expansion_theorem}, which gives a general second order approximation of the conjugate time in dimension greater or equal to $5$, we are able to prove  that the expansion~\eqref{E:sre_exp_3D} cannot be generalized.

In the rest of this paper, statements refer to generic ($d$-dimensional) sub-Rieman\-nian contact manifolds in the following sense: 
such statements  hold for contact sub-Riemannian metrics  in a countable intersection of open and dense sets of the space of smooth ($d$-dimensional) sub-Riemannian contact metrics endowed with the $C^3$-Whitney topology. As an application of  transversality theory, we then prove statements holding on the complementary of stratified subsets of codimension $d'$ of the manifolds, locally unions of finitely many submanifolds of codimension $d'$ at least.
\begin{theorem}\label{T:sre_is_not_so_simple}
Let $(M,\Delta,g)$ be a generic contact sub-Riemannian manifold. 
\\
There exists a codimension $1$ stratified subset $\S$  of $M$ such that for all $q\in M\setminus \S$, for all linearly adapted coordinates at $q$ and for all $T>0$,
\begin{equation}
\label{E:limit_exp-exphat}
\limsup\limits_{h_0\to +\infty}
\left(
h_0^2 \sup_{\tau\in (0,T)}
\left|
	\sre_{q}\left(\frac{\tau}{h_0},(h_1,\dots ,h_{2n},h_0)\right)-\widehat{\sre}_{q}\left(\frac{\tau}{h_0},(h_1,\dots ,h_{2n},h_0)\right)
\right|
\right)
> 0.
\end{equation}
\end{theorem}

This observation needs to be put in perspective with some already observed differences between 3D contact sub-Riemannian manifolds and those of greater dimension. For a given $1$-form $\omega$ such that $\ker \omega=\Delta$ and  $\omega\wedge (\diff \omega)^n\neq 0$, the Reeb vector field is the unique vector field $X_0$ such that $\omega(X_0)=1$ and $\iota_{X_0}\diff \omega=0$. The contact form $\omega$ is not unique (for any smooth non-vanishing function $f$, $f\omega$ is also a contact form),  and neither is $X_0$. In 3D however, the conjugate locus lies tangent to a single line that carries a Reeb vector field that is deemed canonical.
In larger dimension, this uniqueness property is not true in general.
For this reason, we introduce a geometric invariant that plays a similar role in measuring how the conjugate locus lies with respect to the nilpotent conjugate locus and use it to prove Theorem~\ref{E:limit_exp-exphat}.

The main difference seems to be a lack of symmetry in greater dimensions. Indeed the existence of a unique Reeb vector field (up to rescaling) points toward the idea of a natural $\mathrm{SO}(2n)$ symmetry of the nilpotent structure. 
However the actual symmetry of a contact sub-Riemannian manifold (or rather its nilpotent approximation) is $\mathrm{SO}(2)^n$ (on the subject, see, for instance, \cite{Gauthier_2001_SR_metrics_and_isoperimetric_problems}). Of course, when $n=1$,  $\mathrm{SO}(2)^n= \mathrm{SO}(2n)$.
More discussions on this issue can also be found in \cite{boscain2015intrinsic}.

\subsection{Stability in the 5-dimensional case}

We wish to apply these asymptotics to the study of stability of the caustic in the $5$-dimensional case. This study has been carried for 3-dimensional contact sub-Riemannian manifolds in \cite{Gauthier_1996_small_SR_balls} and for 4-dimensional quasi-contact sub-Riemannian manifolds in \cite{charlot_2002_quasi_contact}. To understand the interest of stability in the sense of sub-Riemannian geometry in small dimension, we must first understand stability from the point of view of Lagrangian manifolds. (See, for instance, \cite[Chapters 18, 21]{Arnold_singularities_diff_maps1} and also \cite{Bennequin_caustiques_mystiques,izumiya2016differential}.)

 Let $(E,\sigma)$ be a $2d$-dimensional symplectic manifold. A smooth submanifold $L$ of $M$ is said to be a \emph{Lagrangian submanifold} if $L$ is $d$-dimensional and $\sigma_{|L}=0$. The fiber bundle  $\pi:E\to N$ is said to be a \emph{Lagrangian fibration} if its fibers are Lagrangian submanifolds. 
For $L$ a Lagrangian submanifold of $E$ and $i :L\to E$ an immersion of $L$ into $E$ such that $i^*\sigma=0$, the map $\pi\circ i:L\to N$ is called a \emph{Lagrangian map}.

Let $(E,\sigma)$, $(E',\sigma')$ be two symplectic structures, let $\pi:E\to N$, $\pi' : E'\to N'$  be two Lagrangian fibrations. Two Lagrangian maps $\pi \circ i:L\to N$,  $\pi' \circ i':L'\to N'$ are said to be Lagrange equivalent if there exists two diffeomorphisms $\Phi:E\to E'$ and $\phi:N\to N'$ such that
$\Phi^*\sigma'=\sigma$, $\pi'\circ \Phi=\phi\circ \pi$ (the two Lagrangian fibrations are Lagrange equivalent) and 
$\Phi\circ i(L)=i'(L')$.

The \emph{caustic} of a Lagrangian map  is the set of its critical values. A consequence of the definition of Lagrangian equivalence is that 
if two  Lagrangian maps are Lagrange equivalent then their caustics are diffeomorphic.

A Lagrangian map $f:L\to N$ is said to be \emph{(Lagrange-)stable at $q\in L$} if there exists a neighborhood $V_q$ of $q$ and a neighborhood $V_f$ of $f_{|V_q}$ for the Whitney  $C^{\infty}$-topology such that any Lagrangian map $g\in V_f$ is Lagrange equivalent to $f$ (see \cite{charlot_2002_quasi_contact}).
In the following we may refer to points of a caustic as stable when they are critical values of a stable Lagrangian map.

For dimensions $d\leq 5$,  there exists only a finite number of equivalence classes for stable singularities of Lagrangian maps (for instance, one can find a summary in \cite[Theorem 2]{barilari2016heat}).
\begin{theorem}[Lagrangian stability in dimension $5$]
A generic Lagrangian map $f:\R^5\to \R^5$ has only stable singularities of type $\A_{2},\dots ,\A_6$, $\D_4^{\pm},\D_5^{\pm},\D_6^{\pm}$ and $\mathcal{E}_6^\pm$.
\end{theorem}

Sub-Riemannian exponential maps form a subclass of Lagrangian maps and we can define sub-Riemannian stability as Lagrangian stability restricted to the class of sub-Riemannian exponential maps.
Notouriously, the point $q_0$ is an unstable critical value of the sub-Riemannian exponential $\sre_{q_0}$, as the starting point of the geodesics defining $\sre_{q_0}$.

We focus our study of the stability of the sub-Riemannian caustic on the first conjugate locus.
This work can be summarized  in the following theorem (see also Figures~\ref{F:Section_caustic_A4_intro}, \ref{F:Section_caustic_D4+_intro}). 
\begin{theorem}[Sub-Riemannian stability in dimension $5$]\label{T:Stability}
Let $(M,\Delta,g)$ be a generic 5-dimensional contact sub-Riemannian manifold. There exists a stratified set $\S\subset M$ of codimension $1$ for which all $q_0\in M\setminus \S$ admit an open neighborhood $V_{q_0}$  such that for all $U$ open neighborhood of $q_0$ small enough, the intersection of the interior of the first conjugate locus at $q_0$ with $V_{q_0}\setminus U$ is (sub-Riemannian) stable and has only Lagrangian singularities of type $\mathcal{A}_2$, $\mathcal{A}_3$, $\mathcal{A}_4$, $\mathcal{D}_4^+$ and $\A_5$.

\end{theorem}

\begin{figure}[h]
\begin{center}
\includegraphics[width=12cm]{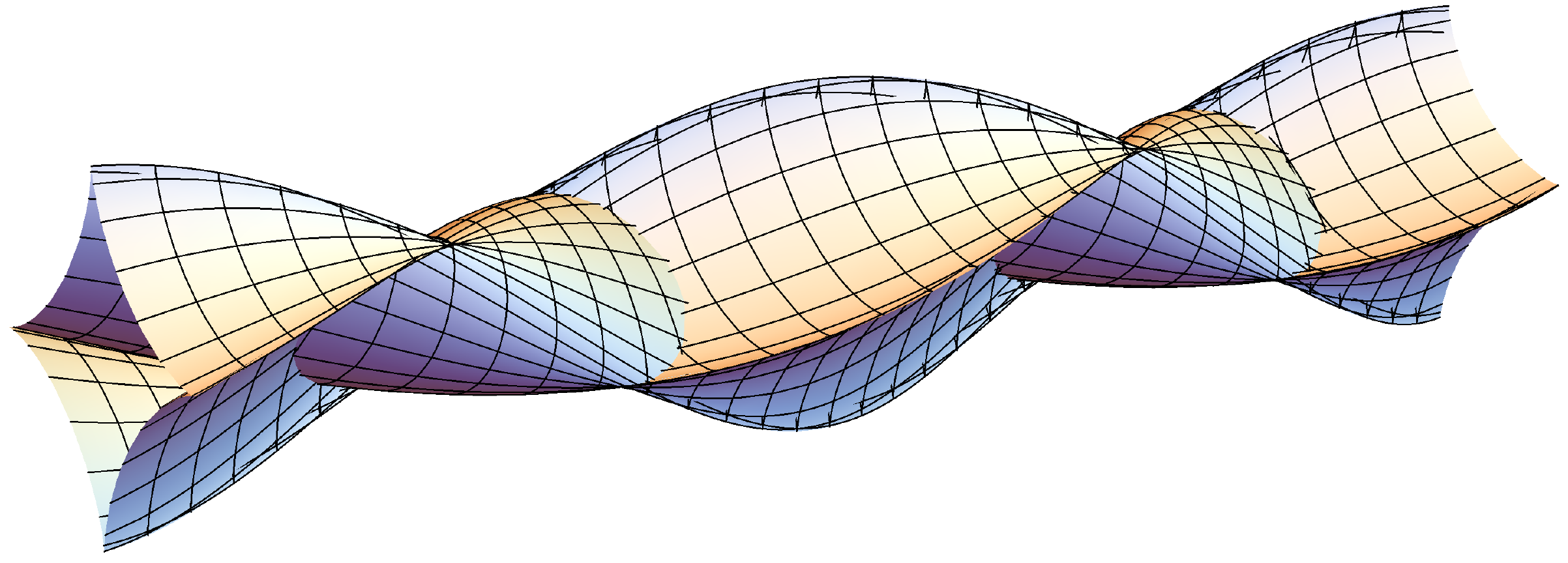}
\end{center}
\caption[Section of the caustic of a $5$-dimensional sub-Riemannian manifold, at a point of the manifold chosen so that it exhibits $\A_4$ singularities]{Section of the caustic of a $5$-dimensional sub-Riemannian manifold, at a point of the manifold chosen so that it exhibits $\A_4$ singularities.
This representation is obtained after sectioning by the hyperplanes $\{z=  z_0\}$, $\{x_3=R_2 \cos \omega\}$, $\{x_4=R_2\sin \omega\}$ (all in Agrachev--Gauthier normal form coordinates), and plotting for all $\omega\in [0,2\pi)$, with fixed $z_0, R_2>0$.
}
\label{F:Section_caustic_A4_intro}
\end{figure}

\begin{figure}[htb]
\begin{center}
\begin{minipage}{6cm}
\begin{center}
\includegraphics[width=5cm]{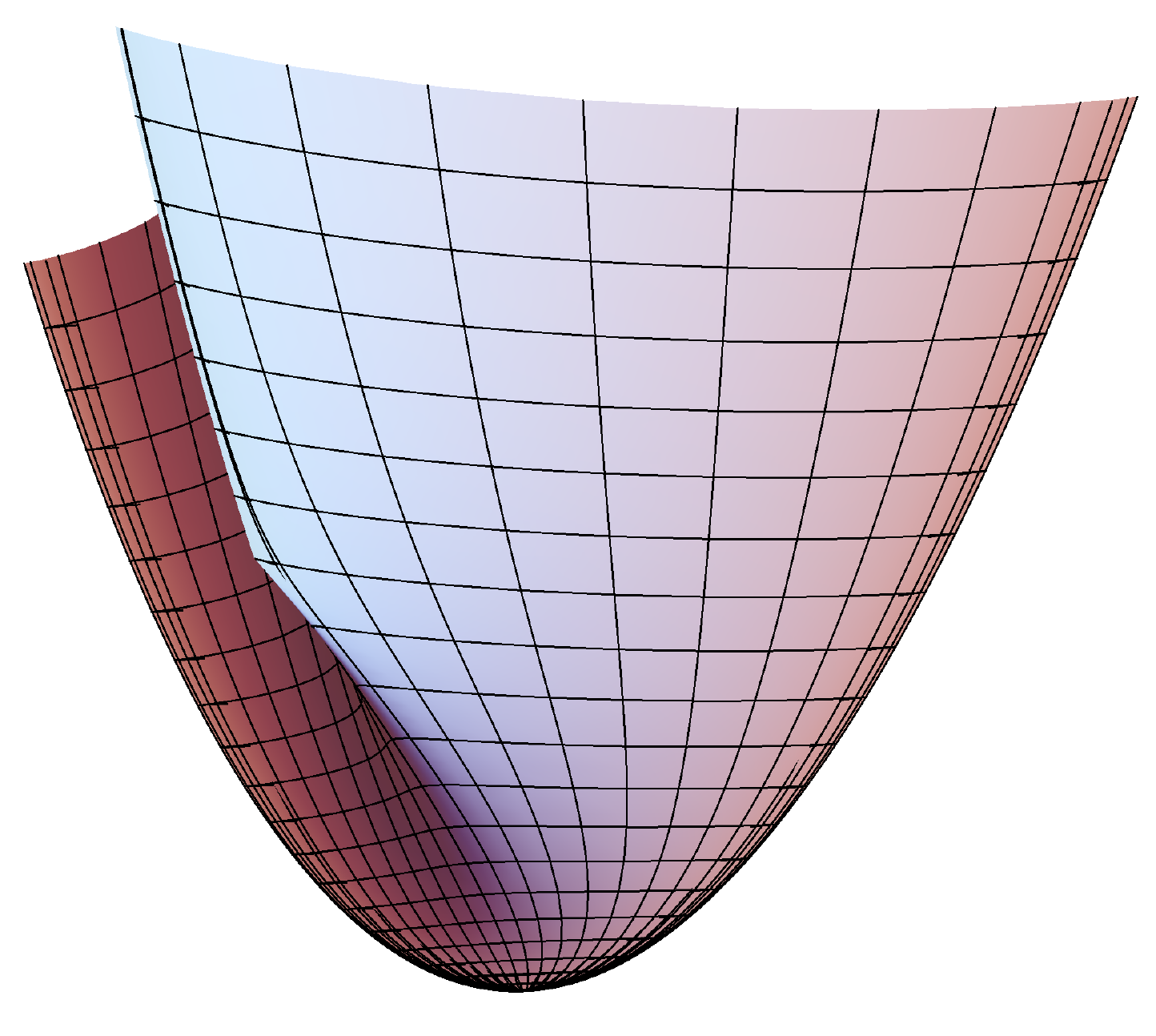}
\end{center}
\end{minipage}
\hspace{.5cm}
\begin{minipage}{6cm}
\begin{center}
\includegraphics[width=4.5cm]{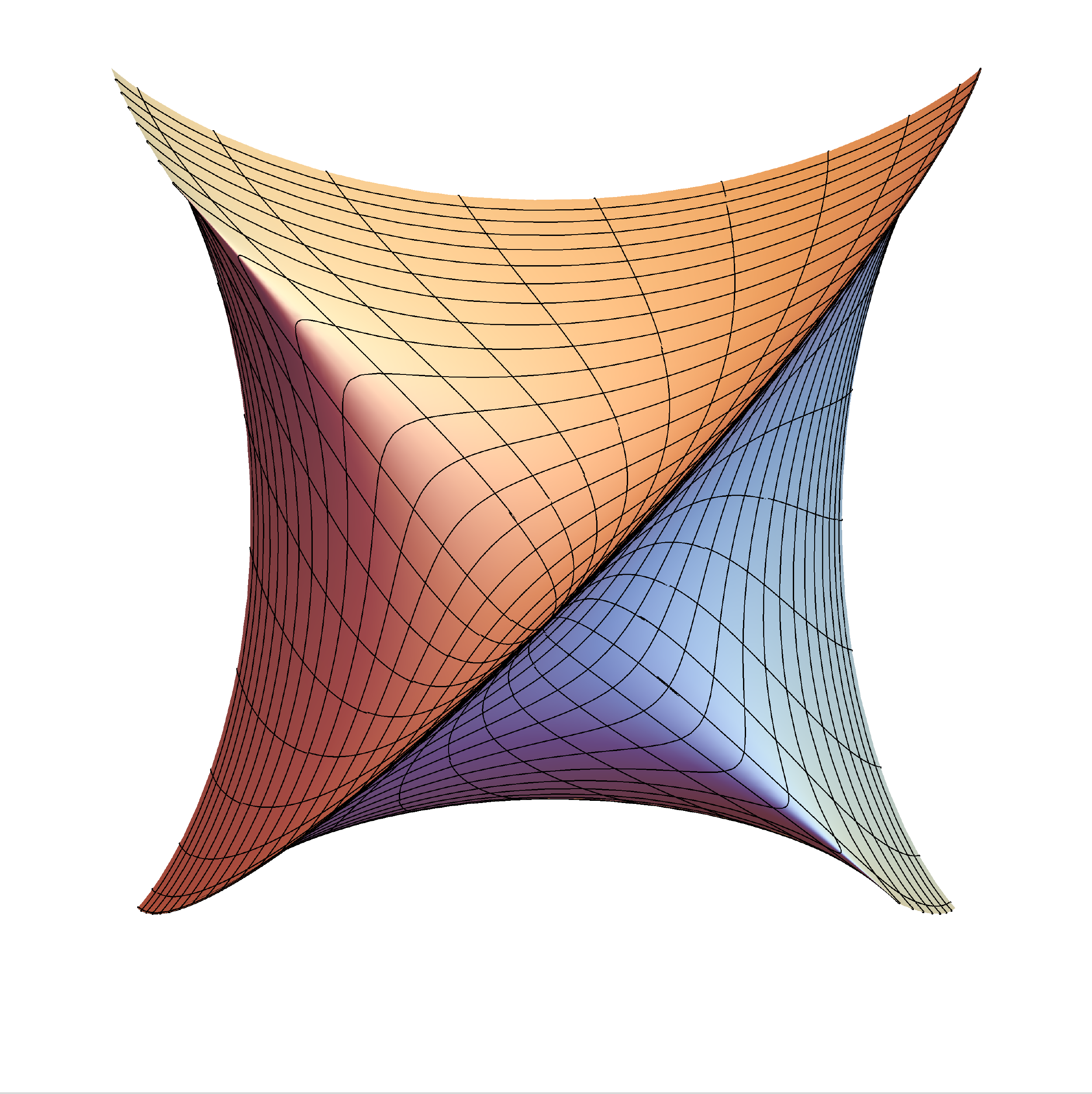}
\end{center}
\end{minipage}
\end{center}

\caption[Section of the caustic of a $5$-dimensional sub-Riemannian manifold, at a point of the manifold chosen so that it exhibits $\mathcal{D}_4^+$ singularities]{Section of the caustic of a $5$-dimensional sub-Riemannian manifold, at a point of the manifold chosen so that it exhibits $\mathcal{D}_4^+$ singularities.
This representation is obtained after sectioning by the hyperplanes $\{z=  z_0\}$, $\{x_3=R_2 \cos \omega\}$, $\{x_4=R_2\sin \omega\}$, and plotting for all $z_0\in [0,\bar{z}_0]$, with fixed $\bar{z}_0, R_2,\omega>0$.}

\label{F:Section_caustic_D4+_intro}

\end{figure}

This result stands on two foundations. On the one hand, a careful study of the problem of conjugate points in contact sub-Riemannian manifolds, and on the other hand, a stability analysis from the point of view of Lagrangian singularities in small dimension.

\subsection{Content}
In Section~\ref{S:Normal_extremals}, we compute an approximation of the exponential map for small time and large $h_0$ (Proposition~\ref{P:expansion}). Using the Agrachev--Gauthier normal form (recalled in Supplementary Materials~\ref{A:Gauthier}), the exponential appears to be a small perturbation of the standard nilpotent exponential. 

Section~\ref{S:conjugate_time}-\ref{S:order_2_approx} are dedicated to the approximation of the conjugate time (as summarized in Theorem~\ref{T:big_expansion_theorem}),   from which an approximation of the conjugate locus can be obtained. A careful analysis of the conjugate time for the nilpotent approximation shows that, under some conditions, the second conjugate time accumulates on the first (Section~\ref{SS:First_approximation}).
We rely on this observation to compute a second order approximation of the conjugate time (Section~\ref{S:order_2_approx}) and treat the problem of a double conjugate time via blow-up (Section~\ref{SS:asymptotics_near_S_1}). 
With the aim of proving stability of the caustic, we conclude the section by computing a third order approximation of the conjugate time for a small subset of initial covectors (Section~\ref{SS:third_order_conjugate}).

Hence we have devised three domains of initial convectors where a stability analysis must be carried (Section~\ref{S:Stab}). We show that we can tackle this analysis relying on  a Lagrangian equivalence classification (Section~\ref{SS:SRtoLagrangian})  and show that only stable Lagrangian singularities appear on the three domains (Section~\ref{SS:classification_three}).

\section{Normal extremals}\label{S:Normal_extremals}
\subsection{Geodesic equation in perturbed form}
In this section we establish the dynamical system satisfied by geodesics in terms of small perturbations of the nilpotent case.

Let $(M,\Delta,g)$ be a $(2n+1)$-dimensional contact sub-Riemannian manifold. Let $V$ be an open subset of $M$ and $(X_1,\dots ,X_{2n})$ be a frame of $(\Delta,g)$ on $V$, that is, a family of vector fields such that  $g_q(X_i(q),X_j(q))=\delta_i^j$ for all $i,j\in \ll1,2n\rr$ and all $q\in V$ (such a family always exists for $V$ sufficiently small). The sub-Riemannian Hamiltonian can be written
$$
H(p,q)= \frac{1}{2} \sum_{i=1} ^{2n}\langle p, X_i(q)\rangle^2.
$$ 
In the case of contact distributions, 
locally-length-minimizing curves are projections of normal extremals, the integral curves of the Hamiltonian vector field $\vec{H}$ on $T^*M$ (see for instance \cite{ABB_2018,agrachev_barilari_rizzi_contact_2016}). In other words, a  normal extremal $t\mapsto (p(t),q(t))$ satisfies  in coordinates the Hamiltonian ordinary differential equation
\begin{equation}\label{E:ham_ode}
\left\{
\begin{aligned}
&\frac{\diff q}{\diff t}=\sum_{i=1}^{2n}\langle p, X_i(q)\rangle X_i(q),
\\
&\frac{\diff p}{\diff t}=-\sum_{i=1}^{2n}\langle p, X_i(q)\rangle \; {}^tp \,D_q X_i(q).
\end{aligned}
\right.
\end{equation}

For $V$ sufficiently small, we can arbitrarily choose  a non-vanishing vector field $X_0$ transverse to $\Delta$ in order to complete $(X_1(q),\dots ,X_{2n}(q))$ into a basis of $T_qM$ at any point $q$ of $V$.
We use the family $(X_1,\dots, X_{2n},X_0)$ to endow $T^*M$ with dual coordinates $(h_1,\dots,h_{2n},h_0)$ such that
$$
h_i(p,q)=\langle p,X_i(q)\rangle \qquad \forall i\in \ll0,2n\rr,\forall q\in V, \forall p\in T_q^*M.
$$
We also introduce the structural constants $(c_{ij}^k)_{ i,j,k\in \ll 0,2n\rr}$ on $V$, defined by the relations
$$
\left[X_i,X_j\right](q)=\sum_{k=0}^{2n}c_{ij}^k(q) X_k(q),\qquad \forall i,j\in \ll 0,2n\rr, \forall q\in V.
$$
In terms of the coordinates $(h_i)_{i\in \ll 0,2n\rr}$, along a normal extremal, Equation~\eqref{E:ham_ode} yields (see \cite[Chapter 4]{ABB_2018})
$$
\frac{\diff h_i}{\diff t}
=
\{H,h_i\}
=
\sum_{j=0}^{2n}\sum_{k=0}^{2n} c_{j i}^k  h_j h_k, \quad \forall i\in \ll 0,2n\rr.
$$
We set $J:V\to \mathcal{M}_{2n}(\R)$ to be the matrix such that 
$
J_{ij}=c_{ji}^0$, for all $ i,j\in \ll1, 2n\rr$,
and   $Q:V \longrightarrow \left( \R^{2n}\rightarrow  \R^{2n}\right)$ to be the map such that for all $i\in \ll1, 2n\rr$,
$$
Q_i (h_1,\dots h_{2n})=\sum_{j=1}^{2n}\sum_{k=1}^{2n} c_{j i}^k  h_j h_k.
$$
By denoting $h=(h_1,\dots, h_{2n})$ we then have 
$
\dfrac{\diff h}{\diff t}=h_0 J h+Q(h).
$

As stated in Section~\ref{S:intro_short_geo}, we want an approximation of the geodesics for small time when $h_0(0)\to +\infty$, thus we introduce $w=\frac{h_0(0)}{h_0}$ and $\eta={h_0(0)}^{-1}$.
Then 
$
\dfrac{\diff w}{\diff t}=-\eta w^2\dfrac{\diff h_0}{\diff t}.
$

 We  separate the terms containing $h_0$ in the derivative of $w$ to obtain an equation similar to the one of $h$.  We set $L:V\to \mathcal{M}_{1\times 2n}(\R)$ to be the line matrix such that 
$L_{i}=c_{i0}^0$, for all $ i\in \ll1, 2n\rr$,
and   $Q_0:V \rightarrow \left( \R^{2n}\rightarrow  \R\right)$ to be the map such that 
$$
Q_0 (h_1,\dots h_{2n})=\sum_{j=1}^{2n}\sum_{k=1}^{2n} c_{j 0}^k  h_j h_k,
$$
so that 
$
\dfrac{\diff w}{\diff t}=-w Lh-\eta w^2 Q_0(h)
$.

Finally, rescaling time with $\tau=t/\eta$, we obtain
\begin{equation}\label{E:Ham_ode_rescaled}
\left\{
\begin{aligned}
&\frac{\diff q}{\diff \tau }=   \eta \sum_{i=1}^{2n} h_i X_i(q),
\\
&\frac{\diff h}{\diff \tau}= \frac{1}{w} Jh  +\eta \,  Q(h),
\\
&\frac{\diff w}{\diff \tau}=-\eta w Lh-\eta^2 w^2 Q_0(h).
\end{aligned}
\right.
\end{equation}

Hence to the solution of \eqref{E:ham_ode} with initial condition $(q_0,(h(0),\eta^{-1}))$ corresponds  the solution of the parameter depending differential equation \eqref{E:Ham_ode_rescaled} of initial condition $(q_0,h(0),w(0))$ and parameter $\eta$. Since $w(0)=1$, the flow of this ODE is well defined (at least for $\tau$ small enough), and smooth with respect to $\eta\in (-\varepsilon,\varepsilon)$, for some $\varepsilon>0$.

This warrants a power series study of its solutions as $\eta \to 0$.

\subsection{Approximation of the Hamiltonian flow}

Let $q_0\in M$. In the rest of the paper, except when explicitly stated otherwise, 
 we assume that the structure at $q_0$ has been put in the Agrachev--Gauthier normal form introduced in \cite{Gauthier_2001_SR_metrics_and_isoperimetric_problems}. That is, we have an open neighborhood $V\subset M$ of $q_0$, linearly adapted coordinates at $q_0$ $(x_1,\dots x_{2n},z):V \rightarrow \R^{2n+1}$ and a frame of $(\Delta,g)$, $(X_1,\dots, X_{2n})$, satisfying many useful symmetries.
(for instance, see Theorems~\ref{T:Ag_Gau_Frame_simpl}-\ref{T:Ag_Gau_Frame} in \ref{A:Gauthier}).
The family is locally completed as a basis of $TM$ with $X_0=\frac{\partial}{\partial z}$.

Let us introduce a few notations. Let $\bar{J}=J(q_0)$. As a consequence of the choice of frame, 
(in particular, see Equations~\eqref{E:AG_nf_hor3}  and \eqref{E:AG_nf_vert1} in \ref{A:Gauthier}), 
$\bar{J}$ is already in reduced form $\mathrm{diag}(\bar{J}_1,\dots, \bar{J}_n)$, that is, block diagonal with $2\times 2$ blocks
$$ 
\bar{J}_i=
\begin{pmatrix}
0 & b_i
\\
-b_i & 0 
\end{pmatrix}, \qquad\forall i\in \ll 1,n\rr,
$$
where $(b_i)_{i\in \ll 1,n\rr,}$ are the nilpotent invariants of the contact structure at $q_0$. 
Then let $\hhat:\R\times\R^{2n}\rightarrow \R^{2n}$, $\xhat:\R\times\R^{2n}\rightarrow \R^{2n}$ and $\zhat:\R\times\R^{2n}\rightarrow \R^{2n}$ be defined by 
$$
\begin{aligned}
\hhat(t,h)=  \e^{t\bar{J}} h, 
\qquad
\xhat(t,h)= \bar{J}^{-1} (\e^{t\bar{J}}-I_{2n}) h,
\\
 \zhat(t,h)=
\sum_{i=1}^{n}
\left(
h_{2i-1}^2+h_{2i}^2
\right) \frac{b_it-\sin(b_i t)}{2b_i},
\end{aligned}
$$
for all $t\in \R$ and all $h\in \R^{2n}$.

We also set $J^{(1)}:\R^{2n}\rightarrow \mathcal{M}_{2n}(\R)$ such that
$$ 
J^{(1)}_{i,j}(y)=\sum_{k=1}^{2n} \left( \frac{\partial^2 (X_i)_{2n+1}}{\partial x_j\partial x_k}-\frac{\partial^2 (X_j)_{2n+1}}{\partial x_i\partial x_k}\right) y_k,
\qquad 
\forall i,j\in \ll1,2n\rr,
$$
where for any vector field $Y$, we denote by $(Y)_i$, $1\leq i \leq 2n+1$, the $i$-th coordinate of $Y$, written in the basis $(\partial_{x_1},\dots,\partial_{x_{2n}},\partial_z)$.

Finally, let us denote $B_R=\{h\in \R^{2n}\mid \sum_{i=1}^{2n}h_i^2\leq R\}$.
\begin{proposition}\label{P:expansion}
For all $T,R>0$,
normal extremals with initial covector $(h(0),\\ \eta^{-1})$ have the following order 2 expansion at time $\eta \tau$, as  $\eta\to 0^+$, uniformly with respect to $\tau\in [0,T]$ and $h(0)\in B_R$.
In normal form coordinates, we denote
$$
\e^{\eta \tau\vec{H}}
\left(
\left(0,0\right),
\left(h(0),\eta^{-1}\right)\right)=\left(\left(x(\tau),z(\tau)\right),\left(h(\tau),\eta w(\tau)^{-1}\right) \right).
$$
Then
$$
\begin{aligned}
&x( \tau)=\eta \xhat(\tau,h(0))
+
\eta^2
\int_{0}^{\tau}
\int_{0}^{\sigma}
\e^{(\sigma-\rho)\bar{J}} J^{(1)}\left(\xhat(\rho,h(0))\right)\hhat(\rho,h(0))\diff \rho\, \diff \sigma
+O(\eta^3),
\\
&
z(  \tau)
=
\eta^2 
\zhat(\tau,h(0))
+O(\eta^3),
\end{aligned}
$$
and
$$
\begin{aligned}
&
h(  \tau)= \hhat(\tau,h(0))+\eta 
\int_{0}^{\tau}
\e^{(\tau-\sigma)\bar{J}} J^{(1)}\left(\xhat(\sigma,h(0))\right)\hhat(\sigma,h(0))\diff \sigma +O(\eta^2),
\\
&
w( \tau)= 1+O(\eta^2).
\end{aligned}
$$
\end{proposition}

\begin{proof}
This is a consequence of the integration of the time-rescaled system~\eqref{E:Ham_ode_rescaled}.
Since the system smoothly depends on $\eta$ near $0$, 
we prove this result by successive integration of the terms of the power series in $\eta$ of 
$x=\sum \eta^k x^{(k)}$,
$z=\sum \eta^k z^{(k)}$,
$h=\sum \eta^k h^{(k)}$, and
$w=\sum \eta^k w^{(k)}$.

Let $T,R>0$. All asymptotic expressions are to be understood uniform with respect to $\tau\in [0,T]$ and $h(0)\in B_R$.
Solutions of \eqref{E:Ham_ode_rescaled} are integral curves of a Hamiltonian vector field $\vec{H}$, hence $H$ is preserved along the trajectory, that is, for all $\tau\in [0,T]$,
$$
\sum_{i=1}^{2n}{h_i( \tau)}^{2}=\sum_{i=1}^{2n}{h_i(0)}^{2}.
$$
Furthermore, we have by \eqref{E:Ham_ode_rescaled}
$\dfrac{\diff x}{\diff \tau}=O(\eta)$, $\dfrac{\diff z}{\diff \tau}=O(\eta)$,
and since $x(0)=0$ and $z(0)=0$, we have $x(\tau)=O(\eta)$ and $z(\tau)=O(\eta)$.

As a consequence of the choice of frame,
(in particular conditions \eqref{E:AG_nf_hor3}--\eqref{E:AG_nf_vert1} in \ref{A:Gauthier}), 
$c_{ij}^k(q_0)\neq 0$ if and only if $k=0$ and there exists $l\in \ll1,n\rr$ such that $\{i,j\}=\{2l-1,2l\}$.

Hence  for all $j\in\ll 1, 2n\rr$, $c_{j 0}^0(q(\tau))=O(\eta)$ and $Lh=O(\eta)$.
Similarly, $Q_i(h)=O(\eta)$  for all $i\in \ll0,2n\rr$, and since  $w(0)=1$, we have that 
$
\dfrac{\diff w}{\diff \tau}=O(\eta^2)
$
and $w(\tau)=1+O(\eta^2)$. 

Since $J(q_0)=\bar{J}$, we have $J(q)=\bar{J}+O(\eta)$ and thus
$\dfrac{\diff h}{\diff \tau}=\bar{J}h  +O(\eta)$.
Hence $h$ is a small perturbation of the solution of $\frac{\diff h}{\diff \tau}=\Jbar h $ with initial condition $h(0)$, that is,  $h(\tau)=\hhat(\tau,h(0))+O(\eta)$.

Since $X_i(q_0)=\frac{\partial}{\partial x_i}$ for all $i\in \ll1,2n\rr$ (as a consequence of 
\eqref{E:AG_nf_hor1}),
$$
\frac{\diff x^{(1)}}{\diff \tau}= h^{(0)}(\tau)=\hhat(\tau,h(0)), \quad \frac{\diff z^{(1)}}{\diff \tau}= 0, 
$$ 
and since $x(0)=0$, $z(0)=0$,
$
x(\tau)=\eta \xhat(\tau,h(0))+O(\eta^2)
$
and $z(\tau)=O(\eta^2)$.

The definition of $J^{(1)}$   implies
$
J^{(1)}\left(x^{(1)}\right)=\eval{\frac{\partial J(q)}{\partial \eta}}{\eta=0}.
$
Then, since $Q(h)=O(\eta)$, $h^{(1)}$ is  solution of
$
\dfrac{\diff h^{(1)}}{\diff \tau}= \bar{J} h^{(1)}+J^{(1)}\left(x^{(1)}\right)
$
with initial condition $h^{(1)}(0)=0$. Hence 
$$
h^{(1)}(\tau)=
\int_{0}^{\tau}
\e^{(\tau-\sigma)\bar{J}} J^{(1)}\left(\xhat(\sigma,h(0))\right)\hhat(\sigma,h(0))\diff \sigma.
$$
Since $\frac{\partial {(X_i)}_j}{\partial x_k}=0$ for all $i,j,k\in\ll1, 2n\rr$ (as stated in 
\eqref{E:AG_nf_hor2}), 
$$
X_{2i-1}(q(\tau))=\partial_{x_{2i-1}}+\eta \,\xhat_{2i}(\tau,h(0))\frac{b_i}{2} \partial_{ z}+O(\eta^2),
$$
$$
X_{2i}(q(\tau))=\partial_{ x_{2i}}-\eta \,\xhat_{2i-1}(\tau,h(0))\frac{b_i}{2}  \partial_{ z}+O(\eta^2).
$$
Thus
$
\dfrac{\diff x^{(2)}}{\diff \tau}=
h^{(1)}$, 
$
\dfrac{\diff z^{(2)}}{\diff \tau}=
\sum_{i=1}^n \frac{b_i}{2}
\left( 
\hhat_{2i-1}\xhat_{2i}
-
\hhat_{2i}\xhat_{2i-1}
\right)
$.
Hence the statement by integration.
\end{proof}

\section{Conjugate time}\label{S:conjugate_time}

\subsection{Singularities of the sub-Riemannian exponential}\label{S:Singularities_SR_exponential}
\begin{definition}
Let $q_0\in M$.
We call \emph{sub-Riemannian exponential at $q_0$} the map 
$$
\begin{array}{rccc}
\sre_{q_0}:&\R^+\times T_{q_0}^*M&\longrightarrow& M
\\
& (t,p_0)&\longmapsto& \sre_{q_0}(t,p_0)=\pi\circ \e^{t\vec{H}}(p_0,q_0)
\end{array}
$$
where $\pi:T^*M\rightarrow M$ is the canonical fiber projection.
\end{definition}

Recall that the flow of the Hamiltonian vector field $\vec{H}$ satisfies the equality
$$
\e^{t\vec{H}}(p_0,q_0)=\e^{\vec{H}}(t p_0,q_0) ,\qquad \forall q_0\in M, p_0\in T_{q_0}^*M, t\in \R.
$$
We use this property to our advantage to compute  the sub-Riemannian caustic. Indeed, the caustic at $q_0$ is defined as the set of critical values of $\sre_{q_0}(1,\cdot)$. But for any time $t>0$, the caustic is also the set of critical values of $\sre_{q_0}(t,\cdot)$. Hence instead of classifying the covectors  $p_0$ such that $\sre_{q_0}(1,\cdot)$ is critical at  $p_0$, we compute for a given $p_0$ the conjugate time $t_c(p_0)$ such that $\sre_{q_0}(t_c(p_0),\cdot)$ is critical at $p_0$.

\begin{definition}
Let $q_0\in M$, and $p_0\in T_{q_0}^*M$. A \emph{conjugate time for $p_0$} is a positive time $t>0$ such that the map $\sre_{q_0}(t,\cdot)$ is critical at $p_0$. The \emph{conjugate locus of $q_0$} is the subset of $M$
$$
\left\{
\sre_{q_0}(t,p_0)
\mid t \text{ is a conjugate time for } p_0\in T_{q_0}M
\right\}.
$$
The \emph{first conjugate time for $p_0$}, denoted $t_c(p_0)$, is the minimum of conjugate times for $p_0$. The \emph{first conjugate locus of $q_0$} is the subset of $M$
$$
\left\{
\sre_{q_0}(t,p_0)
\mid t \text{ is the first conjugate time for } p_0\in T_{q_0}M
\right\}.
$$
In the following, we restrict our study of the sub-Riemannian caustic to the first conjugate locus. 
\end{definition}

From now on, let us index the nilpotent invariants in descending order $b_1\geq b_2\geq \cdots \geq b_n>0$. 
Let $\S_1\subset M$ be the set of points of $M$  such that two invariants coincide, $b_i=b_j$, with $i\neq j$. Assuming genericity of the sub-Riemannian manifold, $\mathfrak{S}_1$ is a stratified subset of $M$ of codimension $3$ (see \cite{charlot_2002_quasi_contact} for instance). 

\begin{remark}
This is  a consequence of Thom's transversality theorem applied  to the jets of the sub-Riemannian structure, seen as  a smooth map.

Furthermore, for a given $q_0\in M$, if the sub-Riemannian structure at $q_0$ is  in Agrachev--Gauthier normal form 
(see Appendix~\ref{A:Gauthier})
then the jets of order $k$ at $q_0$ of the sub-Riemannian structure are given by the jets at $0$ of the vector fields $X_1,\dots, X_{2n}$.
\end{remark}

As stated in the introduction,  the study of the sub-Riemannian caustic near its starting point requires  considering initial covectors in $\C_{q}(1/2)$ such that  $h_0$ is near infinity. Recall that geodesics with initial covectors in $\C_{q}(1/2)$ are parametrized by arclength, hence short conjugate time imply that the conjugate point is close  to the starting point of the caustic. Then one can check that a short conjugate time corresponds only to covectors with large $h_0$. 
From the point of view of the exponential at time $1$, this means that singular points close to the origin of the caustic must belong to a sufficiently narrow cone containing $\C_{q_0}(0)$ (again, because $\sre_{q_0}(t,p)=\sre^1_{q_0}(tp)$).

This observation can be stated in the following way 
(a proof can be found in Appendix~\ref{A:Gauthier}, see Proposition~\ref{L:no_singular_near_0_bis}, as an application of the Agrachev--Gauthier normal form).

\begin{proposition}\label{L:no_singular_near_0}
Let $(M,\Delta,g)$ be a contact sub-Riemannian manifold and $q_0\in M$. For all $\bar{h}_0>0$, there exists $\varepsilon >0$ such that all $p\in \C_{q}(1/2)$ with $t_c(p)<\varepsilon$ have $|h_0(p)|>\bar{h}_0$.
\end{proposition}

In coordinates, conjugate points satisfy the following equality
\begin{equation}\label{E:exp_det}
\eval{
\det
\left(
\frac{\partial \sre_{q_0}}{\partial{h_1}},
\dotsc,
\frac{\partial \sre_{q_0}}{\partial{h_{2n}}},
\frac{\partial \sre_{q_0}}{\partial{h_{0}}}
\right)
}
{(t,p_0)}=0.
\end{equation}
To use this equation in relation with the results of Proposition~\ref{P:expansion}, we introduce 
$$
F(\tau,h,\eta)=\sre_{q_0} (\eta \tau ; (h,\eta^{-1})), \quad \forall \tau>0, h\in \R^{2n},\eta>0.
$$
Then  
$$
\frac{\partial \sre_{q_0}}{\partial{h_0}}(\eta \tau ; (h,\eta^{-1}))=-\eta
\left(
\eta \frac{\partial F}{\partial \eta}(\tau,h,\eta)-\tau \frac{\partial F}{\partial \tau}(\tau,h,\eta)
\right)
$$
and \eqref{E:exp_det} equates to 
\begin{equation}\label{E:F_det}
\eval{
\det
\left(
\frac{\partial F}{\partial{h_1}},
\dotsc,
\frac{\partial F}{\partial{h_{2n}}},
\eta \frac{\partial F}{\partial \eta}-\tau \frac{\partial F}{\partial \tau}
\right)
}
{(\tau,h,\eta)}=0.
\end{equation}

We have shown in Proposition~\ref{P:expansion}, as $\eta\to 0$, that the map $F$ is a perturbation of the map $(\tau,h,\eta)\mapsto (\xhat,\zhat)$, the nilpotent exponential map. Hence the conjugate time is expected to be a perturbation of the conjugate time for $(\xhat,\zhat)$.
To get an approximation of the conjugate time for a covector $(h,\eta^{-1})$ as $\eta\to 0$, we use expansions from Proposition~\ref{P:expansion} to derive equations on a power series expansion of the conjugate time.

\subsection{Nilpotent order and doubling of the conjugate time}\label{SS:First_approximation}

Let us define
\begin{equation}\label{E:def_Phi}
\Phi(\tau,h,\eta)=\eval{
\det
\left(
\frac{\partial F}{\partial{h_1}},
\dotsc,
\frac{\partial F}{\partial{h_{2n}}},
\eta \frac{\partial F}{\partial \eta}-\tau \frac{\partial F}{\partial \tau}
\right)
}
{(\tau,h,\eta)}
\end{equation}
and its power series expansion
$\Phi(\tau,h,\eta)=\sum_{k\geq 0} \eta^k\Phi^{(k)}(\tau,h)$.

As a first application of Proposition~\ref{P:expansion}, notice that $F_i=O(\eta)$ for all $i\in \ll1, 2n\rr$, while $F_{2n+1}=O(\eta^2)$.
Hence, one gets $\Phi^{(k)}=0$ for all $k\in \ll0, 2n+1\rr$, and 
$\Phi^{(2n+2)}$ is the first non-trivial term in the power series.

To study $\Phi^{(2n+2)}$, let us introduce the set
$
Z= \left\{2k \pi/b_i\mid {i\in \ll1, n\rr, k\in \N}\right\}
$
and the map $\psi : (\R^+\setminus Z)\times \R^{n}\rightarrow \R$ defined by
$$
\psi(\tau,r)
=
\sum_{i=1}^{n}
\frac{r_i^2}{2}
\left(
3 \tau
-  b_i \tau ^2\, \frac{ \cos (b_i \tau /2)}{\sin (b_i \tau /2) }-\frac{\sin (b_i \tau )}{ b_i}
\right), 
\quad 
\forall (\tau,r)\in (\R^+\setminus Z)\times \R^{n}.
$$

We first need the following result on the zeros of $\psi$ (see, for instance, Appendix~\ref{A:proof_lemma_psi}).

\begin{lemma}\label{L:lim_r1_to_0}
Assume $b_1>b_2\geq \dots \geq b_n$.  For all $r\in{(\R^+)}^{n}$, let
 $\tau_1(r)$ be the first positive time in $\R^+\setminus Z$ such that $\psi(\tau_1,r)=0$. Then
 $
 \tau_1(r_1,\dots, r_n)>2\pi/b_1 
 $
 and there exists $f(r_2, \dots ,r_n)>0$ such that,  as $r_1\to 0^+$,
 \begin{equation}\label{E:equiv_tau_1}
 \tau_1(r_1,\dots, r_n)=2\pi/b_1+f(r_2, \dots ,r_n) r_1^2 +o(r_1^2).
\end{equation}

\end{lemma}

The zeros of $\Phi^{(2n+2)}$ can be deduced from the zeros of $\psi$, as shown in the following proposition.
\begin{proposition}\label{P:first_order_time}
Assume $b_1>b_2> \dots  > b_n$. Let $h\in \R^{2n}\setminus \{0\}$ and $r\in \R^n$ be such that $r_i=\sqrt{h_{2i-1}^2+h_{2i}^2}$ for all $i\in \ll1 , n\rr$.
Then $\Phi^{(2n+2)}(\tau,h)=0$ if and only if $\tau\in Z$ or $\psi(\tau,r)=0$.
In particular
$$
\Phi^{(2n+2)}(\tau,h)\neq 0 \quad \forall \tau\in (0,2\pi/b_1), \forall h\in \R^{2n}\setminus \{0\}.
$$
\end{proposition}
\begin{proof}
By factorizing powers of $\eta$ in $\Phi$, we obtain that $\Phi^{(2n+2)}$ is given by the determinant of the matrix
$$
M=
\left(
\begin{array}{c|c}
D_h \xhat(\tau)
&
\xhat(\tau)-\tau \hhat(\tau) 
\\
\hline
D_h \zhat(\tau) & \zhat(\tau)-\tau \frac{\diff}{\diff \tau} \zhat(\tau)
\end{array}
\right).
$$
The Jacobian matrix 
$
D_{h} \xhat= \bar{J}^{-1} (\e^{\tau \bar{J}}-I_{2n})
$
is invertible for $\tau\in \R^+\setminus Z$ and of rank $2n-2$ for $\tau \in Z$.
Hence, the matrix $M$ is not invertible for $\tau\in \R^+\setminus Z$ if and only of we have the linear dependance of the family 
$$
\left\{
\frac{\partial}{\partial h_1}
\left(
\begin{array}{c}
\xhat(\tau)
\\
\hline
 \zhat(\tau) 
\end{array}
\right),
\dotsc,
\frac{\partial}{\partial h_{2n}}
\left(
\begin{array}{c}
\xhat(\tau)
\\
\hline
 \zhat(\tau) 
\end{array}
\right),
\left(
\begin{array}{c}
\xhat(\tau)-\tau \hhat(\tau) 
\\
\hline
 \zhat(\tau)-\tau \frac{\diff}{\diff \tau} \zhat(\tau)
\end{array}
\right)
\right\}.
$$

This implies the existence of  $\mu \in \R^{2n}$ such that both
$D_h \xhat(\tau) \mu =\xhat(\tau)-\tau \hhat(\tau) $ and $D_h \zhat(\tau) \mu = \zhat(\tau)-\tau \frac{\diff}{\diff \tau} \zhat(\tau)$. That is
$$
D_h \zhat(\tau)\left(D_h \xhat(\tau)\right)^{-1}  \left(\xhat(\tau)-\tau \hhat(\tau) \right)= \zhat(\tau)-\tau \frac{\diff}{\diff \tau} \zhat(\tau).
$$
We explicitly have
$
 \zhat(\tau)-\tau \frac{\diff}{\diff \tau} \zhat(\tau)
=
\sum_{i=1}^{n}
\frac{r_i^2}{2}
\left(
\tau \cos b_i \tau - \frac{\sin b_i\tau }{b_i}
\right)
$
and
$$
D_h \zhat(\tau)\left(D_h \xhat(\tau)\right)^{-1}  \left(\xhat(\tau)-\tau \hhat(\tau) \right)
=
\sum_{i=1}^{n}r_i^2\,
(\sin b_i \tau-b_i \tau)
\tfrac{
b_i \tau \cos ( b_i \tau /2  )-2\sin( b_i \tau /2)
}{2 b_i \sin( b_i \tau/2)} 
.
$$
Hence
$
D_h \zhat \left(D_h \xhat \right)^{-1}  \left(\xhat -\tau \hhat  \right)
-
\left(\zhat -\tau \frac{\diff \zhat}{\diff \tau} \right)
=
\psi(\tau,r)
$,
and times $\tau\in \R^+$ such that $\Phi_k^{(2n+2)}(\tau,h)=0 $ are either multiples of $2\pi b_i$, $i\in \ll1, n\rr$, or zeros of $\psi$.
Under the assumption that  $h\in \R^{2n}\setminus \{0\}$ and $\tau\in (0,2b_i\pi)$, we have $\psi(\tau,r)> 0$, hence the result.
\end{proof}

We can draw some conclusions regarding our analysis of the conjugate locus via a perturbative approach. From Proposition~\ref{P:first_order_time}, we have that $2\pi/b_1$ is the first zero of $\Phi^{(2n+2)}(\cdot ,h)$   for all $h\in \R^{2n}\setminus \{0\}$.
From Lemma~\ref{L:lim_r1_to_0} we also know that $2\pi/b_1$ is a simple zero if $r_1>0$ and a double zero otherwise (see Figure~\ref{F:Phi_zoom_r1_0}). Zeros of order larger than $1$ can be unstable under perturbation and this case requires a separate analysis, either by high order approximation or by blowup. We choose the latter for computational reasons.

\begin{figure}[h]
\begin{center}
\begin{minipage}{.45\textwidth}
\subfloat[$\Phi^{(2n+2)}$ as $r_1=\dfrac{r_2}{4}$.]{
  \centering
  \begingroup%
  \makeatletter%
  \ifx\svgwidth\undefined%
    \setlength{\unitlength}{.9\linewidth}%
    \ifx\svgscale\undefined%
      \relax%
    \else%
      \setlength{\unitlength}{\unitlength * \real{\svgscale}}%
    \fi%
  \else%
    \setlength{\unitlength}{\svgwidth}%
  \fi%
  \global\let\svgwidth\undefined%
  \global\let\svgscale\undefined%
  \makeatother%
  \begin{picture}(1,0.62305987)%
    \put(0,0){\includegraphics[width=\unitlength,page=1]{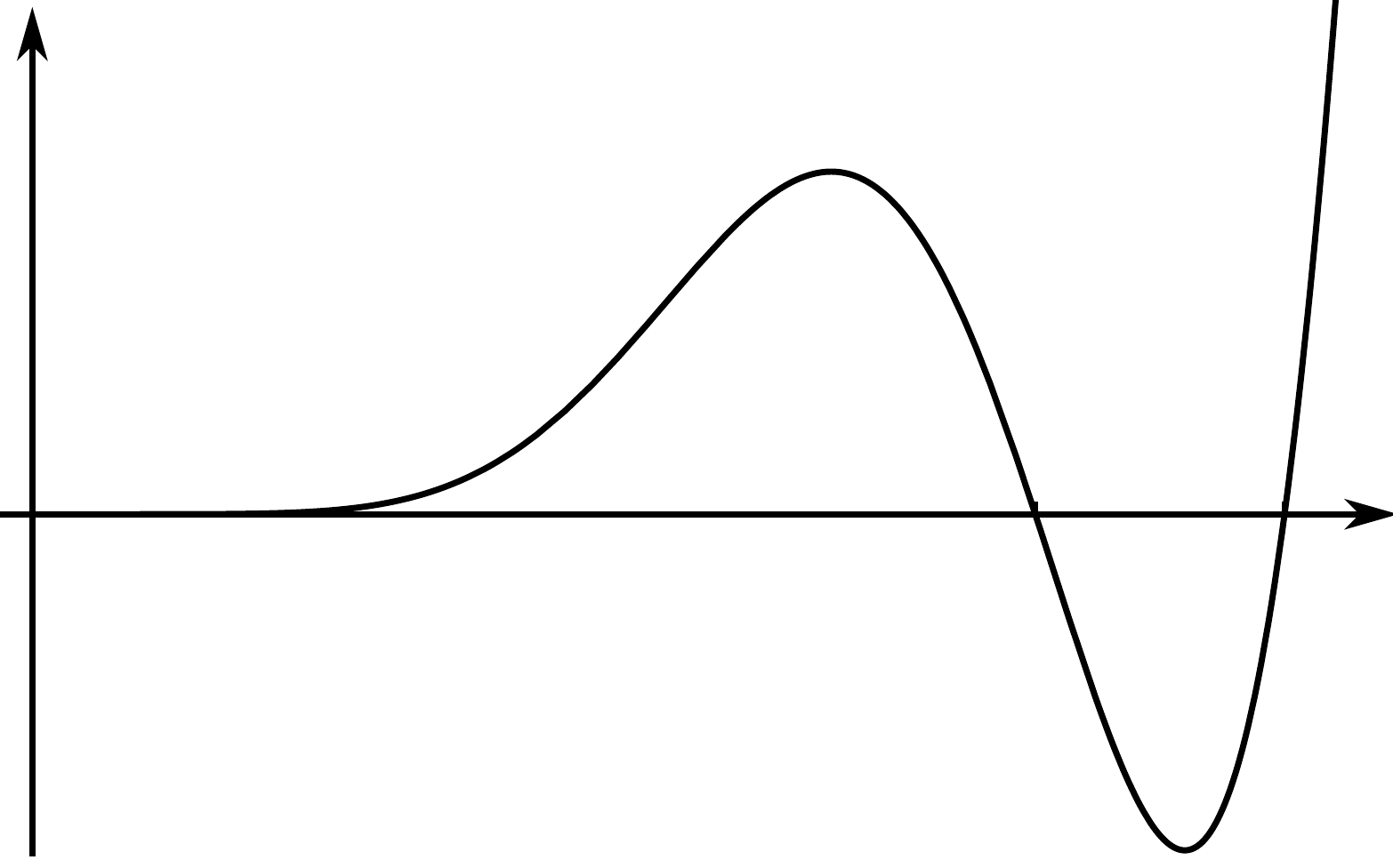}}%
    \put(0.67,0.19){\color[rgb]{0,0,0}\makebox(0,0)[lb]{\smash{$\frac{2\pi}{b_1}$}}}%
    \put(0.925,0.21){\color[rgb]{0,0,0}\makebox(0,0)[lb]{\smash{$\tau_1$}}}%
    \put(0.035,0.6){\color[rgb]{0,0,0}\makebox(0,0)[lb]{\smash{$\Phi^{(2n+2)}$}}}%
    \put(.98,0.27){\color[rgb]{0,0,0}\makebox(0,0)[lb]{\smash{$\tau$}}}%
  \end{picture}%
\endgroup%
}
  
\end{minipage}
\hspace{.08\textwidth}
\begin{minipage}{.45\textwidth}
\begin{center}
\subfloat[$\Phi^{(2n+2)}$ as $r_1=0$.]{
  \centering
  \begingroup%
  \makeatletter%
  \ifx\svgwidth\undefined%
    \setlength{\unitlength}{.9\linewidth}%
    \ifx\svgscale\undefined%
      \relax%
    \else%
      \setlength{\unitlength}{\unitlength * \real{\svgscale}}%
    \fi%
  \else%
    \setlength{\unitlength}{\svgwidth}%
  \fi%
  \global\let\svgwidth\undefined%
  \global\let\svgscale\undefined%
  \makeatother%
  \begin{picture}(1,0.62305987)%
    \put(0,0.0115){\includegraphics[width=\unitlength,page=1]{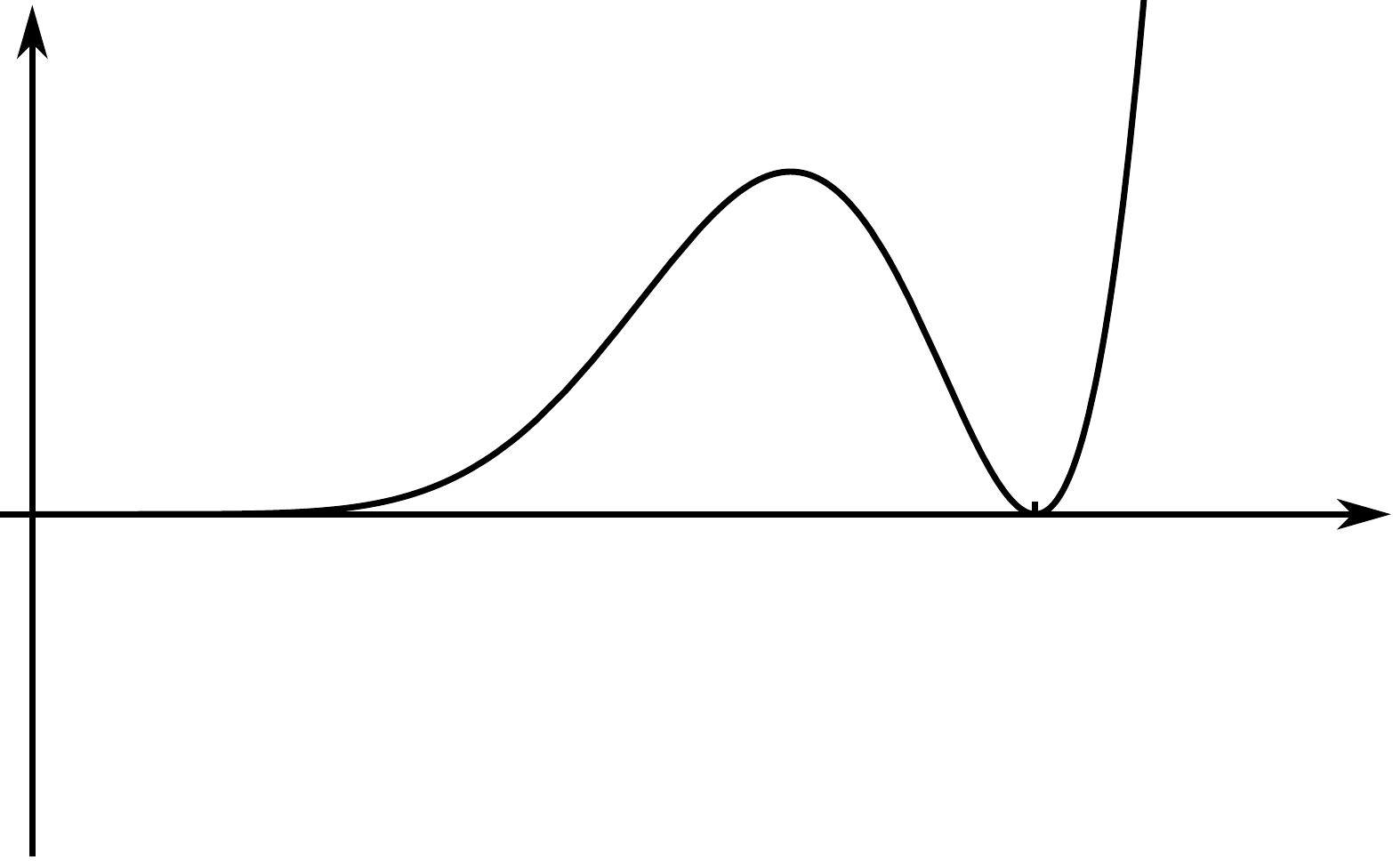}}%
    \put(0.715,0.2){\color[rgb]{0,0,0}\makebox(0,0)[lb]{\smash{$\frac{2\pi}{b_1}$ }}}%
    \put(0.035,0.6115){\color[rgb]{0,0,0}\makebox(0,0)[lb]{\smash{$\Phi^{(2n+2)}$}}}%
    \put(.98,0.2815){\color[rgb]{0,0,0}\makebox(0,0)[lb]{\smash{$\tau$}}}%
  \end{picture}%
\endgroup%
}
\end{center}
\end{minipage}
\end{center}

\caption[Representation of $\Phi^{(2n+2)}$ as a function of $\tau$ in the case $n=2$, as $r_1\neq 0$ and $r_1=0$.]{Representation of $\Phi^{(2n+2)}$ as a function of $\tau$ in the case $n=2$, as $r_1\neq 0$ and $r_1=0$ (with $b_1=2$, $b_2=1/4$ and $r_2=1$).}
\label{F:Phi_zoom_r1_0}
\end{figure}
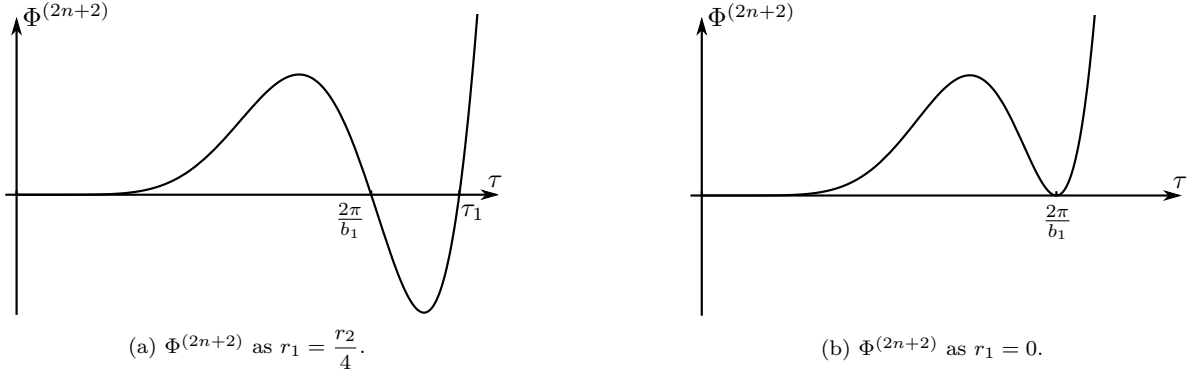

From Equation~\eqref{E:equiv_tau_1} in Lemma~\ref{L:lim_r1_to_0},  the blowup $r_1\leftarrow {\eta}^\alpha r_1$ corresponds to 
$$
\tau_1(\eta^\alpha r_1,r_2,\dots ,r_n)=2\pi/b_1+\eta^{2\alpha} f(r_2,\dots ,r_n)r_1^2+o(\eta^{2\alpha}).
$$
Since we have an approximation of the exponential that is a perturbation of order $\eta$ of the nilpotent exponential, we expect the conjugate time to be a perturbation of order $\eta$ of the nilpotent conjugate time. Hence it is natural to chose $\alpha=1/2$ in hope to capture  a perturbation of comparable order in $\eta$.

We separate the cases in the following way. 
\begin{itemize}
\item We can compute the conjugate time assuming $r_1>\varepsilon$ for some arbitrary $\varepsilon$ (in Section~\ref{SS:asymptotics_away_S});
\item we use the blowup $r_1\leftarrow \sqrt{\eta} r_1$  to get the conjugate time near $r_1=0$ (in Section~\ref{SS:asymptotics_near_S_1}).
\end{itemize}

\subsection{Statement of the conjugate time asymptotics}

The focus of this paper is now devoted to the proof of the following asymptotic expansion theorem for the conjugate time on $M\setminus \S_1$, that is, at points such that $b_1>b_2> \cdots> b_n$. 
Let $S_1$ be the subspace of  $T_{q_0}^*M$ defined by
$$
S_1
=
\left\{
(h_1,\dots ,h_{2n},h_0)\in T_{q_0}^*M \setminus \C_{q_0}(0)\mid h_1=h_2=0,H\neq 0
\right\},
$$
and for all $\varepsilon>0$, let us denote by $S_1^\varepsilon$ the subset of $T_{q_0}^*M$ containing $S_1$:
$$
S_1^\varepsilon
=
\left\{
(h_1,\dots ,h_{2n},h_0)\in T_{q_0}^*M\setminus \C_{q_0}(0) \mid h_1^2+h_2^2<\varepsilon H(h_1,\dots ,h_{2n},h_0)
\right\}.
$$
\begin{theorem}\label{T:big_expansion_theorem}
Let $q_0\in M\setminus \S_1$.
There exist real valued invariants $(\kappa^{ij}_{k})_{\begin{smallmatrix}
i,k\in \ll 1,2\rr,
\\
j\in \ll3, 2n\rr
\end{smallmatrix}}$, $\alpha$, $\beta$, 
 such that we have the following asymptotic behavior for initial covectors $p_0 \in T_{q_0}^*M$ with $h_0\to +\infty$.
 
\item[(Away from $S_1$.)] For all $R>0,\varepsilon\in (0,1)$, uniformly with respect to $p_0=(h_1,\dots ,h_{2n},\\h_0)$ in $ \C_{q_0}((0,R))\setminus S_1^\varepsilon$, we have as $h_0\to +\infty$
$$
t_c\left( h_1,\dots ,h_{2n},h_0\right)
=
\frac{2\pi}{b_1 h_0}
+
\frac{1}{h_0^2}t_c^{(2)}(h_1,\dots ,h_{2n})
+
O\left(\frac{1}{h_0^3}\right)
$$
where $t_c^{(2)}$ satisfies
\begin{equation}\label{E:linear_tc2}
(h_1^2+h_2^2)t_c^{(2)}(h)
=
-2(\alpha h_1+\beta h_2)\left(h_1^2+h_2^2\right)
+
	(\gamma_{12}+\gamma_{21})
h_1 h_2
-
	\gamma_{22}	h_1^2 
-
	\gamma_{11}	h_2^2  ,
\end{equation}
denoting
$$
\gamma_{ij}=\sum_{k=3}^{2n}\kappa^{jk}_{i} h_k,		
 \qquad 
\forall  i,j\in \ll1,2n\rr.
$$

\item[(Near $S_1$.)] The asymptotic expansion
$$
t_c\left(\frac{h_1}{\sqrt{h_0}},\frac{h_2}{\sqrt{h_0}},h_3,\dots ,h_{2n},h_0\right)
=
\frac{2\pi}{b_1 h_0}
+
O\left(\frac{1}{h_0^2}\right)
$$
 holds if and only if the quadratic polynomial equation in $X$
\begin{multline*}
X^2 K- X \left[\frac{2\pi}{b_1}(h_1^2+h_2^2)-K\left(\gamma_{1 1}+\gamma_{2 2}\right)\right]
\\
+
\frac{2\pi}{b_1}  \left[	
	(\gamma_{1 2}+\gamma_{2 1}) h_1 h_2	
-
	\gamma_{2 2}	h_1^2 
-
	\gamma_{1 1}	h_2^2  
\right]
+
K \left(\gamma_{11}\gamma_{22}-\gamma_{12}\gamma_{21}\right)=0
\end{multline*}
admits a real solution, where $K= \sum\limits_{i=2}^{n}(h_{2i-1}^2+h_{2i}^2)\left(1 -\frac{b_i}{b_1} \pi \cot \frac{b_i \pi}{b_1}\right)>0$.

If that is the case, denote by $\tilde{t}_c^{(2)}( h_1,\dots ,h_{2n})$ the smallest of its two (possibly double) solutions.
Then, for all $R>0,\varepsilon\in (0,1)$, uniformly with respect to $p_0=\left(\frac{h_1}{\sqrt{h_0}},\frac{h_2}{\sqrt{h_0}},h_3,\dots ,h_{2n},h_0\right) \in \C_{q_0}((0,R))\cap  S_1^\varepsilon$,
 we have 
$$
t_c \left(\frac{h_1}{\sqrt{h_0}},\frac{h_2}{\sqrt{h_0}},h_3,\dots ,h_{2n},h_0\right)
=
\frac{2\pi}{b_1h_0}
+
\frac{1}{h_0^2}\tilde{t}_c^{(2)}(h_1,\dots ,h_{2n}) 
+
O\left(\frac{1}{h_0^3}\right).
$$

\end{theorem}

\section{Perturbations of the conjugate time}\label{S:order_2_approx}

Thanks to the previous section, we have a sufficiently precise picture of the behavior of the conjugate time for the nilpotent approximation. We now introduce small perturbations of the exponential map in accordance with Proposition~\ref{P:expansion}. As stated previously, we treat separately the case of initial covectors away from $S_1$ and near $S_1$ since $S_1$ corresponds to the set of covectors such that $r_1=\sqrt{h_1^2+h_2^2}=0$. Recall also that we assumed $q_0\in M\setminus \S_1$.

However, rather than computing $t_c$, we compute $\tau_c=t_c/\eta$, the rescaled conjugate time, since we use asymptotics in rescaled time from Proposition~\ref{P:expansion}.

\subsection{Asymptotics for covectors in $T_{q_0}^*M\setminus S_1$}\label{SS:asymptotics_away_S}
In this section we assume that $(h_1,h_2)\neq(0,0)$.
Recall that
$F(\tau,h,\eta)=\sre (\eta \tau ; (h,\eta^{-1}))$, for all $ \tau>0$, $h\in \R^{2n}$, $\eta>0$.
The function $F$ admits a power series expansion 
$$
F(\tau,h,\eta)=\sum_{k\geq 0} \eta^k F^{(k)}(\tau,h),
$$
and for $\delta\! \tau\in \R$, $h\in \R^{2n}$, evaluating $F$ at the perturbed conjugate time $\frac{2\pi}{b_1}+\eta \delta\! \tau$ yields
\begin{equation}\label{E:order_1_F_delta_t}
F\left(\frac{2\pi}{b_1}+\eta \delta\! \tau,h,\eta\right)
=
\eta 
\eval{
	F^{(1)}
}{\tau=\frac{2\pi}{b_1}}
+
\eta^2
\eval{
\left[
 F^{(2)}
 +
 \delta\! \tau
 \frac{\partial F^{(1)}}{\partial \tau}
 \right]
}{\tau=\frac{2\pi}{b_1}}
 + O(\eta^3).
\end{equation}

In the previous section, we highlighted the role of the function $\Phi$ defined by \eqref{E:def_Phi}. Observe that $\tau_c$ must annihilate every term in the Taylor expansion of $\Phi(\tau_c(\cdot ,\eta ),\cdot,\eta)$. 
This first non-trivial term is obtained by straight forward algebraic computations (provided for instance in Appendix~\ref{A:computational_lemmas}, in particular Lemma~\ref{P:mini_det_1}).

\begin{proposition}\label{P:conjugate_time_simple_case}
Let $\tau_c(h,\eta)=\sum_{k=0}^{+\infty} \eta^k\tau_c^{(k)}(h)$ be the formal power series expansion of $\tau_c$, for all $(h,\eta^{-1})\in T_{q_0}^*M$.
Then $\tau_c^{(0)}=2\pi/b_1$ and $\tau_c^{(1)}$ must satisfy
\begin{equation}\label{E:def_tauc1_r1neq0}
(h_1^2+h_2^2)\tau_c^{(1)}(h)
=
-h_1^2  \frac{\partial \left(
			 F^{(2)}
			 \right)_2}{\partial h_2}			
-
h_2^2  \frac{\partial \left(
			 F^{(2)}
			 \right)_1}{\partial h_1}	
+ 
h_1 h_2
\left(  \frac{\partial \left(
			 F^{(2)}
			 \right)_1}{\partial h_2}
			 +
			 \frac{\partial \left(
			 F^{(2)}
			 \right)_2}{\partial h_1}
\right).
\end{equation}
\end{proposition}

\begin{proof}
As discussed in the previous section, $\tau_c^{(0)}=2\pi/b_1$ is a consequence of Proposition~\ref{P:first_order_time}. 
The first non trivial term of the expansion of the determinant $\Phi\left(2\pi /b_1+\eta \delta\! \tau,h,\eta\right)$, that is, the term of order $2n+3$,
is obtained by algebraic computations.
As a consequence of Proposition~\ref{P:expansion}, 
notice that
$\left(F^{(2)}\right)_{2n+1}=\zhat$, $\frac{\partial F^{(1)}}{\partial \tau}=\hhat$, and that
$\partial_{h_1}\zhat= 2\pi h_1/b_1$, $\partial_{h_2}\zhat= 2\pi h_2/b_1$. Hence we get the stated result by  solving for $\delta\! \tau$
$$
\Phi^{(2n+3)}\left(2\pi /b_1+\eta \delta\! \tau,h,\eta\right)
\propto
\left|
\begin{matrix}
			\frac{\partial}{\partial h_1}
			\left(
			 F^{(2)}
			 \right)_1
			+			 
			\delta\! \tau		 
		 &
			 \frac{\partial}{\partial h_2}
			\left(
			 F^{(2)}
			 \right)_1
		 &
		 h_1
	\\
			\frac{\partial}{\partial h_1}
			\left(
			 F^{(2)}
			 \right)_2
		 &
			 \frac{\partial}{\partial h_2}
			\left(
			 F^{(2)}
			 \right)_2
			 +
			 \delta\! \tau
		 &
		 h_2
	\\
			h_1
		 &
			 h_2
		 &
		 0
\end{matrix}
\right|_{\tau=2\pi/b_1}=0.
$$
(Where we denote, for $f,g:\R^n\to \R$, $f\propto g$ if there exists $h:\R^n\to \R\setminus\{0\}$ such that $f=g h$.)
\end{proof}
\begin{remark}
Relation \eqref{E:def_tauc1_r1neq0} is degenerate at $h_1=h_2=0$. This is another illustration of the behavior we highlighted in the previous section, that is,  $\tau_c^{(1)}$ can be a zero of order 2 at $r_1=0$.
\end{remark}

As a consequence of Proposition \ref{P:expansion}, it appears that for all $k\in \ll1, 2n\rr$ and all $\tau>0$, each function $h\mapsto x_k^{(2)}(\tau)$ can be seen as a quadratic form on $(h_1,\dots, h_{2n})$. Hence we introduce the   invariants $\left(\kappa^{ij}_{k}\right)_{ i,j,k\in \ll1, 2n\rr}$ such that 
$$
	F^{(2)}_k\left(\frac{2\pi}{b_1},h\right)\\
	=
	\sum_{1\leq i\leq j\leq 2n} \kappa^{ij}_{k} h_i h_j \qquad \forall  k\in \ll1, 2n\rr.
$$

These invariants satisfy some useful properties (of which a proof can be found in Appendix~\ref{A:Computation_invariants}, Lemmas~\ref{L:invariants_order_1} through \ref{L:invariants_kappa131}). We give the following summary.

\begin{proposition}\label{P:summary_kappa}

The   invariants $\left(\kappa^{ij}_{k}\right)_{ i,j,k\in \ll1, 2n\rr}$ depend linearly  on the family
$$
\left(
\frac{\partial^2 (X_i)_{2n+1}}{\partial x_j\partial x_k}(q_0)
\right)_{ i,j,k\in \ll1, 2n\rr}.
$$
There exist $\alpha,\beta\in \R$ such that we have the symmetries
$$
\kappa^{1,1}_1=3\alpha,
\quad 
\kappa^{2,2}_1=\alpha, 
\quad 
\kappa^{1,2}_2=2 \alpha,
\qquad
\kappa^{1,1}_2=\beta,
\quad
 \kappa^{2,2}_2=3\beta,
\quad 
\kappa^{1,2}_1=2 \beta
$$
%
and for all $i\in \ll2,n\rr$, $\left(\kappa^{kl}_{m}\right)_{\begin{smallmatrix}k,m\in\ll1,2\rr \\ l\in \ll2i-1,2i\rr \end{smallmatrix}}$ only depend on the family
$$
\left\{
\left(\frac{\partial^2 (X_k)_{2n+1}}{\partial x_l\partial x_m}(q_0)\right)
\mid (k,l,m)\in \ll2i-1,2i\rr\times \ll1,2\rr^2 \cup   \ll1,2\rr^2 \times\ll2i-1,2i\rr
\right\}.
$$
Furthermore, the corresponding linear map $\zeta_{i}:\R^{15}\rightarrow \R^8$ such that 
$$
\zeta_i\left(\left(\frac{\partial^2 (X_k)_{2n+1}}{\partial x_l\partial x_m}(q_0)\right)_{k,l,m\in\{1,2\}\cup\{2i-1,2i\}}\right)
=
\left(\kappa^{k l}_m\right)_{\begin{smallmatrix}k,m\in\{1,2\} \\ l\in \{2i-1,2i\} \end{smallmatrix}}
$$
 is of rank at least $7$ (and of rank $8$ on the complementary of a codimension $1$ subset $\S_3$ of $M$).
\end{proposition}

\begin{remark}\label{R:codim}
A consequence of the rank of $\zeta_i$ being $7$, for all $2\leq i \leq n$, is that a single condition of codimension $ k\geq 2$ on $\left(\kappa^{kl}_m\right)_{\begin{smallmatrix}k,m\in\ll1,2\rr \\ l\in \ll2i-1,2i\rr \end{smallmatrix}}$ is then a condition of codimension at least $k-1$ on the jets of order 2 of the sub-Riemannian structure at $q_0$.
\end{remark}

Using this notation, we can give a first approximation of the conjugate locus.
\begin{proposition}\label{P:CL_order2}
Let $q_0\in M\setminus \S_1$.
As $\eta\rightarrow 0^+$, uniformly with respect to $p_0=(h_1,\dots ,h_{2n},\eta^{-1})\in \C_{q}((0,R))\setminus S_1^\varepsilon$ for all $R>0,\varepsilon \in(0,1)$, we have (in normal form coordinates)
$$
(F(\tau_c(h,\eta)),h,\eta))_1
=
\frac{ \eta^2
 }{h_1^2+h_2^2}
 \left((\gamma_{11}-\gamma_{22})
	h_1^3
	+
	\gamma_{12}
	h_2^3
	+
	(\gamma_{21}+2 \gamma_{12})
	h_1^2h_2	
	+\delta_1\right)
+O(\eta^3)
$$
$$
(F(\tau_c(h,\eta)),h,\eta))_2
=
\frac{ \eta^2
 }{h_1^2+h_2^2}
 \left(
 	\gamma_{12}
	h_1^3
	-
	(\gamma_{11}-\gamma_{22})
	h_2^3
	+
	(\gamma_{12}+2 \gamma_{21})
	h_1 h_2^2
	+\delta_2
	\right)
+O(\eta^3)
$$
with
$$
\gamma_{ij}=\sum_{k=3}^{2n}\kappa^{jk}_{i} h_k,		
 \qquad 
\forall  i,j\in \ll1,2n\rr,
$$
$$
\delta_1=\alpha(h_1^2+h_2^2)^2+
\sum_{3\leq i<j\leq 2n}^{2n}
\kappa^{ij}_1 h_i h_j,
\qquad
\delta_2=\beta(h_1^2+h_2^2)^2+
\sum_{3\leq i<j\leq 2n}^{2n}
\kappa^{ij}_2 h_i h_j.
$$
\end{proposition}

If there exists a covector such that $\gamma_{11}-\gamma_{22}=\gamma_{12}=  \gamma_{21}=0$ then this first order approximation of the conjugate locus is not sufficient to prove stability and more orders of approximation are necessary. This occurs for instance when $h_3=\dots=h_{2n}= 0$, and
$$
\begin{aligned}
(F(\tau_c(h,\eta)),h,\eta))_1
=
\eta^2 \alpha (h_1^2+h_2^2)
+O(\eta^3),
\\
(F(\tau_c(h,\eta)),h,\eta))_2
=
 \eta^2\beta (h_1^2+h_2^2)
+O(\eta^3).
\end{aligned}
$$

\begin{proposition}\label{P:def_S_2}
Let $M$ be a generic contact sub-Riemannian manifold of dimension $2n+1\geq 5$. Let $\mathfrak{S}_2\subset M$ be the set of points at which the linear system in $(h_3,\dots h_{2n})$
$$
\left\{
\begin{array}{l}
\sum_{i=3}^{2n}
(\kappa^{1,i}_1 -\kappa^{2,i}_2 )h_{i}=0,
\\
\sum_{i=3}^{2n}
\kappa^{1,i}_2 h_{i}=0,
\\
\sum_{i=3}^{2n}
\kappa^{2,i}_1 h_{i}=0,
\end{array}
\right.
$$
admits non-trivial solutions. If $\dim M\geq 7$, then $M=\S_2$. However if $\dim M=5$, the set $\mathfrak{S}_2$ is codimension 1 stratified subset of $M$.
\end{proposition}
\begin{proof}
If we assume $(r_2,\dotsc,r_n)\neq 0$ then $\gamma_{11}-\gamma_{22}=\gamma_{12}=\gamma_{21}=0$ reduces to the existence of a non-zero vector of $\R^{2n-2}$ in the intersection
$$
{\Span\{(
\kappa^{1,3}_1 -\kappa^{2,3}_2 ,\dotsc,\kappa^{1,2n}_1 -\kappa^{2,2n}_2 )
\}}^{\perp}
 \cap\,
{\Span\{(
\kappa^{2,3}_1  ,\dotsc,\kappa^{2,2n}_1  )
\}}^{\perp}
 \cap\,
{\Span\{(
\kappa^{1,3}_2  ,\dotsc,\kappa^{1,2n}_2  )
\}}^{\perp}
.
$$
This space is never reduced to a single point for $n>2$, hence $M=\mathfrak{S}_2$. However for $n=2$, this requires the three vectors
\begin{equation}
(\kappa^{1,3}_1 -\kappa^{2,3}_2 ,\kappa^{1,4}_1 -\kappa^{2,4}_2 ),
\quad
(\kappa^{2,3}_1,\kappa^{2,4}_1  ),
\quad
(\kappa^{1,3}_2 ,\kappa^{1,4}_2  ),
\end{equation}
to be co-linear, which is a constraint of codimension $2$ on the family $\left(\kappa^{k l}_m\right)_{\begin{smallmatrix}k,m\in\{1,2\} \\ l\in \{3,4\} \end{smallmatrix}}$. By Remark~\ref{R:codim}, this is  a codimension 1 (at least) constraint on the jets of order $2$ of the sub-Riemannian structure at $q_0$, hence the result.
\end{proof}

\subsection{Asymptotics for covectors near $S_1$} \label{SS:asymptotics_near_S_1}

We repeat the previous construction for a special class of initial covector in the vicinity of  $S_1=
\{
(h_1,\dots ,h_{2n},h_0)\in T_{q_0}^*M \mid\\ h_1=h_2=0
\}$, in accordance with the discussion of Section~\ref{SS:First_approximation}.

Let $\h\in\R^{2n}$ be such that  $(\h_3,\dots,\h_{2n})\neq (0,\dotsc,0)$. We blowup the singularity at $h_1=h_{2}=0$ by computing an approximation of the conjugate locus for
\begin{equation}\label{E:special_covector}
h(0)=
(\sqrt{\eta}\h_1,\sqrt{\eta}\h_2, \h_3,\dotsc, \h_{2n}).
\end{equation}
Let $\L $ be the square $2n\times 2n$ matrix such that 
\begin{equation}\label{E:Lambda}
\L _{i,j}=
\begin{cases}
1 &\text{ if } i=j=1 \text{ or } i=j=2,
\\
0& \text{otherwise,}
\end{cases}
\end{equation}
so that $h(0)=\sqrt{\eta}\L \h+  (I_{2n}-\L) \h$.

Recall the power series notation
$f(\eta \tau ,h(0))=\sum \eta^k f^{(k)}(\tau, h(0))$. As a consequence of Proposition~\ref{P:expansion}, we can give a new expansion of the Hamiltonian flow for the special class of initial covectors of type \eqref{E:special_covector} in terms of coefficients of the power series of $x,z,h,w$.
(Recall that for all $R>0$, $B_R$ denotes the set $\{h\in \R^{2n}\mid \sum_{i=1}^{2n}h_i^2\leq R\}$.)

\begin{proposition}\label{P:expansion_r1_eta}
For all $T,R>0$,
normal extremals with initial covector 
$$
(\sqrt{\eta}\L \h+ (I_{2n}-\L) \h,\eta^{-1})
$$
 have the following order 3 expansion at time $\eta \tau$, as  $\eta\to 0^+$, uniformly with respect to $\tau\in [0,T]$ and $\h\in B_R$:
\begin{multline*}
	x(\eta \tau)=
	\eta \xhat(\tau,( I_{2n}-\L) \h)
	+
	\eta^{3/2}
		\left[
			\xhat\left(\tau,  \L \h\right)
		\right]
	+
	\eta^{2}
		\left[
			x^{(2)} \left(\tau,  (I_{2n}- \L) \h\right)
		\right]
	\\
	+
	\eta^{5/2}
		\left[
			x^{(2)} \left(\tau, \h\right)-x^{(2)} \left(\tau,  (I_{2n}- \L) \h\right)-x^{(2)} \left(\tau,  \L\h\right)
		\right]
		+O(		\eta^3),
\end{multline*}
$$
	z(\eta \tau)
	=
	\eta^2 
	\zhat(\tau,(I_{2n}-\L)\h)
	+\eta^3
	\left[
	z^{(3)}(\tau,(I_{2n}-\L)\h)+\zhat(\tau, \L \h)
	\right]
	+O(\eta^4).
$$

Likewise, the associated covector has the  expansion
\begin{multline*}
h(\eta \tau)= \hhat(\tau,(I_{2n}-\L) \h)
+\sqrt{\eta }
	\left[
		\hhat(\tau, \L \h)	
	\right]
+
\eta 
	\left[
		h^{(1)}(\tau,(I_{2n}-\L) \h)
	\right]
\\
+
\eta^{3/2}
	\left[
		h^{(1)}(\tau, \h)	
		-
		h^{(1)}(\tau, \L \h)	
		-
		h^{(1)}(\tau,(I_{2n}-\L) \h)	
	\right]
+O(\eta^2),
\end{multline*}
$$
w(\eta \tau)= 1+ O(\eta^2).
$$

\end{proposition}
\begin{proof}
Let $h,h'\in \R^{2n}$ and
let $\psi :\R^{2n}\rightarrow \R$ be a quadratic form, we have by polarization identity
$
\psi \left(h+\sqrt{\eta}h'\right)=\psi(   h)
+\sqrt{\eta}
\left[ \psi(  h+h') - \psi(  h) -\psi( h') \right]
+
\eta \psi( h')$.
Applying this identity with $h=\L\h$  and $h'=(I_{2n}-\L) \h$, we get the statement since we proved in Proposition~\ref{P:expansion} that 
$x^{(1)}(\eta \tau, \cdot)$, $h^{(0)}(\eta \tau, \cdot)$ are linear and $x^{(2)}(\eta \tau, \cdot)$, $h^{(1)}(\eta \tau, \cdot)$, $z^{(2)}(\eta \tau, \cdot)$ are quadratic, coordinate-wise. The  case of $w$ comes from the fact that $w^{(1)}=0$.
\end{proof}

We set $G(\tau,\h,\eta)=F\left(\tau ,\sqrt{\eta}\L \h+  (I_{2n}-\L) \h,\eta\right)$,  for all  $\tau>0$, $\h\in \R^{2n}$ and $\eta>0$.
The function $G$ admits a power series expansion in $\sqrt{\eta}$
$$
G(\tau,\h,\eta)=\sum_{k\geq 0} \eta^{k/2} G^{(k/2)}(\tau,\h).
$$
We prove the following proposition on the conjugate time for such initial covectors.
\begin{proposition}\label{P:t_c_bu_eta1/2}
Let us define the quadratic polynomial in $\delta\! \tau$
\begin{multline*}
P(\delta\! \tau)=-\delta\! \tau^2 K+ 
\delta\! \tau 
\left(\frac{2\pi}{b_1}\left(\h_1^2+\h_2^2\right)-K\left(\frac{\partial G^{(5/2)}_1}{\partial \h_1}+\frac{\partial G^{(5/2)}_2}{\partial \h_2}\right)\right)
\\
+
\frac{2\pi}{b_1}  \left(\h_2^2 \frac{\partial G^{(5/2)}_1}{\partial \h_1}+\h_1^2 \frac{\partial G^{(5/2)}_2}{\partial \h_2}-\h_1 \h_2 \left(\frac{\partial G^{(5/2)}_2}{\partial \h_1}+\frac{\partial G^{(5/2)}_1}{\partial \h_2}\right)\right)
\\
+
K \left(
			\frac{\partial G^{(5/2)}_2}{\partial \h_1} \frac{\partial G^{(5/2)}_1}{\partial \h_2}
			-
			\frac{\partial G^{(5/2)}_1}{\partial \h_1} \frac{\partial G^{(5/2)}_2}{\partial \h_2}
\right),
\end{multline*}
and let $\Delta(\h)$ be its discriminant.
We have the following cases:
\begin{itemize}
	\item If $\Delta (\h)\geq 0$, let $\delta\! \tau^*$ be the smallest of the (possibly equal) two roots of $P$. Then 
	$$
	\tau_c(\sqrt{\eta}\L \h+  (I_{2n}-\L) \h)=2\pi/b_1+\eta \delta  t^*+o(\eta).
	$$
	\item If $\Delta (\h)<0$, 
	$$
	\limsup\limits_{\eta\to 0} \left|\tau_c(\sqrt{\eta}\L \h+  (I_{2n}-\L) \h)-2\pi /b_1 \right|>0,
	$$	
	that is,
	the first conjugate time is not a perturbation of $2 \pi/b_1$.
\end{itemize}
\end{proposition}
\begin{proof}
We first have to check that the conjugate time is not a perturbation of order $\sqrt{\eta}$ of the nilpotent conjugate time $2\pi/b_1$.
We apply the same method as before to evaluate
$\Phi
\left(
2\pi /b_1+\sqrt{\eta}\delta\! \tau ,  \sqrt{\eta}\L\h+ (I_{2n}-\L)\h  ,\eta
\right)$,  $\delta\! \tau \in \R$, $\bar{h}\in \R^{2n}$. Notice that 
$$
\frac{\partial F}{\partial{h_i}}=\frac{1}{\sqrt{\eta}} \frac{\partial G}{\partial{\h_i}},\quad  \forall i\in \ll 1,2 \rr,
\quad \text{ and } \quad 
\frac{\partial F}{\partial{h_i}}=\frac{\partial G}{\partial{\h_i}}
\quad 
\forall i \in \ll 3, 2n\rr.
$$

With $\delta\! \tau\in \R$, $\h\in \R^{2n}$, we have
\begin{equation}\label{E:G_sqrt_eta}
G\left(\frac{2\pi}{b_1}+\sqrt{\eta}\delta\! \tau ,\h,\eta\right)
=
\eta 
\eval{
	G^{(1)}
}{\tau=\frac{2\pi}{b_1}}
+
\eta^{3/2}
\eval{\left(
	G^{(3/2)}
	+\delta\! \tau
	\frac{\partial G^{(1)}}{\partial \tau}
	\right)
}{\tau=\frac{2\pi}{b_1}}
+ O(\eta^{5/2}).
\end{equation}
Hence $\Phi
\left(
2\pi /b_1+\sqrt{\eta}\delta\! \tau ,  \sqrt{\eta}\L\h+ (I_{2n}-\L)\h  ,\eta
\right)=O(\eta^{2n+3})$ (see, for instance, Appendix~\ref{A:Conjugate_time_equations}).
By capturing the first non trivial term in the expansion of $\Phi$, one has
 $$
\Phi^{(2n+3)}
\left(
2\pi /b_1+\sqrt{\eta}\delta\! \tau,  \sqrt{\eta}\L\h+ (I_{2n}-\L)\h  ,\eta
\right) \propto \delta\! \tau^2
 $$
 (see also Lemma~\ref{L:deltat1/2} in the appendix).
Hence perturbations of the nilpotent conjugate time $2\pi/b_1$ must be of order 1 in $\eta$ at least for $\Phi$ to vanish.

Computing the perturbation of the conjugate time is then a matter of computing 
$\Phi$ at time $2\pi /b_1+\eta \delta\! \tau$.
Regarding $G$, we have 
\begin{equation}\label{E:G_sqrt_eta_bis}
\begin{aligned}
G\left(\frac{2\pi}{b_1}+\eta \delta\! \tau,\h,\eta\right)
=
&
\eta 
\eval{
	G^{(1)}
}{\tau=\frac{2\pi}{b_1}}
+
\eta^{3/2}
\eval{ 
	G^{(3/2)}
}{\tau=\frac{2\pi}{b_1}}
+
\eta^2
\eval{
\left[
	 G^{(2)}
	 +
	 \delta\! \tau
	 \frac{\partial G^{(1)}}{\partial \tau}
 \right]
}{\tau=\frac{2\pi}{b_1}}
\\
& +
\eta^{5/2}
\eval{
\left[
	 G^{(5/2)}
	 +
	 \delta\! \tau
	 \frac{\partial G^{(3/2)}}{\partial \tau}
 \right]
}{\tau=\frac{2\pi}{b_1}}
+ O(\eta^{3}).
 \end{aligned}
\end{equation}
Thus 
$\Phi
\left(
2\pi /b_1+\eta \delta\! \tau ,  \sqrt{\eta}\L\h+ (I_{2n}-\L)\h  ,\eta\right)=O(\eta^{2n+5})$. 
Again, computing the first nontrivial term in the expansion yields 
(for instance, see Lemma~\ref{P:mini_det_3})
$$
\Phi^{(2n+5)}
\left(
2\pi /b_1+\eta \delta\! \tau,h,\eta
\right)
\propto
P(\delta\! \tau).
$$
This implies the statement: either $P$ admits real roots, of which the smallest is $\tau_c^{(1)}$, or the system does not admit a perturbation of $2\pi/b_1$ as a first conjugate time.
\end{proof}
\begin{remark}
Contrarily to \eqref{E:def_tauc1_r1neq0}, the equation $P(\delta\! \tau)=0$ not degenerate at $\h_1=\h_2=0$.
\end{remark}

\subsection{Proof of Theorems~\ref{T:sre_is_not_so_simple} and \ref{T:big_expansion_theorem}}
It appears now that proving Theorem~\ref{T:big_expansion_theorem} is a matter of summarizing what we know about the conjugate time from the previous results of Section~\ref{S:order_2_approx}.
\begin{proof}[Proof of Theorem~\ref{T:big_expansion_theorem}]
In the previous section we computed the rescaled conjugate time $\tau_c$. We have for all covector $p_0=(\h_1,\dots ,\h_{2n},\eta^{-1})\in T_{q_0}^*M$,
$$
t_c(\h,\eta^{-1})=\eta\tau_c(\h,\eta^{-1})
$$
From Proposition~\ref{P:first_order_time}, we deduce that under the assumption $(\h_1,\h_2)\neq (0,0)$, we have as $\eta\to 0^+$ that $\tau_c(\h,\eta^{-1})=2\pi/b_1+O(\eta)$. From Proposition~\ref{P:conjugate_time_simple_case}, we deduce the existence of $t_c^{(2)}=\eta\tau_c^{(1)}$ that satisfies the given equation, using the  invariants introduced in Proposition~\ref{P:summary_kappa}.

On the other hand, by performing the blow up at $(0,0,\h_3\dots,\h_{2n})$, we compute an approximation of 
$$
t_c(\sqrt{\eta}\h_1,\sqrt{\eta}\h_2,\h_3,\dots ,\h_{2n},\eta^{-1})=\eta\tau_c(\sqrt{\eta}\h_1,\sqrt{\eta}\h_2,\h_3,\dots ,\h_{2n},\eta^{-1}).
$$
Again, from Proposition~\ref{P:first_order_time}, we deduce that under the assumption $(\h_1,\h_2)\neq (0,0)$, a possible approximation is $\tau_c(\sqrt{\eta}\L \h+(I_{2n}-\L)\h,\eta^{-1})=2\pi/b_1+O(\eta)$. However from Lemma~\ref{L:lim_r1_to_0}, we now know that in the nilpotent case, $2\pi/b_1$ is a zero of order two at $(\h_1,\h_2)= (0,0)$. Thus computing a perturbation of the conjugate time, one gets the statement for $\tilde{t}^{(2)}_c$ from Proposition~\ref{P:t_c_bu_eta1/2} and the expression in terms of invariants from Proposition~\ref{P:expansion_r1_eta}.
\end{proof}

Having proved Theorem~\ref{T:big_expansion_theorem}, we can introduce a geometrical invariant that will help us prove Theorem~\ref{T:sre_is_not_so_simple}.
For all $q\in M\setminus \S_1$, let 
$$
\mathcal{A}_{q}=\overline{\left\{t_c(p)p\mid H(p,q)=1/2\right\}}.
$$ 
By the usual property of the Hamiltonian flow, the first conjugate locus at $q$ is given by $\sre_{q}(1,\mathcal{A}_q)$. Furthermore, the set $\mathcal{A}_q$ is an immersed hypersurface of $T^*_{q}M$ and $\mathcal{A}_q\cap \C_{q}(0)$ is reduced to the two points  $p^+=(0,\dots, 0,2\pi/b_1)$, $p^-=(0,\dots, 0,-2\pi/b_1)$.
Then let $\mathcal{A}_{q}^+$ be the tangent cone to $\mathcal{A}_q$ at $p^+$. 

Observe that $\mathcal{A}_{q}^+$ is a geometrical invariant independent of the choice of coordinates on $M$. It can be computed once the asymptotics of the conjugate time are known.

\begin{proof}[Proof of Theorem~\ref{T:sre_is_not_so_simple}]
We prove the theorem by contradiction.
Assume there exists a set of coordinates for which \eqref{E:limit_exp-exphat} does not hold, {\it i.e.}
$$
\lim_{h_0\to +\infty}
\left(
h_0^2 \sup_{\tau\in (0,T)}
\left|
	\sre_{q}\left(\frac{\tau}{h_0},(h_1,\dots ,h_{2n},h_0)\right)-\widehat{\sre}_{q}\left(\frac{\tau}{h_0},(h_1,\dots ,h_{2n},h_0)\right)
\right|
\right)
= 0.
$$
Then we have  that uniformly with respect to $\tau\in (0,T)$,
$$
\sre_{q}\left(\eta \tau,(h_1,\dots ,\h_{2n},\eta^{-1})\right)=\widehat{\sre}_{q}\left(\eta \tau,(h_1,\dots ,\h_{2n},\eta^{-1})\right)+o(\eta^2).
$$
That is, the exponential is a second order perturbation of the nilpotent exponential. If that is the case, as a consequence of Section~\ref{S:order_2_approx}, and in particular Proposition~\ref{P:conjugate_time_simple_case}, we have that for $p_0=(h_1,\dots,h_{2n},\eta^{-1}) \in T_{q}^*M$, 
$$
t_c(p_0)=\frac{2\pi}{b_1}\eta+o(\eta^2).
$$
Then
$$
t_c(p_0)p_0= 
\left(
0,\dots, 0,\frac{2\pi}{b_1}
\right)
+
\eta
\left(
 \frac{2\pi}{b_1}h_1,\dots , \frac{2\pi}{b_1}h_{2n},0
\right)
+o(\eta)
$$
and the cone $\mathcal{A}_{q}^+$ is the affine plane $\{h_0=2\pi/b_1\}$.

However, as a consequence of Theorem~\ref{T:big_expansion_theorem}, the cone $\mathcal{A}_{q}^+$ can be computed using the Agrachev--Gauthier frame, where  we have for $p_0=(h_1,\dots,h_{2n},\eta^{-1}) \in T_{q}^*M\setminus S$,
$$
t_c(p_0)p_0= 
\left(
0,\dots, 0,\frac{2\pi}{b_1}
\right)
+
\eta 
\left(
\frac{2\pi}{b_1}h_1,\dots ,\frac{2\pi}{b_1}h_{2n}, t_c^{(2)}(h_1,\dots ,h_{2n})
\right)
+o(\eta).
$$
For $\mathcal{A}_{q}^+$ to be planar, the following symmetry for $t_c^{(2)}$ is needed (with $r_1^2=h_1^2+h_2^2$):
$$
\lim_{r_1\to 0^+} t_c^{(2)}(h_1,h_2,h_3,\dots ,h_{2n})=-\lim_{r_1\to 0^+} t_c^{(2)}(-h_1,-h_2 ,h_3,\dots ,h_{2n})
$$
for all $(h_3,\dots, h_{2n})\in \R^{2n-2}$.
Given the expression \eqref{E:linear_tc2}, we have rather
$$
\lim_{r_1\to 0^+} t_c^{(2)}(h_1,h_2,h_3,\dots ,h_{2n}) = \lim_{r_1\to 0^+} t_c^{(2)}(-h_1,-h_2 ,h_3,\dots ,h_{2n}),
$$
which is not everywhere zero unless $\gamma_{11}=\gamma_{22}=\gamma_{12}+\gamma_{21}=0$ for all $(h_3,\dots, h_{2n})\in \R^{2n-2}$. That is 
$\kappa^{1i}_1=\kappa^{2i}_2=\kappa^{2i}_1+\kappa^{1i}_2=0$ for all $i\in \ll 3,2n\rr$, which is not generic with respect to the sub-Rieman\-nian structure at $q\in M\setminus (\S_1\cup \S_3)$  (see Proposition~\ref{P:summary_kappa} and Appendix~\ref{A:Computation_invariants}).

In consequence, we have proven that generically with respect to the sub-Rieman\-nian structure at $q\in M\setminus \S$, there does not exist a set of privileged coordinates at $q$ and $T>0$ such that the limit \eqref{E:limit_exp-exphat} holds.
\end{proof}
\begin{remark}
Regarding the non-genericity of $\kappa^{1i}_1=\kappa^{2i}_2=\kappa^{2i}_1+\kappa^{1i}_2=0$, notice that it constitutes $6(n-1)$ independent conditions on the family $\left( \kappa^{ij}_{k}\right)_{\begin{smallmatrix} i,k\in \ll 1,2\rr, \\ j\in \ll 3,2n\rr\end{smallmatrix}}$ and thus a codimension $5(n-1)$ condition (at least) on the $2$-jets of the sub-Riemannian structure at $q$.

Notice that $5n-5> 2n+1$ if $n>2$ and $5n-5= 2n+1$ when $n=2$. Hence in the $n=2$ case, assuming  $q\in M\setminus\S_3$  (see Proposition~\ref{P:summary_kappa}), we ensure the codimension of the condition on the $2$-jets of the sub-Riemannian structure to be $6$.
\end{remark}

\subsection{Next order perturbations}\label{SS:third_order_conjugate}
\label{SS:third_order}

As observed in Section~\ref{SS:asymptotics_away_S}, there exists a subset of initial covectors in $T_{q_0}^*\setminus S_1$ for which our approximation of the conjugate locus is degenerate (this makes the  second order approximation unstable as a Lagrangian map). In particular, for all $q_0\in M$,  this set contains $S_2=\{(h_1,h_2,0,\dots, 0,\eta^{-1})\in T_{q_0}^*M\}$.
As proved in Proposition~\ref{P:def_S_2},  this set is reduced to $S_2$ at points $q_0$ in the complement of a startified codimension $1$ subset $\S_2$ of $M$ if $n=2$.

Hence in preparation of the stability analysis of Section~\ref{S:Stab}, we compute here a third order approximation of the conjugate time in the case of covectors near $S_2$. When $n=2$, we get  a complete description of the sub-Riemannian caustic at points of $M\setminus \S_2$ as a result.

We use a blowup technique similar to the one of  Section~\ref{SS:asymptotics_near_S_1}.
Let $\h\in\R^{2n}$ be such that $(\h_1,\h_2)\neq (0,0)$. We blowup the singularity at $(\h_1,\h_2,0,\dots,0)$ by computing an approximation of the conjugate locus with 
$$
h(0)=
(\h_1,\h_2,\eta \h_3,\dotsc,\eta \h_{2n}).
$$
With $\L $  the square $2n\times 2n$ matrix defined in \eqref{E:Lambda}, $h(0)=\L \h+\eta (I_{2n}-\L) \h$.

We give an equivalent of 
Proposition~\ref{P:expansion_r1_eta} for this case. 
\begin{proposition}\label{P:expansion_r2_eta}
For all $T,R>0$,
normal extremals with initial covector $(\L \h+\eta (I_{2n}-\L) \h,\eta^{-1})$ have the following order 3 expansion at time $\eta \tau$, as  $\eta\to 0^+$, uniformly with respect to $\tau\in [0,T]$ and $h(0)\in B_R$:
\begin{multline*}
	x(\eta \tau,\L \h+\eta (I_{2n}-\L) \h)=
	\eta \xhat(\tau, \L \h)
	+
	\eta^2
		\left[
			x^{(2)} \left(\tau,  \L \h\right)
			+
			\xhat\left(\tau,  (I_{2n}- \L) \h\right)
		\right]
	\\
		+
		\eta^3
		\left[
			x^{(3)} \left(\tau,  \L \h\right)
			+
			x^{(2)} \left(\tau,   \h\right)
			-
			x^{(2)} \left(\tau,   \L\h\right)
			-
			x^{(2)} \left(\tau, (I_{2n}-\L) \h\right)
		\right]
	+O(\eta^4),
	\end{multline*}
$$
	z(\eta \tau)
	=
	\eta^2 
	\zhat(\tau, \L\h)
	+\eta^3
	z^{(3)}(\tau,\L\h)
	+O(\eta^4).
$$

Likewise, the associated covector has the following expansion:
\begin{multline*}
h(\eta \tau,\L \h+\eta (I_{2n}-\L) \h)= \hhat(\tau,\L \h)
+
\eta 
	\left[
		h^{(1)}(\tau,\L \h)
		+
		\hhat(\tau,(I_{2n}-\L) \h)	
	\right]
\\
+
\eta^2
	\left[
		h^{(2)}(\tau, \L \h)
		+
		h^{(1)}(\tau,  \h)	
		-
		h^{(1)}(\tau,  \L \h)	
		-
		h^{(1)}(\tau,(I_{2n}-\L) \h)	
	\right]
+O(\eta^3),
\end{multline*}
$$
w(\eta \tau)= 1+\eta^2 w^{(2)}(\tau, \L\h)+ O(\eta^4).
$$

\end{proposition}

\begin{proof}
The proof relies on the same arguments as that of Proposition~\ref{P:expansion_r1_eta}.
\end{proof}

We aim to obtain a second order approximation of $\tau_c$ in the case of an initial covector of the form $(\L \h+\eta (I_{2n}-\L) \h,\eta^{-1})$, for $\h\in \R^{2n}$.
The previous section, together with Proposition~\ref{P:expansion_r2_eta}, applies to give us 
$$
\tau_c^{(1)}( \L\h+\eta (I_{2n}-\L)\h)=\tau_c^{(1)}( \L\h),\qquad \forall \h\in \R^{2n}.
$$ 
Similarly to Section~\ref{SS:asymptotics_near_S_1}, for all $\tau>0$, $h\in \R^{2n}$ and $\eta>0$, we denote $F(\tau,h,\eta)=\sre (\eta \tau ; (h,\eta^{-1}))$,
and we set
$$
G(\tau,\h,\eta)=F\left(\tau ,\L \h+\eta (I_{2n}-\L) \h,\eta\right) , \quad \forall \tau>0, \h\in \R^{2n},\eta>0.
$$

The function $G$ admits a formal power series expansion in $\eta$:
$
G(\tau,\h,\eta)=\\\sum_{k\geq 0} \eta^k G^{(k)}(\tau,\h)$. Techniques similar to those introduced in Sections~\ref{SS:asymptotics_away_S} and \ref{SS:asymptotics_near_S_1}  yield the following statement on second order approximations of the conjugate time $\tau_c$.
\begin{proposition}
The second order perturbation of $\tau_c$ with initial covector \\$h(0)= \L\h+\eta (I_{2n}-\L)\h$ satisfies the equation
\begin{multline*}
(\h_1^2+\h_2^2)\tau_c^{(2)}(h(0))
=
-\h_1^2  \frac{\partial \left(
			 G^{(3)}
			 \right)_2}{\partial \h_2}			
-
\h_2^2  \frac{\partial \left(
			 G^{(3)}
			 \right)_1}{\partial \h_1}	
+
\h_1 \h_2
\left(  \frac{\partial \left(
			 G^{(3)}
			 \right)_1}{\partial \h_2}
			 +
			 \frac{\partial \left(
			 G^{(3)}
			 \right)_2}{\partial \h_1}
\right)
\\
+
 (\h_1^2+\h_2^2)(\alpha \h_2-\beta \h_1)\left(
\frac{b_1}{2\pi}(\beta \h_1-\alpha \h_2)
+
4b_1(\alpha \h_1 +\beta \h_2)
\right)
+\sum_{i=3}^{2n} d_i,
\end{multline*}
where $\alpha$ and $\beta$ are the second order invariants introduced in Proposition~\ref{P:summary_kappa} and
$$
d_{k}=
\frac{2\pi^2}{b_1^2}e_k \left(-h_2 \partial_{h_{k}}\left(G^{(3)}\right)_1+h_1 \partial_{h_{k}}\left(G^{(3)}\right)_2\right)\qquad \forall k\in\ll 3,2n\rr,
$$
with $e\in \R^{2n-2}$ the vector such that $Ae= \left(h_2 \partial_{h_1}G^{(2)}-h_1  \partial_{h_2}G^{(2)}\right)_{3,\dots,2n} $, where $A\in \mathcal{M}_{2n-2}(\R)$  is the matrix introduced in Lemma~\ref{P:mini_det_1} and where we denote $\left(v\right)_{3,\dots,2n}=(v_3,\dots,v_{2n})\in \R^{2n-2}$ for all $v\in \R^{2n+1}$.
\end{proposition}

\begin{proof}

With $\delta\! \tau_1,\delta\! \tau_2 \in \R$, $\h\in \R^{2n}$, we have
\begin{multline*}
G\left(\frac{2\pi}{b_1}+\eta\delta\! \tau_1+\eta^2 \delta\! \tau_2,\h,\eta\right)
=
\eta 
\eval{
	G^{(1)}
}{\tau=\frac{2\pi}{b_1}}
+
\eta^2 
\eval{\left(
	G^{(2)}
	+\delta\! \tau_1
	\frac{\partial G^{(1)}}{\partial \tau}
	\right)
}{\tau=\frac{2\pi}{b_1}}
\\
+
\eta^3
\eval{
\left[
 G^{(3)}
 +
 \delta\! \tau_2
 \frac{\partial G^{(1)}}{\partial \tau}
 +
 \frac{\delta\! \tau_1^2}{2}
 \frac{\partial^2 G^{(1)}}{\partial \tau^2}
 +
  \delta\! \tau_1
 \frac{\partial G^{(2)}}{\partial \tau}
 \right]
}{\tau=\frac{2\pi}{b_1}}
 + O(\eta^3).
\end{multline*}

To evaluate $\Phi
\left(
2\pi /b_1+\eta\delta\! \tau_1+\eta^2 \delta\! \tau_2,  \L\h+\eta (I_{2n}-\L)\h  ,\eta
\right)$,  $\delta\! \tau_1,\delta\! \tau_2\in \R$, $\bar{h}\in \R^{2n}$, notice that 
$$
\frac{\partial F}{\partial{h_i}}=\frac{\partial G}{\partial{\h_i}},\quad \forall i\in \ll 1,2\rr
\quad\text{ and }\quad
\frac{\partial F}{\partial{h_i}}=\frac{1}{\eta} \frac{\partial G}{\partial{\h_i}},\quad \forall i\in \ll 3,2n\rr .
$$
Hence with  $\delta\! \tau_1=\tau_c^{(1)}(\L \h)$, one has
 $
 \Phi
\left(
2\pi /b_1+\eta \delta\! \tau_1+\eta^2 \delta\! \tau_2,  \L\h+\eta (I_{2n}-\L)\h  ,\eta
\right)
= O(\eta^{4n+2})$.
The result is again obtained by computing the first nontrivial term in the expansion of the determinant $\Phi$ (see Lemma~\ref{L:mini_det_4}). We obtain the stated result by refining this evaluation thanks to Lemma~\ref{L:expression_d_i}.
\end{proof}

Up to the computation of $G^{(3)}$, which is carried out in Appendix~\ref{A:third_order}, we have enough information to compute the conjugate time, similarly to Proposition~\ref{P:conjugate_time_simple_case}.

\begin{remark}
By definition of the invariants $\chi_{11},\chi_{12},\chi_{22}$ introduced in Appendix~\ref{A:Computation_invariants}, the third dimensional case would correspond to the case $\kappa^{ij}_{k}=0$ if $3\leq i,j,k\leq 2n$, $\alpha=\beta	=0$.
Under these conditions, one has 
$\tau_c^{(1)}(\h)=0$, $\tau_c^{(2)}(\h)=-3 (\chi_{11}+\chi_{22})(\h_1^2+\h_2^2)$
and
$$
\begin{aligned}
\left[\sre (\eta \tau_c ; (h,\eta^{-1}))\right]_1
=
\eta^3\left(2 \h_1^3 (\chi_{22}-\chi_{11})+3 \h_1^2 \h_2 \chi_{12}+\h_2^3 \chi_{12}\right)+O(\eta^4),
\\
\left[\sre (\eta \tau_c ; (h,\eta^{-1}))\right]_2
=
\eta^3\left(2 \h_2^3 (\chi_{11}-\chi_{22})+3 \h_1 \h_2^2 \chi_{12}+\h_1^3 \chi_{12}\right)+O(\eta^4).
\end{aligned}
$$
This expression corresponds to the classical astroidal caustic expansion observed in the $3$-dimensional contact case.
\end{remark}


\section{Stability of the sub-Riemannian caustic}
\label{S:Stab}

\subsection{Sub-Riemannian to Lagrangian stability}
\label{SS:SRtoLagrangian}
The aim of the whole classification is to prove Theorem~\ref{T:Stability}. 
Recall we denote by $\sre_{q_0}^1:T_{q_0}^*M\to M$ the sub-Riemannian exponential at time $1$, that is $\sre_{q_0}^1=\sre_{q_0}(1,\cdot)$.
We first  observe the following immediate fact.

\begin{proposition}\label{P:Lag_impl_Sub}
Let $(M,\Delta,g)$ be a sub-Riemannian manifold and  let $q_0\in M$. If the exponential map at time $1$,  $\sre_{q_0}^1:T_{q_0}^*M\to M$, is Lagrange stable at $p\in T_{q_0}M$, then 
 $\sre_{q_0}^1$ is sub-Riemannian stable at $p$.
\end{proposition}

As a consequence of Proposition~\ref{L:no_singular_near_0}, classifying Lagrangian stable singularities of the sub-Riemannian exponential  near the starting point $q_0$ requires  considering  inital covectors in $\C_{q}(1/2)$ such that $h_0$ is very large. 
As stated in the previous sections, some restrictions on the starting point are necessary to prove stability. Hence we consider points on the complementary of a codimension $1$ stratified subset $\S$ of $M$, containing $\S_1$, $\S_2$ and $\S_3$, introduced in Section~\ref{S:Singularities_SR_exponential}, Proposition~\ref{P:def_S_2} and Proposition~\ref{P:summary_kappa} respectively. In Section~\ref{SS:classification_three}, we prove the following theorem, of which  Theorem~\ref{T:Stability} is a corollary.

\begin{theorem}\label{T:stability_time_dep_exp}
Let $(M,\Delta,g)$ be a generic $5$-dimensional contact sub-Riemannian manifold and let $q_0\in M\setminus  \S$. There exist $\bar{\eta}>0$ such that for all 
$(h_1,h_2,h_3,h_4,h_0)\in \C_{q}(1/2)\cap \{|h_0|> \bar{\eta}^{-1}\}$, the first conjugate point of $\sre_{q_0}$ with initial covector $(h_1,h_2,h_3,\\h_4,h_0)$ is a Lagrange stable singular point of type  $\mathcal{A}_2$, $\mathcal{A}_3$, $\mathcal{A}_4$, $\mathcal{D}_4^+$ or $\A_5$.
\end{theorem}

Assuming Theorem~\ref{T:stability_time_dep_exp} holds, we can now Theorem~\ref{T:Stability}.
\begin{proof}[Proof of Theorem~\ref{T:Stability}]
As a consequence of Proposition~\ref{P:Lag_impl_Sub}, we prove the Lagrange stability of the singular points of $\sre_{q_0}^1$. For all $t>0$, $p_0\in T_{q_0}^*M$, $\sre_{q_0}^1(t p_0)=\sre_{q_0}(t,p_0)$. Hence for a given covector $p_0 \in \{ H\neq 0 \}$,  $t_c(p_0) p_0$ is a critical point of $\sre_{q_0}^1$.

Recall that for all $q\in M$, we have set $\mathcal{A}_{q_0}=\overline{\left\{t_c(p_0)p_0\mid H(p_0,q_0)=1/2\right\}}$, and the caustic is the set $\sre_{q_0}^1\left(\mathcal{A}_{q_0} \right)$.

Since $\sre_{q_0}^1\left(\C_{q}(0)\right)=q_0$, to prove the statement it is sufficient to show the existence of $V_{q_0}$ neighborhood of $q_0$ such that $\sre_{q_0}^1$ is Lagrange stable at every point of 
$\mathcal{A}_{q_0}\cap  {\left(\sre_{q_0}^1\right)}^{-1}\left(V_{q_0}\right)\cap \{H>0\} $ (and satisfies the stated classification). As a result of Theorem~\ref{T:stability_time_dep_exp}, what remains to prove is that there exists $R>0$ such that for all covectors $p\in \mathcal{A}_{q_0}\cap\C_{q_0}((0,R))$, 
$$
\frac{p}{\sqrt{2H(p,q_0)}}\in  \C_{q}(1/2)\cap \{|h_0|> \bar{\eta}^{-1}\}
$$
 with $\bar{\eta}>0$ as in the statement of Theorem~\ref{T:stability_time_dep_exp}, but this is Proposition~\ref{L:no_singular_near_0}.
\end{proof}

\subsection{Classification methodology}

We first recall normal forms for the stable singularities that appear in Theorem~\ref{T:stability_time_dep_exp}.
\begin{definition}
Let $f:\R^5\rightarrow \R^5$ be a smooth map singular at $q\in \R^5$.
Assume there exist variables $\x$ centered at $q$ and  and variables centered at $f(q)$ such that 
\begin{itemize}
\item
$f(\x_1,\dots, \x_5)=(\x_1^2,\,\x_2,\,\x_3,\,\x_4,\,\x_5)$, then the singularity is of type $\A_2$;

\item
$f(\x_1,\dots, \x_5)=(\x_1^3+\x_1\x_2 , \, \x_2, \, \x_3, \, \x_4, \, \x_5)$,  then the singularity is  of type $\A_3$;

\item
$f(\x_1,\dots, \x_5)=(\x_1^4+\x_1^2\x_2+\x_1\x_3, \, \x_2, \, \x_3, \, \x_4, \, \x_5)$,  then the singularity is  of type $\A_4$;

\item
$f(\x_1,\dots, \x_5)=(\x_1^5+\x_1^3\x_2+\x_1^2\x_3+\x_1\x_4, \, \x_2, \, \x_3, \, \x_4, \, \x_5)$,  then the singularity is   of type $\A_5$;

\item
$f(\x_1,\dots, \x_5)=(\x_1^2+\x_2^2+\x_1\x_3, \, \x_1\x_2, \, \x_3, \, \x_4, \, \x_5)$,  then the singularity is of type $\mathcal{D}_4^+$.

\end{itemize}
\end{definition}

We use these normal forms to characterize the singularities in terms of jets.
Let $M$ be a $5$-dimensional manifold, let $q_0\in M$ and let $g:T^*_{q_0}M\to M$ be a Lagrangian map. Let $p_0$ be a critical point of $g$. We transpose the normal form definition of stable singularities to condition on the jets of $g$. Given a set of coordinates $\x$ on $T^*_{q_0}M$,
let us introduce the functions
 (depending on whether the kernel of the Jacobian matrix of $g$ is of dimension 1 or 2)
$$
\phi_{i_1\dots i_k}(p_0)=\det\left(\partial_{\x_{i_1}}\dots \partial_{\x_{i_k}} g, V_2,V_3,V_4,V_5\right),
\quad 
\text{ if }\dim \ker \Jac_{p_0} g=1,
$$
$$
\phi'_{i_1\dots i_k}(p_0)=\det\left(\partial_{\x_{i_1}}\dots \partial_{\x_{i_k}} g, \partial_{\x_1}\partial_{\x_2}g,V'_3,V'_4,V'_5\right),
\quad 
\text{ if } \partial_{\x_1}g=\partial_{\x_2}g=0.
$$
(Where we denote by $V_2,V_3,V_4,V_5$, linearly independent vectors, depending smoothly on $p_0$, generating $\im \Jac_{p_0} g$ if $\dim \ker \Jac_{p_0} g=1$ and likewise
$V'_3,V'_4,V'_5$, linearly independent vectors, depending smoothly on $p_0$, generating $\im \Jac_{p_0} g$ if $\dim \ker \Jac_{p_0} g=2$.)

In terms of $\phi_{i_1,\dots i_k}$, we have the following characterization of Lagrangian equivalence classes.

\begin{proposition}\label{C:sing_phi}
Let $M$ be a $5$-dimensional manifold, let $g:T^*_{q_0}M\to M$ be a Lagrangian map and let $p_0\in T_{q_0}^*M$. Assume  $\ker \Jac_{p_0} g $ is $1$-dimensional, if there exists coordinates $(\x_1,\x_2,\x_3,\x_4,\x_5)$ such that $\partial_{\x_1}g(p_0)=0$ and the following holds at $p_0$
\begin{itemize}
\item $ \phi_{11} \neq 0$, then $p_0$ is a singular point of type $\A_2$;

\item  $ \phi_{11} =0$, $ \phi_{111}\cdot \phi_{12} \neq 0$, then $p_0$ is a singular point of type $\A_3$;

\item  $ \phi_{11} =\phi_{111}  = \phi_{12} =0$, $ \phi_{1111} \cdot \phi_{112} \cdot \phi_{13} \neq 0$, then $p_0$ is a singular point of type $\A_4$;

\item $ \phi_{11} =\phi_{111}   = \phi_{12} = \phi_{1111} = \phi_{112} = \phi_{13} =0$, $ \phi_{11111}  \cdot \phi_{1112}  \cdot \phi_{113}  \cdot \phi_{14} \neq 0$, then $p_0$ is a singular point of type $\A_5$.
\end{itemize}
Assume   $\ker \Jac_{p_0} g$ is $2$-dimensional, if there exists coordinates $(\x_1,\x_2,\x_3,\x_4,\x_5)$ such that $\partial_{\x_1}g=\partial_{\x_2}g=0$ and 
$\phi'_{11}\cdot \phi'_{22}(p_0)>0 $, $\phi'_{13}(p_0)\neq 0$ then $p_0$ is a singular point of type $\D_4^+$.
\end{proposition}
\begin{proof}
This is a matter of proving that $g$ has the same $k$-jets as the normal form for $\A_k$ singularities, $k\in \ll 2,5\rr$, and $2$-jet for $\mathcal{D}_4^+$.    For each of the stated cases, the existence of changes of coordinates at $p_0$ and $g(p_0)$ such that it is the case is then warranted by the stated conditions.
\end{proof}
\begin{remark}
The condition $\phi'_{11}\cdot \phi'_{22}(p_0)>0 $ corresponds to the distinction between $\mathcal{D}_4^+$ and $\mathcal{D}_4^-$ singularities, the latter corresponding to the opposite sign.
\end{remark}

Recall that we are considering points $q_0\in M\setminus (\mathfrak{S}_1\cup \mathfrak{S}_2)$, where $\mathfrak{S}_1$ (introduced at the beginning of Section~\ref{S:conjugate_time}) and $\mathfrak{S}_2$ (introduced in Proposition~\ref{P:def_S_2}) are both stratified subsets of $M$ of codimension $1$ at most.

Let $(M,\Delta,g)$ be a contact sub-Riemannian manifold of dimension $5$ and let $q_0\in M$. 
To study the sub-Riemannian caustic at $q_0$, we study for a given $p_0$ the stability at $p_0\in \C_{q}(1/2)$ of $\sre_{q_0}(t_c(p_0),\cdot)$. To apply Proposition~\ref{C:sing_phi}, we first compute an approximation the linear spaces $\ker \Jac_{p_0} \sre_{q_0}(t_c(p_0))$ and  $\im \Jac_{p_0} \sre_{q_0}(t_c(p_0))$. Then we compute approximations of the functions $\phi_{i_1\dots i_k}$ with by approximating the map
$$
v\mapsto \det\left(v,\im \Jac_{p_0} \sre_{q_0}(t_c(p_0))\right),
$$
for a well-chosen representation of $\im \Jac_{p_0} \sre_{q_0}(t_c(p_0))$.

Let $S_1
=
\left\{
p\in T_{q_0}^*M \mid h_1=h_2=0
\right\}$
and 
$
S_2
=
\left\{
p\in T_{q_0}^*M \mid h_3=h_4=0
\right\}
$.
As a consequence of Section~\ref{S:conjugate_time}, this stability analysis is carried independently on the three domains of initial covectors, $T^*_{q_0}M\setminus (S_1\cup S_2)$, near $S_1$ and near $S_2$, after blowup of the exponential map.

\begin{remark}\label{R:method_derivatives}
Let $\tau\in \R^+$ and $(h,\eta)\in \R^5$. The map $\sre_{q_0}(\eta \tau)$ is critical at $(h,\eta^{-1})$ if there exists $v\in \R^5$ such that $\Jac_{p_0} \sre_{q_0}(\eta \tau) \cdot v=0$.

With $F(\tau,h,\eta)=\sre_{q_0} (\eta \tau ; (h,\eta^{-1}))$, for all $\tau>0$, $h\in \R^{4}$, $\eta>0$, we denote $\partial_i=\partial_{h_i}$, for all $i\in \ll 1, 4 \rr $, and $\partial_5=\partial_{h_0}=-\eta^2 \partial_{\eta}+\eta \tau \partial_{\tau}$, we have
$$
\Jac_{p_0} \sre_{q_0}(\eta \tau)= \left(\partial_1 F, \partial_{2}F,\partial_{3}F,\partial_{4}F,\partial_{5}F\right).
$$
Higher order derivations of the map $F$ are then  computed using the chain rule.

\end{remark}

\begin{remark}
Precisely checking the conditions of Proposition~\ref{C:sing_phi} requires explicit computations executed in the computer algebra system Mathematica.
\end{remark}

\subsection{Classification of singular points of the caustic}
\label{SS:classification_three}

Let $p_0\in T_{q_0}^*M$ and let  $(\x_1,\x_2,\x_3,\x_4,\x_5)$ be arbitrary coordinates on a neighborhood of  $p_0\in T_{q_0}^*M$ such that $\partial_{\x_1}\sre_{q_0}(t_c(p_0))=0$. Apart from singularities of type $\D_4^+$ on the second domain, only singularities of corank $1$ are expected. Hence gauging the degree of the singularities is sufficient to classify them, provided that singularities of degree $k$ effectively correspond to singularities of type $\A_k$.

\subsubsection{First domain}

Consider initial covectors of the form $(h_1,h_2,h_3,h_4,{\eta}^{-1})$.
Algebraic computations, similar to those of the previous sections and left as appendix,  lead to the following proposition on the  $\phi$ functions.
(See Appendix~\ref{A:dom1}.)

(With $n=2$, recall that for all $R>0$, $B_R$ denotes the set $\{h\in \R^{4}\mid \sum_{i=1}^{4}h_i^2\leq R\}$.)
\begin{proposition}
Let us denote $p_0=(h_1,h_2,h_3,h_4,\eta^{-1})$.
There exist a family of vectors $(V_2,V_3,V_4,V_5)$,  smoothly  depending on $p_0$, generating $\im \Jac_{p_0} \sre_{q_0}(t_c(p_0))$ for which we have the following.
For all $R>0$, uniformly with respect to $h\in B_R$,  as $\eta\to 0$ 
$$
\phi_{11}(p_0)=O(\eta^8),\qquad \phi_{111}(p_0)=O(\eta^8), \qquad \phi_{1111}(p_0)=O(\eta^8). 
$$
Furthermore, there exists a function $\Psi:\R^4\times \R^5\to \R$ such that for all $V\in \R^5$,
$\Psi(h,V)\neq 0$ implies $V\notin \im \Jac_{p_0} \sre_{q_0}(t_c(p_0))$
and with 
$$
\Psi_k(h)=\Psi\left(h, \partial_{\x_1}^k \sre_{q_0}(t_c(p_0)) \right)^{(2)},\qquad  \forall k\in\ll 2,4\rr,$$
we have
$$
\Psi_2(h)=\phi_{11}^{(8)}(h),
\qquad 
\Psi_3(h)=\phi_{111}^{(8)}(h),
\qquad
\Psi_4(h)=\phi_{1111}^{(8)}(h).
$$
\end{proposition}

As a consequence of this proposition we obtain that  for $\eta$ small enough
$$
\Psi_2(h)\neq 0 \Rightarrow \phi_{11}(p_0)\neq 0,
\qquad
\Psi_3(h)\neq 0 \Rightarrow \phi_{111}(p_0)\neq 0,
$$
$$
\Psi_4(h)\neq 0 \Rightarrow \phi_{1111}(p_0)\neq 0.
$$
We can further numerically check as an application of Proposition~\ref{C:sing_phi} that 
\begin{itemize}
\item if $\Psi_2\neq 0$ then the singularity is of type $\A_2$;
\item if $\Psi_3\neq 0$ and the singularity is not of type $\A_2$ then the singularity is of type $\A_3$;
\item if $\Psi_4\neq 0$ and the singularity is not of type $\A_2,\A_3$ then the singularity is of type $\A_4$.
\end{itemize}
Then we have the following conclusion.

\begin{proposition}\label{P:classification_domaine_1_fin}
Let $(M,\Delta,g)$ be a generic sub-Riemannian structure and let $q_0\in M\setminus \mathfrak{S}$. There exists $\bar{\eta}>0$ such that for all covectors $p_0$ in $(\C_{q}(1/2)\cap \{h_0>\bar{\eta}^{-1}\})\setminus(S_1\cup S_2)$, the singularity at $p_0$ of $\sre_{q_0}(t_c(p_0))$ is a Lagrange stable singular point of type $\mathcal{A}_2$, $\mathcal{A}_3$ or $\mathcal{A}_4$.
\end{proposition}

\begin{proof}
As a consequence of our discussion, what remains to be proved is that 
generically with respect to the sub-Riemannian structure, there are no points  $(h_1,h_2,\\h_3,h_4)\in (\R^2\setminus \{0\})\times (\R^2\setminus \{0\})$ such that 
$$
\Psi_2(h_1,h_2,h_3,h_4)=\Psi_3(h_1,h_2,h_3,h_4)= \Psi_4(h_1,h_2,h_3,h_4)=0.
$$
However, one can check that this equation admits solutions in $(\R^2\setminus \{0\})\times (\R^2\setminus \{0\})$ only if $q_0\in \mathfrak{S}_2$. By assumption $\S_2\subset \S$, hence the statement.
\end{proof}

\subsubsection{Second domain}

Consider initial covectors of the form $(\sqrt{\eta}h_1,\sqrt{\eta}h_2, h_3, \\h_4,\eta^{-1})$.
Again, algebraic computations  left as appendix   lead to the following proposition on the  $\phi$ functions.
(See Appendix~\ref{A:dom2}.)
\begin{proposition}\label{P:second_domain_func}
Let us denote $p_0=(\sqrt{\eta}h_1,\sqrt{\eta}h_2,h_3,h_4,\eta^{-1})$.
Let $S^+$ be the subset of $T_{q_0}^*M$ where $\dim \ker\Jac_{p_0} \sre_{q_0}(t_c(p_0))=2$. 

For $p_0\notin S^+$, $\dim \ker\Jac_{p_0} \sre_{q_0}(t_c(p_0))=1$, 
there exist a family of vectors $(V_2,V_3,V_4,V_5)$,  smoothly  depending on $p_0$,  generating $\im \Jac_{p_0} \sre_{q_0}(t_c(p_0))$ for which we have the following.
For all $R>0$, uniformly with respect to $h\in B_R$,  as $\eta\to 0$ 
$$
\phi_{11}(p_0)=O(\eta^{10}),\quad \phi_{111}(p_0)=O(\eta^{10}), \quad \phi_{1111}(p_0)=O(\eta^{10}), \quad \phi_{11111}(p_0)=O(\eta^{10}). 
$$
Furthermore, there exists a function $\Phi:\R^4\times \R^5\to \R$ such that for all $V\in \R^5$,
$\Phi(h,V)\neq 0$ implies $V\notin \im \Jac_{p_0} \sre_{q_0}(t_c(p_0))$
and with 
$$
\Phi_k(h)=\Phi\left(h, \partial_{\x_1}^k \sre_{q_0}(t_c(p_0)) \right)^{(5/2)},\qquad  \forall k\in\ll 2,4\rr,$$
we have
$$
\phi_{11}^{(10)}(h)=\Phi_2(h),\quad \phi_{111}^{(10)}(h)=\Phi_3(h),\quad
\phi_{1111}^{(10)}(h)=\Phi_4(h),\quad \phi_{11111}^{(10)}(h)=\Phi_5(h).
$$

\end{proposition}

As a consequence of Remark~\ref{R:2dim_ker}, we can check that the singularity is of type $\D_4^+$ if $p_0\in S^+$ and that that singular points of the exponential of the such that $(h_1,h_2)=(0,0)$ are of type $\A_3$.

As an application of Proposition~\ref{P:second_domain_func}, we  obtain that  for $\eta$ small enough, if $p_0\notin S^+$, 
$$
\Phi_2(h)\neq 0 \Rightarrow \phi_{11}(p_0)\neq 0,
\qquad
\Phi_3(h)\neq 0 \Rightarrow \phi_{111}(p_0)\neq 0,
$$
$$
\Phi_4(h)\neq 0 \Rightarrow \phi_{1111}(p_0)\neq 0
\qquad
\Phi_5(h)\neq 0 \Rightarrow \phi_{11111}(p_0)\neq 0,
$$
We can further numerically check as an application of Proposition~\ref{C:sing_phi} that 
\begin{itemize}
\item if $\Phi_2\neq 0$ then the singularity is of type $\A_2$;
\item if $\Phi_3\neq 0$ and the singularity is not of type $\A_2$ then the singularity is of type $\A_3$;
\item if $\Phi_4\neq 0$ and the singularity is not of type $\A_2,\A_3$ then the singularity is of type $\A_4$;
\item if $\Phi_5 \neq  0$and the singularity is not of type $\A_2,\A_3,\A_4$ then the singularity is of type $\A_5$.
\end{itemize}
Then we have the following conclusion.

\begin{proposition}
Let $(M,\Delta,g)$ be a generic sub-Riemannian structure and let $q_0\in M\setminus \mathfrak{S}$. There exists $\bar{\eta}>0$ such that for all covectors $p_0$ in $\C_{q}(1/2)\cap \{h_0>\bar{\eta}^{-1}\}\cap\{h_1^2+h_2^2<\bar{\eta} \}$, the singularity at $p_0$ of $\sre_{q_0}(t_c(p_0))$ is a Lagrange stable singular point of type $\mathcal{A}_2$, $\mathcal{A}_3$, $\mathcal{A}_4$, $\mathcal{A}_5$ or $\D_4^+$.
\end{proposition}

\begin{proof} 
As a consequence of our discussion and Proposition~\ref{P:second_domain_func}, what remains to be proved is that  
there are no element  $(h_1,h_2,h_3,h_4)\in (\R^2\setminus \{0\})\times (\R^2\setminus \{0\})$ such that $\Phi_2(h)=\Phi_3(h)=\Phi_4(h)=\Phi_5(h)=0$.

Similarly to the proof of Proposition~\ref{P:classification_domaine_1_fin}, this is excluded on the complementary of $\S$.
\end{proof}

\begin{remark}
An intuition can be given on the reason $\A_5$ singularities can appear on the second (and third) domain but not the first one.  In the first domain, our approximation of the exponential presents symmetries that do not appear in the other domains. For instance these symmetries appear in the computations of the approximations of the $\phi$ functions of Proposition~\ref{C:sing_phi}. 

Indeed, we have on the first domain a two-parameter symmetry: for all $\lambda,\mu>0$, $h\in \R^4$,  
$$
\Psi_i(\lambda h_1, \lambda h_2 , \mu h_3, \mu h_4)= \lambda^2\mu \Psi_i(h_1,h_2,h_3,h_4), \qquad i\in\ll2,4\rr.
$$ 
On the second domain on the other hand, we only have a one-parameter symmetry:
$$
\Phi_i(\lambda^3 h_1,\lambda^3 h_2,\lambda^2 h_3,\lambda^2 h_4)=\lambda^{14}\Phi_i(h_1,h_2,h_3,h_4),
\qquad i\in\ll2,5\rr.
$$

In other words, the exponential map reduces to a 3-dimensional Lagrangian map on the first domain and only singularities of type $\A_2$ to $\A_4$ should appear.
Conversely, the symmetry on the second domain implies that the exponential reduces to a 4-dimensional Lagrangian map and $\A_5$ singularities can be expected.

A similar argument can be made in the 3-dimensional contact case for the presence of $\A_2$ and $\A_3$ singularities (see \cite{ABB_2018} for instance).
\end{remark}

\subsubsection{Third domain}
Consider initial covectors of the form $(h_1,h_2,\eta h_3,\eta  h_4,\\ \eta^{-1})$.
Algebraic computations  left as appendix  lead to the following proposition on the  $\phi$ functions.
(See Appendix~\ref{A:dom3}.)

\begin{proposition}
Let us denote $p_0=(h_1,h_2,\eta h_3, \eta h_4,\eta^{-1})$. There exist a family of vectors $(V_2,V_3,V_4,V_5)$,  smoothly  depending on $p_0$,  generating $\im \Jac_{p_0} \sre_{q_0}(t_c(p_0))$ for which we have the following.
For all $R>0$, uniformly with respect to $h\in B_R$,  as $\eta\to 0$,
$$
\phi_{11}(p_0)=O(\eta^{11}),\quad \phi_{111}(p_0)=O(\eta^{11}), \quad \phi_{1111}(p_0)=O(\eta^{11}),\quad  \phi_{11111}(p_0)=O(\eta^{11}). 
$$
Furthermore,
there exists a function $\Gamma:\R^4\times \R^5\to \R$ such that for all $V\in \R^5$,
$\Gamma(h,V)\neq 0$ implies $V\notin \im \Jac_{p_0} \sre_{q_0}(t_c(p_0))$ 
and with 
$$
\Gamma_k(h)=\Gamma\left(h, \partial_{\x_1}^k \sre_{q_0}(t_c(p_0)) \right)^{(3)},\qquad  \forall k\in\ll 2,5\rr,$$
we have
$$
\phi_{11}^{(11)}(h)=\Gamma_2(h),\quad \phi_{111}^{(11)}(h)=\Gamma_3(h),\quad \phi_{1111}^{(11)}(h)=\Gamma_4(h), \quad \phi_{11111}^{(11)}(h)=\Gamma_5(h).
$$
\end{proposition}

As a consequence of this proposition we obtain that  for $\eta$ small enough
$$
\Gamma_2(h)\neq 0 \Rightarrow \phi_{11}(p_0)\neq 0,
\qquad
\Gamma_3(h)\neq 0 \Rightarrow \phi_{111}(p_0)\neq 0,
$$
$$
\Gamma_4(h)\neq 0 \Rightarrow \phi_{1111}(p_0)\neq 0,
\qquad
\Gamma_5(h)\neq 0 \Rightarrow \phi_{11111}(p_0)\neq 0.
$$
We can further numerically check as an application of Proposition~\ref{C:sing_phi} that 
\begin{itemize}
\item if $\Gamma_2\neq 0$ then the singularity is of type $\A_2$;
\item if $\Gamma_3\neq 0$ and the singularity is not of type $\A_2$ then the singularity is of type $\A_3$;
\item if $\Gamma_4\neq 0$ and the singularity is not of type $\A_2,\A_3$ then the singularity is of type $\A_4$;
\item if $\Gamma_5\neq 0$ and the singularity is not of type $\A_2,\A_3,\A_4$ then the singularity is of type $\A_5$.
\end{itemize}
Then we have the following conclusion.
\begin{proposition}
Let $(M,\Delta,g)$ be a generic sub-Riemannian structure and let $q_0\in M\setminus \mathfrak{S}$. There exists $\bar{\eta}>0$ such that for all covectors $p_0$ in $\C_{q}(1/2)\cap \{h_0>\bar{\eta}^{-1}\}\cap\{h_3^2+h_4^2<\bar{\eta}^2 \}$, the singularity at $p_0$ of $\sre_{q_0}(t_c(p_0))$ is a Lagrange stable singular point of type $\mathcal{A}_2$, $\mathcal{A}_3$, $\mathcal{A}_4$ or $\A_5$.
\end{proposition}

\begin{proof}
The argument is the same as in the other two cases, that is, as a consequence of our discussion, there are no points  $h\in (\R^2\setminus \{0\})\times (\R^2)$ such that 
$\Gamma_2(h)=\Gamma_3(h)= \Gamma_4(h)=\Gamma_5(h)=0$. Again, this is excluded on the complementary of $\S$.
\end{proof}


\newpage

\appendix


\section{Agrachev--Gauthier normal form}\label{A:Gauthier}
Let $(M,\Delta,g)$ be a contact sub-Riemannian manifold of dimension $2n+1$.
In \cite{Gauthier_2001_SR_metrics_and_isoperimetric_problems}, the authors prove the existence at any $q_0\in M$ of a set of coordinates and vector fields for which the contact sub-Riemannian structure satisfies interesting symmetries. Here we recall the properties of this normal form, that we call Agrachev--Gauthier normal form.

On a contact manifold, there exists a $1$-form $\omega$ such that $\omega \wedge (\diff \omega)^n$ never vanishes and  $\ker \omega =\Delta$. Notice that for any smooth non-vanishing function  $f:M\to \R$, $\ker f\omega =\Delta$. Hence $\omega$ can be chosen so that 
$$
\left(\diff \omega\right)^n_{|\Delta}= \mathrm{vol}_g
$$
where $\mathrm{vol}_g$ is the volume form induced by $g$ on $\Delta$. Then there exists a unique vector field $X_0$, the Reeb vector field, such that 
$$
\omega(X_0)=1\quad \text{ and } \quad \iota_{X_0} \diff \omega =0.
$$

In the following, for any vector field $Y$, for all $i\in \ll1, 2n+1\rr$, we denote by $(Y)_i$ the $i$-th coordinate of $Y$ written in the basis $(\partial_{x_1},\dots,\partial_{x_{2n}},\partial_z)$.

\begin{theorem}[{\cite[Section 6]{Gauthier_2001_SR_metrics_and_isoperimetric_problems}}]\label{T:Ag_Gau_Frame_simpl}
Let $(M,\Delta,g)$ be a contact sub-Riemannian manifold of dimension $2n+1$ and $q_0\in M$. There exist privileged coordinates at $q_0$, $(x_1,\dots x_{2n},z):M\to \R^{2n+1}$, and a frame of $(\Delta,g)$, $(X_1,\dots ,X_{2n})$, that satisfy the following properties on a small neighborhood of $q_0=(0,\dots,0)$.
\begin{enumerate}[label={\bf (\arabic*)}]
\item \label{C:Gauthier_frame_hor} The horizontal components of the vector fields $X_1,\dots, X_{2n}$ satisfy the following two symmetries:
for all $1\leq i,j\leq 2n$, we have 
$$
\left(X_i\right)_j=\left(X_j\right)_i
$$ 
and 
$$
\sum_{j=1}^{2n} \left(X_j\right)_i x_j=x_i.
$$

\item \label{C:Gauthier_frame_vert} The vertical components of $X_1,\dots, X_{2n}$ satisfy the  symmetry
$$
\sum_{j=1}^{2n} \left(X_j\right)_{2n+1} x_j = 0.
$$

\item \label{C:Gauthier_frame_Reeb}$X_0=\frac{\partial }{\partial z}$, $ \omega(X_0)=1$ and $\iota_{X_0} \diff \omega =0$.
\end{enumerate}
\end{theorem}

This is further detailed by evaluating the elements $\left(X_i\right)_j$ at some well chosen points. 
Let us denote by $V_1,\dots, V_n$ the 3-dimensional subspaces of $M$ defined by
$$
V_i=\cap_{j\neq i} \left\{ x_{2j-1}=0\right\}\cap \left\{x_{2j}=0\right\}\qquad \forall i\in \ll 1,n\rr.
$$
\begin{theorem}[{\cite[Theorem 6.6]{Gauthier_2001_SR_metrics_and_isoperimetric_problems}}]\label{T:Ag_Gau_Frame}
Let $(M,\Delta,g)$ be a contact sub-Rieman\-nian manifold of dimension $2n+1$ and $q_0\in M$. 
Let $(x_1,\dots x_{2n},z):M\to \R^{2n+1}$ be privileged coordinates at $q_0$, and $(X_1,\dots ,X_{2n})$ be a frame of $(\Delta,g)$, both as in statement of Theorem~\ref{T:Ag_Gau_Frame_simpl}.
Then
\begin{enumerate}[label={ {\normalfont (\roman*)}}]
\item \label{C:Gauthier_frame_hor_coord}For all $i,j\in \ll1,2n\rr$,  
\begin{equation}\label{E:AG_nf_hor1}
\left(X_i\right)_j(0,z)
=
\begin{cases}
	1 \text{ if } i=j,
	\\
	0 \text{ otherwise}
\end{cases}
\end{equation}
and for all $k\in \ll1,2n\rr$ 
\begin{equation}\label{E:AG_nf_hor2}
\partial_{x_k}\left(X_i\right)_j(0,z)
=
0.
\end{equation}
Furthermore, there exist $\beta_1,\dots, \beta_n:\R^3\rightarrow \R$ such that for all $i\in \ll1,n\rr$,  $\beta_i(0,0,z)=0$ and 
\begin{equation}\label{E:AG_nf_hor3}
\begin{aligned}
\left\{
\begin{aligned}
&\left.\left(X_{2i-1}\right)_{2i-1}\right|_{V_i}&= &1+x_{2i}^2\beta_i(x_{2i-1},x_{2i},z),
\\
&\left.\left(X_{2i-1}\right)_{2i}\right|_{V_i}&=&-x_{2i-1}x_{2i}\beta_i(x_{2i-1},x_{2i},z),
\end{aligned}\right.
\\
\left\{
\begin{aligned}
&\left.\left(X_{2i}\right)_{2i-1}\right|_{V_i}&=&-x_{2i-1}x_{2i}\beta_i(x_{2i-1},x_{2i},z),
\\
&\left.\left(X_{2i}\right)_{2i}\right|_{V_i}&=&1+x_{2i-1}^2\beta_i(x_{2i-1},x_{2i},z).
\end{aligned}
\right.
\end{aligned}
\end{equation}

\item \label{C:Gauthier_frame_vert_coord}There exist $\alpha_1,\dots, \alpha_n:\R^3\rightarrow \R$ such that for all $i\in \ll1,n\rr$,
\begin{equation}\label{E:AG_nf_vert1}
\begin{aligned}
\left.\left(X_{2i-1}\right)_{2n+1}\right|_{V_i}=x_{2i}\alpha_i(x_{2i-1},x_{2i},z)/2,
\\
\left.\left(X_{2i}\right)_{2n+1}\right|_{V_i}=-x_{2i-1}\alpha_i(x_{2i-1},x_{2i},z)/2.
\end{aligned}
\end{equation}

\item\label{C:Gauthier_frame_Reeb_coord} We have
$$
\prod_{i=1}^n \alpha_i(0,0,z)=\frac{1}{n!},
$$ 
and for all $i\in \ll1,n\rr$, we denote
$$
\widetilde{L}_i=\frac{\partial (X_{2i})_{2n+1}}{\partial x_{2i-1}}-\frac{\partial (X_{2i-1})_{2n+1}}{\partial x_{2i}}.
$$
Then for all $ i\in \ll 1, n\rr$,
$$
\left.\widetilde{L}_i\right|_{V_i}= \alpha_i, \qquad \forall i\in \ll 1, n\rr,
$$
and 
$$
\sum_{j=1}^{n}   \partial_{x_{2k-1}}\widetilde{L}_j(0,z) \prod_{i\neq j}\alpha_i(0,z) 
=
\sum_{j=1}^{n}   \partial_{x_{2k}}\widetilde{L}_j(0,z) \prod_{i\neq j}\alpha_i(0,z) 
=0.
$$
\end{enumerate}
\end{theorem}

\begin{remark}
A few observations on Theorem~\ref{T:Ag_Gau_Frame}.
\begin{itemize}
\item Notice that points \ref{C:Gauthier_frame_hor_coord}, \ref{C:Gauthier_frame_vert_coord}, \ref{C:Gauthier_frame_Reeb_coord} are respectively consequences of points \ref{C:Gauthier_frame_hor}, \ref{C:Gauthier_frame_vert}, \ref{C:Gauthier_frame_Reeb} of Theorem~\ref{T:Ag_Gau_Frame_simpl}.

\item The nilpotent invariants  $b_1,\dots ,b_n$ at $q_0$ satisfy (up to reordering)
$$
b_i=\alpha_i(0,0,0), \qquad \forall i\in \ll 1, n\rr.
$$

\item In the Agrachev--Gauthier normal form, the frame $(X_1,\dots,X_{2n})$ naturally appears as a perturbation of the frame of a nilpotent contact structure over $\R^{2n+1}$, $\left(\widehat{X}_1,\dots,\widehat{X}_{2n}\right) $, written in the normal form
$$
\widehat{X}_{2i-1}=\partial_{x_{2i-1}} +\frac{ b_i }{2}x_{2i}\partial_z,
\quad
\widehat{X}_{2i}=\partial_{x_{2i}}-\frac{ b_i }{2}x_{2i-1}\partial_z,
\qquad \forall i\in \ll1,n\rr.
$$

\item We can deduce from \ref{C:Gauthier_frame_hor_coord} the following equalities. For all $r,s\in \N$,
\begin{equation}\label{E:AG_nf_hor4}
\begin{aligned}
2\left(\partial_{x_{2i-1}}\right)^r\left(\partial_{x_{2i}}\right)^{s}\beta_i(0,z)
&
=
\left(\partial_{x_{2i-1}}\right)^r\left(\partial_{x_{2i}}\right)^{s+2}\left(X_{2i-1}\right)_{2i-1}(0,z)
\\&=
\left(\partial_{x_{2i-1}}\right)^{r+2}\left(\partial_{x_{2i}}\right)^{s}\left(X_{2i}\right)_{2i}(0,z)
\\&=
-2
\left(\partial_{x_{2i-1}}\right)^{r+1}\left(\partial_{x_{2i}}\right)^{s+1}\left(X_{2i-1}\right)_{2i}(0,z)
\\&=
-2
\left(\partial_{x_{2i-1}}\right)^{r+1}\left(\partial_{x_{2i}}\right)^{s+1}\left(X_{2i}\right)_{2i-1}(0,z).
\end{aligned}
\end{equation}
In particular,
\begin{equation}\label{E:AG_nf_hor5}
\begin{aligned}
0=\beta_i(0,0,z)&
=\left(\partial_{x_{2i}}\right)^{2}\left(X_{2i-1}\right)_{2i-1}(0,z)
\\&=
\left(\partial_{x_{2i-1}}\right)^{2}\left(X_{2i}\right)_{2i}(0,z)
\\&=
-2
\left(\partial_{x_{2i-1}}\right)\left(\partial_{x_{2i}}\right)\left(X_{2i-1}\right)_{2i}(0,z)
\\&=
-2
\left(\partial_{x_{2i-1}}\right)\left(\partial_{x_{2i}}\right)\left(X_{2i}\right)_{2i-1}(0,z).
\end{aligned}
\end{equation}
\end{itemize}
\end{remark}

As an application of these results, we give a proof of the following classical observation. Using notations of Section~\ref{S:conjugate_time}.

\begin{proposition}\label{L:no_singular_near_0_bis}
Let $(M,\Delta,g)$ be a contact sub-Riemannian manifold and $q_0\in M$. For all $\alpha>0$, there exists $R>0$ such that the set of singular points of the exponential at time $1$ in $\C_{q_0}((0,R))$ is a subset of $\{h_0^2>\alpha H\}$.

Equivalently, for all $\bar{h}_0>0$, there exists $\varepsilon >0$ such that all $p\in \C_{q}(1/2)$ with $t_c(p)<\varepsilon$ have $|h_0(p)|>\bar{h}_0$.
\end{proposition}
\begin{proof}
Notice that both statements are equivalent since any $p\in \C_{q}(1/2)$ satisfies $t_c(p)=\sqrt{2 H(t_c(p)p,q_0)}$.

We prove this statement by contradiction. Assume there exist $\alpha>0$ and 
a sequence of singular points for $\sre_{q_0}^1$, $(p_k)_{k\in \N}\in \{H>0\}$, such that $H(p_k,q_0)=\frac{1}{2k^2}$ and  $h_0(p_k)^2\leq \alpha H(p_k,q_0)$.

Then $k p_k =\frac{p_k}{\sqrt{2 H(p_k,q_0)}}\in  \C_{q}(1/2)\cap \{h_0^2\leq  \alpha/2\}$. The sequence $(k p_k)_{k\in \N}$ converges up to extraction and there exist $(k_n)_{n\in \N}\in \N$, $p'_\infty\in \C_{q}(1/2)\cap \{h_0^2 \leq  \alpha/2\}$ such that $k_n p_{k_n}\to p'_\infty$.

Hence there exists a converging sequence $(p_{k_n})_{n\in \N}\in \C_{q}(1/2)\cap \{|h_0| \leq  \alpha'\}$ that admits as conjugate time $t_c(p_{k_n})=1/{k_n}$. Let us prove that this is contradictory with the assumptions on the contact sub-Riemannian structure.

Since the sequence $(p_{k_n})_{n\in \N}$ converges towards $p_\infty'$, we can chose an arbitrarily small neighborhood of $p_\infty'$, $V\subset T_{q_0}^*M$, and assume the sequence $(p_{k_n})_{n\in \N}$ stays in $V$. Then we use the expansion of $q(t)=\sre_{q_0}(t,h_1,\dots ,h_{2n},h_0)$, uniform with respect to $p\in V$,
$$
q(1/k)
=
\sum_{l=1}^3 \frac{q^{(l)}(0)}{k^l l!} +o(1/k^4).
$$
We use the Agrachev--Gauthier normal form to prove that this map cannot be singular for $p\in V$ and $k$ large enough.

Indeed, notice first that the Jacobian of $\dot{q}(0)=\sum_{i=1}^{2n} h_i(0) X_i(q_0)$ is just the diagonal matrix $\mathrm{diag}(1,\dots ,1,0)$.
Furthermore, for all $i\in \ll 1,n\rr$, as a consequence of \eqref{E:AG_nf_hor1}-\eqref{E:AG_nf_vert1},
$$
\begin{aligned}
h_{2i-1} D_{q_0}X_{2i-1} \dot{q}(0)=(0, \dots, 0,2b_i h_{2i} h_{2i-1})
\\
h_{2i} D_{q_0}X_{2i} \dot{q}(0)=(0, \dots, 0,-2b_i h_{2i} h_{2i-1}),
\end{aligned}
$$
hence the last line of the Jacobian of $\ddot{q}(0)$ is empty. Thus the Jacobian matrix has the form
$$
\Jac_{p} q(1/k)=
\frac{1}{k}\mathrm{diag}(1,\dots ,1,0)+
\frac{1}{k^2}
\begin{pmatrix}
*&\cdots&*
\\
\vdots&*&\vdots
\\
*&\cdots&*
\\
0&\cdots&0
\end{pmatrix}
+O\left(\frac{1}{k^3}\right).
$$
Hence if the $(2n+1,2n+1)$-coefficient is not a $o(1/k^3)$, the Jacobian matrix has a non-zero determinant for $k$ large enough.

Then for $i\in\ll1,2n\rr$,  
$$
\partial_{h_0}\partial_t^2(h_i(t)X_i(q(t))_{|t=0}=\partial_{h_0}\dot{h}_i(0)D_{q_0}X_i \cdot h(0) = \left(\bar{J}h(0)\right)_i\left(2 \bar{J}h(0)\right)_i
$$
 and the $(2n+1,2n+1)$-coefficient is $2|J h(0)|^2_2>0$, hence the result.
\end{proof}

\section{Computation of invariants}
\label{A:Computation_invariants}
\subsection{Second order invariants}
For all $l\in \ll1,2n\rr$, let $J_l\in \mathcal{M}_{2n}(\R)$ be the matrix such that 
$$ 
(J_l)_{k,m}=\frac{\partial^2 (X_l)_{2n+1}}{\partial x_k\partial x_m}(q_0)-\frac{\partial^2 (X_k)_{2n+1}}{\partial x_l\partial x_m}(q_0),
\qquad \forall k,l,m\in \ll1,2n\rr,
$$
so that for all $x,y\in \R^{2n}$, the vector $J^{(1)}(x)\,y$ satisfies $(J^{(1)}(x)\, y)_l=J_l  x\cdot y$.

Let $V_{i,j}(\sigma)\in \R^{2n}$ be the vector such that
$$
\left(V_{i,j}(\sigma)\right)_l=
\left(
	\left(
		 \e^{-\sigma \Jbar} - I_{2n}  
	\right)	
 	\Jbar^{-1}
	\;{}^t J_l 
	\,
	 \e^{\sigma \Jbar}
	\right)_{i,j}
+
\left(
	\left(
		 \e^{-\sigma \Jbar} - I_{2n}  
	\right)	
 	\Jbar^{-1}
	\;{}^t J_l 
	\,
	 \e^{\sigma \Jbar}
	\right)_{j,i}.
$$

\begin{lemma}\label{L:invariants_order_1}
For all $i,j,k\in \ll 1,2n\rr$
$$
\kappa^{ij}_k= \varepsilon(i,j)
\int_{0}^{\frac{2\pi}{b_1}}
\int_{0}^{\tau}
\left[\e^{(\tau-\sigma)\Jbar}
V_{i,j}(\sigma)
\right]_k
\diff \sigma
\diff \tau,
$$
where 
$$
\varepsilon(i,j)
=
\left\{
\begin{array}{ll}
1 & \text{ if }i\neq j,
\\
1/2& \text{ if }i= j.
\end{array}
\right.
$$
\end{lemma}
\begin{proof}

From Proposition \ref{P:expansion}, we have to compute for all $i,j,k\in \ll 1,2n\rr$,
$$
\varepsilon(i,j)
\frac{\partial^2x_k^{(2)}}{\partial h_i\partial h_j}\left(\frac{2\pi}{b_1},h\right)
=
\kappa^{ij}_k  
$$
Observe that for all $i,j\in \ll 1,2n\rr$,
$$
\frac{\partial^2x^{(2)}}{\partial h_i\partial h_j}\left(\frac{2\pi}{b_1},h\right)
=
\int_{0}^{\frac{2\pi}{b_1}}
\int_{0}^{\tau}
\e^{(\tau-\sigma)\Jbar}
\left(
J^{(1)}\left(\xhat (\sigma ,e_i)\right) \hhat(\sigma , e_j)
+
J^{(1)}\left(\xhat (\sigma ,e_j)\right) \hhat(\sigma , e_i)
\right)
\diff \sigma
\diff \tau,
$$
where, for all $m\in \ll 1,2n\rr$, $e_m\in\R^{2n}$ is the vector such that $(e_m)_l=1$ if $l=m$ and $(e_m)_l=0$ otherwise.
Using the fact that $(J^{(1)}(x)y)_l= (J_l x)\cdot y$, we have
$$
\begin{aligned}
\left[J^{(1)}\left(\xhat (\sigma ,e_i)\right) \hhat(\sigma , e_j)\right]_l
&=
\left(
	J_l  \Jbar^{-1}
	\left(
		 \e^{\sigma \Jbar} - I_{2n}  
	\right)
	e_i 
\right)
\cdot \e^{\sigma \Jbar} e_j
\\
&=
e_i 
\cdot
\left(
	{}^t
	\left(
		 \e^{\sigma \Jbar} - I_{2n}  
	\right)	
 	{}^t\Jbar^{-1}
	\;{}^t J_l 
	\right)
 \e^{\sigma \Jbar} e_j
 \\
&=
e_i 
\cdot
	\left(
 I_{2n}  -		 \e^{-\sigma \Jbar}
	\right)	
 	\Jbar^{-1}
	\;{}^t J_l 
	\,
 \e^{\sigma \Jbar} e_j
 \\
&=
\left(
	\left(
  I_{2n}  -		 \e^{-\sigma \Jbar}
	\right)	
 	\Jbar^{-1}
	\;{}^t J_l 
	\,
	 \e^{\sigma \Jbar}
	\right)_{i,j}.
\end{aligned}
$$
Hence the result.
\end{proof}
To compute $\kappa^{ij}_k$ we use the following lemma.
\begin{lemma}\label{L:expression_blocks}
For all $ r,s\in \ll 1,n\rr $, for all $M\in \mathcal{M}_{2n}(\R)$, let us define the $(r,s)$ $2\times2$ sub-block of $M$, $\mathsf{B}_{rs}\left[M\right]\in \mathcal{M}_{2}(\R)$ by
$$
\mathsf{B}_{rs}\left[M\right]
=
\begin{pmatrix}
M_{2r-1,2s-1} & M_{2r,2s-1}
\\
M_{2r-1,2s} & M_{2r,2s}
\end{pmatrix}.
$$
For all $\theta\in \R$, let 
$$
R(\theta)=
\begin{pmatrix}
\cos \theta& -\sin \theta
\\
\sin \theta & \cos \theta
\end{pmatrix}
\qquad
S(\theta)=
\begin{pmatrix}
\sin \theta& 1-\cos \theta
\\
\cos \theta-1 & \sin \theta
\end{pmatrix}.
$$
Then
$$
\mathsf{B}_{rs}\left[\left(V(\sigma)\right)_l\right]
=\frac{1}{b_r}S(b_r \sigma)
	\mathsf{B}_{rs}\left[{}^t J_l \right]
	\,
	R(b_s\sigma)
	+
	\frac{1}{b_s}S(b_s \sigma)
	\mathsf{B}_{sr}\left[{}^t J_l \right]
	\,
	R(b_r\sigma).
$$
\end{lemma}
\begin{proof}
Since the matrices $\bar{J}$ and $\e^{\sigma\bar{J}}$ are block-diagonal,
$$
\mathsf{B}_{rs}\left[\left(
		 \e^{-\sigma \Jbar} - I_{2n}  
	\right)	
 	\Jbar^{-1}
	\;{}^t J_l 
	\,
	 \e^{\sigma \Jbar}\right]
=
\mathsf{B}_{rr}
		\left[ 
			\left(
				I_{2n}  -		 \e^{-\sigma \Jbar}
			\right)	
 			\Jbar^{-1}
 		\right]
	\mathsf{B}_{rs}\left[{}^t J_l \right]
	\,
	\mathsf{B}_{ss}\left[\e^{\sigma \Jbar}\right].
$$
Hence the result since 
$$
\mathsf{B}_{rr}
		\left[ 
			\left(
				I_{2n}  -		 \e^{-\sigma \Jbar}
			\right)	
 			\Jbar^{-1}
 		\right]=\frac{1}{b_r}S(b_r \sigma),
\quad
\mathsf{B}_{rr}
		\left[ 
			\left(
\e^{\sigma \Jbar}
			\right)	
 			\Jbar^{-1}
 		\right]=R(b_r \sigma),
\quad
\forall r\in \ll1,n\rr.
$$	
\end{proof}
Some interesting computational properties can be deduced from this result.
\begin{lemma}
Let 
$$
\begin{aligned}
\alpha
&=\frac{\pi}{b_1^3}\left(
\frac{\partial^2 (X_2)_{2n+1}}{\partial x_1\partial x_2}(q_0)-\frac{\partial^2 (X_1)_{2n+1}}{\partial x_2^2}(q_0)\right),
\\
\beta
&=-\frac{\pi}{b_1^3}\left(
\frac{\partial^2 (X_2)_{2n+1}}{\partial x_1^2}(q_0)-\frac{\partial^2 (X_1)_{2n+1}}{\partial x_1\partial x_2}(q_0)\right).
\end{aligned}
$$
Then
$$
\begin{aligned}
&\kappa^{1,1}_1=3\alpha, & &\kappa^{1,1}_2=\beta,
\\
&\kappa^{2,2}_1=\alpha, && \kappa^{2,2}_2=3\beta,
\\
&\kappa^{1,2}_1=2 \beta, && \kappa^{1,2}_2=2 \alpha.
\end{aligned}
$$
\end{lemma}
\begin{lemma}\label{L:invariants_kappa131}
For all $i\in \ll 2,n\rr$, 
$\left(\kappa^{kl}_m\right)_{\begin{smallmatrix}k,m\in\{1,2\} \\ l\in \{2i-1,2i\} \end{smallmatrix}}$ only depend on the family
$$
\left\{
\left(\frac{\partial^2 (X_k)_{2n+1}}{\partial x_l\partial x_m}(q_0)\right)
\mid (k,l,m)\in \{2i-1,2i\}\times \{1,2\}^2 \cup   \{1,2\}^2 \times\{2i-1,2i\}
\right\}.
$$
Let $\zeta_i:\R^{15}\rightarrow \R^8$ be the linear map such that
$$
\zeta_i\left(\left(\frac{\partial^2 (X_k)_{2n+1}}{\partial x_l\partial x_m}(q_0)\right)_{k,l,m\in\{1,2\}\cup\{2i-1,2i\}}\right)
=
\left(\kappa^{kl}_m\right)_{\begin{smallmatrix}k,m\in\{1,2\} \\ l\in \{2i-1,2i\} \end{smallmatrix}}
$$
is of rank $8$ on the complementary of codimension $1$ subset $\mathfrak{S}_3\subset M$, and rank $7$ on  $\mathfrak{S}$.
\end{lemma}
\begin{proof}
The first part of the result is a direct application of Lemma~\ref{L:expression_blocks}.
Let 
$\bar{\zeta}_i$ be the restriction of $\zeta_i$ to 
$$
\left(\frac{\partial^2 (X_k)_{2n+1}}{\partial x_l\partial x_m}(q_0)\right)_{\begin{smallmatrix}k,l\in\{1,2\}\\ m\in\{2i-1,2i\}\end{smallmatrix}}.
$$
Explicit computation of $\zeta_i$ yields that the rank of $\bar{\zeta}_i$ is $8$, except for when
\begin{equation}\label{E:eq_rho}
\begin{aligned}
0=&2 \pi ^2 \rho^5+2 \pi ^2 \rho^4-2 \pi ^2 \rho^3-2 \pi ^2 \rho^2-2 \rho+1
\\
&+\left(-4 \pi  \rho^3+10 \pi  \rho^2+2 \pi  \rho\right) \sin (2 \pi  \rho)+
\left(2 \pi  \rho^3-6 \pi  \rho^2+4 \pi  \rho\right) \sin (4 \pi  \rho)
\\
&+\left(-4 \pi ^2 \rho^5+8 \pi ^2 \rho^4+4 \pi ^2 \rho^3-8 \pi ^2 \rho^2+3 \rho-3\right) \cos (2 \pi  \rho)+
(2-\rho) \cos (4 \pi  \rho)
\end{aligned}
\end{equation}
where $\rho=b_i/b_1<1$. Furthermore, if $\rho$ satisfies \eqref{E:eq_rho}, then the rank of $\bar{\zeta}_i$  is $7$. Hence the existence of  $\mathfrak{S}_3\subset M$, by the existence of a codimension $1$ constraint on the $1$-jet of the sub-Riemannian structure at $q_0$.
\end{proof}

\subsection{Third order invariants}\label{A:third_order}

In this section we compute a more precise approximation of the exponential map in the case of initial covectors of the form $(h_1,h_2,0,\dots,0,\eta^{-1})\in T_{q_0}^*M$.

\begin{lemma}
For all $T,R>0$,
normal extremals with initial covector $(\L \h ,\eta^{-1})$ have the following order 3 terms at time $\eta \tau$, uniformly with respect to $h(0)\in B_R$ and  $\tau\in [0,T]$, as  $\eta\to 0^+$:
$$
\begin{aligned}
&x^{(3)}(\tau,  \L \h)=
\int_{0}^{\tau}
h^{(2)}(\sigma,\L \h)\diff \sigma,
\\
&z^{(3)}(\tau,  \L \h)=
\int_{0}^{\tau}
\left(
h_2^{(1)} 
\xhat_1 
-
h_1^{(1)} 
\xhat_2 
+
\hhat_1 \left(X^{(2)}_1\right)_{2n+1}
+
\hhat_2 \left( X^{(2)}_2\right)_{2n+1}
\right)(\sigma , \L \h)
\diff \sigma,
\end{aligned}
$$
with
$$
\begin{aligned}
h^{(2)}(\tau, \L \h)=
		\int_{0}^{\tau}
		\e^{(\tau-\sigma)\bar{J}} 
		\left[
		\vphantom{x^{(2)}}
		\right.	
		&
		J^{(1)} (x^{(2)} )
		\hhat 
	+
	J^{(1)}\left(\xhat \right)
		h^{(1)} 
	+
	J^{(2)}\left(\xhat \right)
		\hhat
	+
	J_z\left(\zhat \right)
		\hhat 
\\
& 
	+
	Q^{(1)}\left(\xhat ,\hhat \right)
	-
	\left.
		w^{(2)} \bar{J} \hhat 	
	\right](\sigma,  \L\h)
		\diff \sigma 
\end{aligned}
$$
and
$$
Q^{(1)}(x,h)=\sum_{i=1}^{2n}\frac{\partial Q(h)}{\partial x_i} x_i,
$$
$$
J^{(1)}(x)=\sum_{i=1}^{2n}\frac{\partial J}{\partial x_i} x_i
\qquad
J^{(2)}(x)=\sum_{i=1}^{2n}\sum_{j=1}^{2n}\frac{\partial^2 J}{\partial x_i\partial x_j} x_i x_j,
\qquad
J_z(z)=\frac{\partial J}{\partial z} z.
$$
\end{lemma}

\begin{proof}
We have
$$
\frac{\diff q^{(3)}}{\diff \tau}
=
\sum_{i=1}^{2n} \hhat_i(\tau, \L \h)X_i^{(2)}(\xhat(\tau, \L\h))+ h_i^{(1)}(\tau, \L \h)X_i^{(1)}(\xhat(\tau, \L\h))+ h_i^{(2)}(\tau, \L \h)X_i^{(0)}.
$$
Since $\xhat(\tau,\L\h)_i=0$ and $\hhat(\tau,\L \h)=0$ for $3\leq i\leq 2n$, the horizontal part of 
$$
h_i^{(0)}(\tau,\L \h)X_i^{(2)}(\xhat(\tau,\L\h))
$$
 vanishes in the Agrachev--Gauthier frame. The same goes for the horizontal part of $X_i^{(1)}$, $1\leq i\leq 2n$. Thus
$$
\begin{aligned}
&\frac{\diff x^{(3)}}{\diff \tau}= h^{(2)}(\tau,\L \h)
\\
&\frac{\diff z^{(3)}}{\diff \tau}= \sum_{i=1}^{2n}
\left[h^{(1)}\left(X_i^{(1)}\right)_{2n+1}
+\hhat \left(X_i^{(2)}\right)_{2n+1}\right](\tau,\L\h).
\end{aligned}
$$
Regarding $h^{(2)}$, we get the result by computing the order 2 in $\eta$ of $\frac{\diff h}{\diff \tau}$.
We have 
$$
\frac{\diff h}{\diff \tau}
		= \frac{\eta}{w} J h+\eta Q(h)
$$
with
$$
\frac{1}{w}=(1+\eta^2 w^{(2)}+O(\eta^3))^{-1}=1-\eta^2 w^{(2)}+O(\eta^3),
$$
$$
Q(h)=\eta Q^{(1)}(x^{(1)},h^{(0)})+O(\eta^2),
$$
$$
J =\bar{J}+\eta  J^{(1)}\left(x^{(1)}\right) 
+\eta^2\left(
J^{(1)} ( x^{(2)} )
	+
	J^{(2)} ( x^{(1)} ) 
	+
	J_z ( z^{(2)} )
\right)+O(\eta^3).
$$
Then evaluated at $(\tau,\L\h)$, we have
$$
\begin{aligned}
\frac{\diff h^{(2)}}{\diff \tau}
		=& \bar{J} h^{(2)}
		+
		J^{(1)}\left(x^{(2)} \right)
		\hhat
	+
	J^{(1)}\left(\xhat \right)
		h^{(1)} 
	\\&+
	J^{(2)}\left(\xhat \right)
		\hhat 
	+
	J_z\left(\zhat \right)
		\hhat 
	+
	Q^{(1)}\left(\xhat ,\hhat \right)
		\hhat 	
		-
		w^{(2)} \bar{J} \hhat 		.
\end{aligned}
$$
Hence the result.
\end{proof}

We can immediately apply this result to give an expression of $z^{(3)}$, using only the second order invariants introduced in the previous sections.

\begin{lemma}\label{L:z3}
Using the prior notations, we have
$$
z^{(3)}\left(\frac{2\pi}{b_1},\L\h\right)
=
\frac{1}{2}  \left(\h_1^2+\h_2^2\right) (\alpha \h_1+\beta \h_2).
$$
\end{lemma}

\begin{proof}
As stated before, it is a matter of evaluating the terms for the Agrachev--Gauthier frame. We have
$$
\frac{\diff z^{(3)}}{\diff \tau}= \sum_{i=1}^{2n}
\left[h_i^{(1)}\left(X_i^{(1)}\right)_{2n+1}
+\hhat_i \left(X_i^{(2)}\right)_{2n+1}\right](\tau,\L\h).
$$
For $3\leq i\leq 2n$, $\left(X_i^{(1)}\right)_{2n+1}\left(\xhat\left(\tau,\L\h\right)\right)=0$, 
$$
\left(X_1^{(1)}\right)_{2n+1}\left(\xhat\left(\tau,\L\h\right)\right)=\frac{b_1}{2} \xhat_2
\qquad\text{ and }\qquad
\left(X_2^{(1)}\right)_{2n+1}\left(\xhat\left(\tau,\L\h\right)\right)=-\frac{b_1}{2} \xhat_1.
$$
We have
$$
h^{(1)}(\tau, \L \h)=\int_{0}^{\tau}
\e^{(\tau-\sigma)\bar{J}} J^{(1)}\left(\xhat(\sigma,\L \h)\right)\hhat(\sigma,\L \h)\diff \sigma,
$$
with
$$
\left[J^{(1)}\left(\xhat(\tau,\L \h)\right)\right]_{12}=\xhat_1\frac{\partial c_{21}^0}{\partial x_1}+\xhat_2\frac{\partial c_{21}^0}{\partial x_2}.
$$
Since 
$\frac{\partial c_{21}^0}{\partial x_1}=\frac{b_1^3}{\pi}\beta$, and 
$\frac{\partial c_{21}^0}{\partial x_2}=-\frac{b_1^3}{\pi}\alpha$, with $h(0)=\L \h$, 
$$
J^{(1)}\left(\xhat \right)\hhat 
=
\frac{b_1^3}{\pi}
\begin{pmatrix}
\hhat_2 (\beta \xhat_1-\alpha \xhat_2)
\\
-\hhat_1 (\beta \xhat_1-\alpha \xhat_2)
\\
0
\\
\vdots
\\
0
\end{pmatrix}.
$$
Similarly, we have
$$
\begin{aligned}
\left(X_1^{(2)}\right)_{2n+1}=\frac{b_1^3}{2\pi} 
(
 -\beta \xhat_1 \xhat_2+\alpha  \xhat_2^2/2
),
\\
\left(X_2^{(2)}\right)_{2n+1}= \frac{b_1^3}{2\pi} (\beta \xhat_1 ^2/2-\alpha  \xhat_1\xhat_2).
\end{aligned}
$$
We then obtain obtain the result by integration.
\end{proof}

Since we are only interested in the first two coordinates of the exponential map, we state the following result.

\begin{lemma} 
For all $\tau$, for all $\h\in \R^{2n}$,
$$
Q^{(1)}_{1}\left(\xhat(\tau,\L \h),\hhat(\tau,\L\h)\right)=Q^{(1)}_{2} \left(\xhat(\tau,\L \h),\hhat(\tau,\L\h)\right)=0.
$$
\end{lemma}

\begin{proof}
Recall that $Q:\R^{2n}\rightarrow  \R^{2n}$  is the map such that 
$$
Q_i (h_1,\dots h_{2n})=\sum_{j=1}^{2n}\sum_{k=1}^{2n} c_{j i}^k  h_j h_k, \quad  \quad i\in \ll 1, 2n\rr.
$$
Since $h(\tau)=\hhat(\tau,\L \h)+O(\eta)$ and $x(\tau)=\eta \xhat(\tau,\L \h)+O(\eta^2)$,
$$
\begin{aligned}
&Q_1 (h)=  c_{21}^1(\xhat)  \hhat_1 \hhat_2+ c_{2 1}^2(\xhat) \hhat_2^2+O(\eta^2),
\\
&Q_2 (h)= c_{12}^1(\xhat) \hhat_1^2+ c_{1 2}^2(\xhat)  \hhat_1 \hhat_2+O(\eta^2).
\end{aligned}
$$
Recall that for all $i,j\in \ll 1,  2n\rr$
$$
\frac{\partial c_{12}^j}{\partial x_i}=\frac{\partial {(X_2)}_j}{\partial x_i\partial x_1}-\frac{\partial {(X_1)}_j}{\partial x_i\partial x_2},
$$
and thus in the Agrachev--Gauthier frame, evaluated at $q_0$,
$$
\frac{\partial c_{12}^1}{\partial x_1}=\frac{\partial c_{12}^1}{\partial x_2}=\frac{\partial c_{12}^2}{\partial x_1}=\frac{\partial c_{12}^2}{\partial x_2}=0.
$$ 
Hence
$$
Q^{(1)}_{1}\left(\xhat(\tau,\L \h),\hhat(\tau,\L\h)\right)=Q^{(1)}_{2} \left(\xhat(\tau,\L \h),\hhat(\tau,\L\h)\right)=0.
$$
\end{proof}

Let us introduce the invariant $\xi\in \R$, given in the Agrachev--Gauthier frame by the formula 
$$
\xi=\frac{\pi}{ b_1^3}\frac{\partial^2 X_1}{\partial z\partial x_2}(q_0).
$$
This invariant, which is $0$ in the 3-dimensional contact case,
naturally appears in some terms of the third order expansion of the exponential map.
\begin{lemma} We have
$$
w^{(2)}(\tau,\L\h)= - \frac{2 b_1^2\xi }{\pi} \zhat(\tau,\L\h)
$$
and
$$
J_z(\zhat(\tau,\L\h))\hhat(\tau,\L\h)=- \frac{2 b_1^2\xi }{\pi} \zhat(\tau,\L\h)\bar{J}\L \hhat (\tau).
$$
\end{lemma}

\begin{proof}
As seen in the proof of Proposition~\ref{P:expansion}, $\dfrac{\diff w}{\diff \tau}=-\eta w L h-\eta^2 w^{2}Q_0(h)=O(\eta^2)$.
Then 
$$
\dfrac{\diff w^{(2)}}{\diff  \tau}=
-w^{(1)}L^{(0)}h^{(0)}
-w^{(0)}L^{(1)}h^{(0)}
-w^{(0)}L^{(0)}h^{(1)}-Q_0^{(0)}\left(h^{(0)}\right).
$$

In the Agrachev--Gauthier frame, $c_{i0}^j(q_0)=-\partial_z (X_i)_j$, for all $i,j\in \ll1,2n\rr$.  Hence $c_{i0}^j(q_0)=0$, which implies $Q_0^{(0)}=0$. Likewise, $c_{i0}^0(q_0)=-\partial_z (X_i)_{2n+1}$ for all $i\in \ll1,2n\rr$, hence $c_{i0}^0(q_0)=0$ and $L^{(0)}=0$.

With $h(\tau)=\hhat(\tau,\L \h)+O(\eta)$ and $x(\tau)=\eta\xhat(\tau,\L h)+O(\eta^2)$, we then have 
$$
\frac{\diff w^{(2)}}{\diff  \tau}
= \left(\frac{\partial c_{1 0}^0}{\partial  x_1}\xhat_1+\frac{\partial c_{1 0}^0}{\partial x_2}\xhat_2\right) \hhat_1+\left(\frac{\partial c_{2 0}^0}{\partial  x_1}\xhat_1+\frac{\partial c_{2 0}^0}{\partial x_2}\xhat_2\right) \hhat_2.
$$
Again in the Agrachev--Gauthier frame, at $q_0$, 
$$
\frac{\partial c_{1 0}^0}{\partial  x_1}=\frac{\partial c_{2 0}^0}{\partial  x_2}=0,
\qquad 
\text{ and }
\qquad
\frac{\partial c_{1 0}^0}{\partial  x_2}=-\frac{\partial c_{2 0}^0}{\partial  x_1}=-\frac{1}{2}\frac{\partial b_1}{\partial  z}=-\frac{b_1^3 \xi}{\pi}.
$$
As a result
$$
\frac{\diff w^{(2)}}{\diff  \tau}
=-\frac{b_1^3\xi}{\pi } \left(\xhat_2\hhat_1-\xhat_1  \hhat_2\right),
$$
hence the result by recognizing $\frac{\diff \zhat}{\diff \tau}$ and $w^{(2)}(0)=0$.

The same reasoning applies for $J_z$, where $(J_z)_{1,2}=-\frac{\partial c_{21}^0}{\partial z}=-\frac{2 b_1^3}{\pi}\xi$.
\end{proof}

We now know enough to compute $x^{(3)}(2\pi/b_1,\L\h)$ (or at least its first two coordinates).
By direct integration we have the following expression for the terms of the expansion that depend on $\xi$.
\begin{lemma}
Let
$$
x_{w^{(2)}}=
\int_{0}^{2\pi/b_1}
\int_{0}^{\tau}
		\e^{(\tau-\sigma)\bar{J}} 
		\left[
		-w^{(2)} \bar{J}
		\hhat 
		\right](\sigma,\L\h)
	\diff \sigma
	\diff \tau
$$
and
$$
x_{J_z}=
\int_{0}^{2\pi/b_1}
\int_{0}^{\tau}
		\e^{(\tau-\sigma)\bar{J}} 
		\left[
		J_z\left(\zhat \right)
		\hhat 
		\right](\sigma,\L\h)
	\diff \sigma
	\diff \tau.	
$$
Then $x_{w^{(2)}}=-x_{J_z}$.
\end{lemma}

We use the same method to compute the other terms of the expansion.
Let 
$$
\chi_{11}=-\frac{b_1^4}{\pi}\frac{\partial^3 X_1}{\partial x_1^2\partial x_2 },
\qquad
\chi_{12}=\frac{2b_1^4}{\pi}\frac{\partial^3 X_1}{\partial x_1\partial x_2^2 },
\qquad
\chi_{22}=-\frac{b_1^4}{\pi}\frac{\partial^3 X_1}{\partial x_2^3}.
$$

\begin{lemma}
Let
$$
x_{J^{(2)}}=
\int_{0}^{2\pi/b_1}
\int_{0}^{\tau}
		\e^{(\tau-\sigma)\bar{J}} 
		\left[
		J^{(2)}\left(\xhat \right)
		\hhat
		\right](\sigma,\L\h)
	\diff \sigma
	\diff \tau.
$$
We have 
$$
\begin{aligned}
\left(x_{J^{(2)}}\right)_1
=
	\left(\chi _{11}+5 \chi _{22}\right) \h_1^3
	+
	3 \chi _{12} \h_2 \h_1^2
	+
	3 \left(\chi _{11}+\chi _{22}\right) \h_2^2 \h_1
	+
	\chi _{12} \h_2^3,
\\
\left(x_{J^{(2)}}\right)_2
=
	\left(5 \chi _{11}+\chi _{22}\right) \h_2^3
	+
	3 \chi_{12} \h_2^2 \h_1
	+
	3 \left(\chi _{11}+\chi _{22}\right) \h_1^2 \h_2
	+
	\chi_{12}  \h_1^3.
\end{aligned}
$$
\end{lemma}

\begin{proof}
First notice that
$$
{J^{(2)}\left(\xhat(\tau, \L \h) \right)}_{1,2}
=
-{J^{(2)}\left(\xhat(\tau, \L \h) \right)}_{2,1}=
\frac{\pi}{b_1^4}
\left(
-\chi_{11} \,\xhat_1^2+\chi_{12}\, \xhat_1 \xhat_2 -\chi_{22}\, \xhat_2 ^2
\right)(\tau, \L \h).
$$
The stated result is obtained by direct integration.
\end{proof}

\begin{lemma}
Let 
$$
x_{J^{(1)}}=
\int_{0}^{2\pi/b_1}
\int_{0}^{\tau}
		\e^{(\tau-\sigma)\bar{J}} 
		\left[
		J^{(1)} \left(x^{(2)} \right)
		\hhat 
		+
		J^{(1)}\left(\xhat \right)
		h^{(1)} 
		\right](\sigma,\L\h)
	\diff \sigma
	\diff \tau.
$$
We have 
$$
\begin{aligned}
\left(x_{J^{(1)}}(\tau, \L \h)\right)_1
=&
\frac{1}{2b_1^2} 
\left[
-\h_1^3 \left(15 \alpha ^2+3 \beta ^2\right)
+
\h_1^2 \h_2 \left(4 \pi  \alpha ^2-18 \alpha  \beta \right)
\right.
\\
&
\left.
\qquad\qquad
-\h_1 \h_2^2 \left(9 \alpha ^2-8 \pi  \alpha  \beta +9 \beta ^2\right)
+
\h_2^3 \left( 4 \pi  \beta ^2-6 \alpha  \beta\right)
\right],
\\
\left(x_{J^{(1)}}(\tau, \L \h)\right)_2
=&
-\frac{1}{2b_1^2} 
\left[
\h_1^3 \left(4 \pi  \alpha ^2+6 \alpha  \beta \right)
 +
 \h_1^2 \h_2 \left(9 \alpha ^2+8 \pi  \alpha  \beta +9 \beta ^2\right)
 \right.
 \\
 &\qquad\qquad+
 \left.
 \h_1 \h_2^2 \left(4 \pi  \beta ^2+18 \alpha  \beta \right)
 +
 \h_2^3 \left(3 \alpha ^2+15 \beta ^2\right)
 \right].
\end{aligned}
$$
\end{lemma}

\begin{proof}
Let $\tau \in \R$, $h\in \R^{2n}$. Evaluated at $(\tau, \L \h) $, we have
$$
\left(J^{(1)}\left(x^{(2)} \right)\hhat \right)_1
=
\hhat_2 (\beta x^{(2)}_1 -\alpha x^{(2)}_2 ),
\qquad
\left(J^{(1)}\left(x^{(2)} \right)\hhat \right)_2
=
-\hhat_1 (\beta x^{(2)}_1 -\alpha x^{(2)}_2 )
$$
and
$$
\left(J^{(1)}\left(\xhat\right)h^{(1)}\right)_1
=
h^{(1)}_2 (\beta \xhat_1 -\alpha \xhat_2),
\qquad
\left(J^{(1)}\left(\xhat\right)h^{(1)}\right)_2
=
-h^{(1)}_1 (\beta \xhat_1 -\alpha \xhat_2 ).
$$
Both $h^{(1)}$ and $x^{(2)}$ have been computed before and we have the stated result by integration.
\end{proof}

Summing up, we have proven the following.

\begin{proposition}
We have 
$
\left[
	x^{(3)}\left(\frac{2\pi}{b_1},\L\h\right)
\right]_{1,2}
= 
\left[
	x_{J^{(1)}}+x_{J^{(2)}}
\right]_{1,2}
$.
Explicitly, this yields
$$
\begin{aligned}
\left[x^{(3)}\left(\frac{2\pi}{b_1},\L\h\right)\right]_1=
&
\h_1^3 \left(\frac{3 }{2 b_1^2 } \left(5 \alpha ^2+\beta ^2\right)
+\chi_{11}+5 \chi_{22}\right)\\
&+
\h_1^2 \h_2 \left(\frac{\alpha }{b_1^2} (2 \pi  \alpha -9 \beta )
+3 \chi_{12}\right)
\\
&+
\h_1 \h_2^2 \left(-\frac{1}{2 b_1^2 } \left(9 \alpha ^2-8 \pi  \alpha  \beta +9 \beta ^2\right)
+3 (\chi_{11}+\chi_{22})\right)
\\
&+
\h_2^3 \left(-\frac{\beta  }{b_1^2 } (2 \pi  \beta-3 \alpha  )
+\chi_{12}\right),
\end{aligned}
$$
$$
\begin{aligned}
\left[x^{(3)}\left(\frac{2\pi}{b_1},\L\h\right)\right]_2=
&
\h_1^3 \left(-\frac{\alpha  }{b_1^2} (2 \pi  \alpha +3 \beta )
+\chi_{12}\right)
\\
&+
\h_1^2 \h_2 \left(-\frac{1}{2 b_1^2 } \left(9 \alpha ^2+8 \pi  \alpha  \beta +9 \beta ^2\right)
+3 (\chi_{11}+\chi_{22})\right)
\\
&+
\h_1 \h_2^2 \left(-\frac{\beta }{ b_1^2  } (2 \pi  \beta +9 \alpha )
+3 \chi_{12}\right)
\\
&+
\h_2^3 \left(-\frac{3 }{2 b_1^2 } \left(\alpha ^2+5 \beta ^2\right)
+5 \chi_{11}+\chi_{22}\right).
\end{aligned}
$$
\end{proposition}

\section{Computational lemmas}\label{A:computational_lemmas}

\subsection{Determinant formulas}
In this section we prove some computational results useful in multiple proofs. Let $n\in \N$, $n>1$, and $b_1,\dots ,b_n\in \R$ be  such that $0<b_i<b_1 $ for all $i\in \ll 2, n\rr$. 

Let
$A\in \mathcal{M}_{2n-2}(\R)$ be the block-diagonal square matrix 
$$
\begin{pmatrix}
\dfrac{1}{b_2}
\begin{pmatrix}
\sin (\frac{2 b_2 \pi}{b_1} ) &  1-\cos (\frac{2 b_2 \pi}{b_1})
 \\
\cos (\frac{2 b_2 \pi}{b_1}) -1 &  \sin (\frac{2 b_2 \pi}{b_1}) 
\end{pmatrix}
 &&(0)
 \\
&\ddots&
\\
(0)&&\dfrac{1}{b_n}
\begin{pmatrix}
\sin (\frac{2 b_n \pi}{b_1}) &  1-\cos (\frac{2 b_n \pi}{b_1})
 \\
 \cos (\frac{2 b_n \pi}{b_1})-1  &  \sin (\frac{2 b_n \pi}{b_1}) 
\end{pmatrix}
\end{pmatrix}.
$$

\begin{lemma}\label{L:detA}
We have
$$
\det(A)=2^{2n-2}
\prod_{i=2}^{n}
\frac{
1
}{b_i^2}
 \sin^2\left(
\frac{\pi  b_i}{
b_1
}
\right)>0.
$$
\end{lemma}
\begin{proof}
This is a consequence of 
$$
\begin{vmatrix}
\sin (\frac{2 b_i \pi}{b_1} ) &  1-\cos (\frac{2 b_i \pi}{b_1})
 \\
\cos (\frac{2 b_i \pi}{b_1})-1  &  \sin (\frac{2 b_i \pi}{b_1}) 
\end{vmatrix}
=4\sin^2 
\left(
	\frac{\pi  b_i}{b_1}
\right)\qquad \forall i\in \ll 2, n\rr.
$$
Since $0<b_i<b_1$ for all $i\in \ll 2, n\rr$, we have the stated sign.
\end{proof}

\begin{lemma}\label{L:det(AW)}
Let $V,W\in \mathcal{M}_{1\times 2n-2}(\R)$, $v\in \R$. Then 
\begin{multline}
\frac{1}{\det(A)}
\left|
\begin{array}{c|c}	
	A	
	&
W
	\\
	\hline
		 ^{t}V
		 &
		 v
\end{array}
\right|=\\
v+	
\frac{1}{2}
\sum_{i=2}^{n}
b_i\left(
	V_{2i-1}W_{2i}-V_{2i}W_{2i-1}-\left(V_{2i-1}W_{2i-1}+V_{2i}W_{2i}\right)\cot \frac{b_{i}\pi}{b_1} 
	\right).
\end{multline}
\end{lemma}

\begin{proof}
To prove this result, we develop along the last column the determinant of 
$$
\left(
\begin{array}{c|c}	
	A	
	&
W
	\\
	\hline
		 ^{t}V
		 &
		 v
\end{array}
\right).
$$
We get
\begin{multline*}
\frac{1}{\det(A)}
\left|
\begin{array}{c|c}	
	A	
	&
W
	\\
	\hline
		 ^{t}V
		 &
		 v
\end{array}
\right|
=
v+
\\
 \sum_{i=2}^n
\frac{
b_i^2}{4\sin^2\left(
\frac{\pi  b_i}{
b_1
}\right)}
\left(
	\frac{W_{2i-1}}{b_i}
	\begin{vmatrix}
\cos (\frac{2 b_i \pi}{b_1})-1  &  \sin (\frac{2 b_i \pi}{b_1}) 
 \\
 V_{2i-1} & V_{2i}
\end{vmatrix}
-	
\frac{W_{2i}}{b_i}
	\begin{vmatrix}
\sin (\frac{2 b_i \pi}{b_1} ) &  1-\cos (\frac{2 b_i \pi}{b_1})
 \\
 V_{2i-1} & V_{2i}
\end{vmatrix}
\right).
\end{multline*}
Hence the result by trigonometric identification.
\end{proof}

\subsection{Conjugate time equations}\label{A:Conjugate_time_equations}

\subsubsection{Proof of Lemma~\ref{L:lim_r1_to_0}}\label{A:proof_lemma_psi}

\begin{proof}[Proof of Lemma~\ref{L:lim_r1_to_0}]
Let $T=\min(2\pi/b_2,4\pi/b_1)$, so that $(2\pi/b_1, T)$ is a connected component of $\R^+\setminus Z$.
For all $ i\in \ll1,  n\rr$, let
$$
\begin{array}{rccc}
\psi_i: 
&
\R\setminus \cup_{k\in \N} \{2k \pi/b_i\} &\longrightarrow & \R
\\
&
\tau
&\longmapsto 
&3 \tau
-  b_i \tau ^2\, \frac{ \cos (b_i \tau /2)}{\sin (b_i \tau /2) }-\frac{\sin (b_i \tau )}{ b_i}.
\end{array}
$$
For all $ i\in \ll1,  n\rr$, $\psi_i$ is smooth and has a positive derivative over $\R\setminus \cup_{k\in \N} \{2k \pi/b_i\}$. Moreover $\psi_i(0)=0$, and for all $k\in \N$, $k>0$, 
$$
\lim\limits_{t\to {2k \pi/b_i}^+}\psi_i(t)=-\infty\quad \text{ and  }\quad\lim\limits_{t\to {2k \pi/b_i}^-}\psi_i(t)=+\infty.$$
This immediately implies that $ \tau_1(r)>2\pi/b_1$. Furthermore, since 
$$
\psi(\tau,r)=\sum_{i=1}^{n}r_i^2 \psi_i(\tau),\qquad  \forall r\in {(\R^{+})}^{n},
$$
both $\lim\limits_{t\to {2 \pi/b_1}^+}\psi(\tau,r)=-\infty$ and
$\lim\limits_{t\to {T}^-}\psi(\tau,r)=+\infty$,
and $\psi(\cdot,r)$ vanishes exactly once on $(2\pi/b_1, T)$, at time $\tau_1(r)$.
Since for all $ i\in \ll2,n\rr $, $\psi_i>0$ on $(2\pi/b_1, T)$, we have that
$$
\psi_1(\tau_1(r))=-\frac{1}{r_1^2}\sum_{i=2}^{n}r_i^2 \psi_i(\tau_1(r))<0.
$$

This equality implies that $r_1\mapsto \tau_1(r)$ is an increasing function. Indeed let
$r,r'\in  {(\R^{+})}^{n}$ be such that $r_1<r'_1$ and $r_i=r'_i$ for  all $i\in \ll 2,  n\rr$, then for all $\tau \in (2\pi/b_1, T)$,
$$
-\frac{1}{r_1^2}\sum_{i=2}^{n}r_i^2 \psi_i(\tau)<- \frac{1}{{r'_1}^2}\sum_{i=2}^{n}{r'_i}^2 \psi_i (\tau)<0.
$$
Since $\tau\mapsto-\frac{1}{r_1^2}\sum_{i=2}^{n}r_i^2 \psi_i$ and $\tau\mapsto- \frac{1}{{r'_1}^2}\sum_{i=2}^{n}{r'_i}^2 \psi_i$
are both decreasing functions over $(2\pi/b_1, T)$, 
since $\psi_1$ is an increasing function over $(2\pi/b_1, T)$, this implies 
$\tau_1(r)<\tau_1(r')$.

In particular, $\tau_1$ being continuous, it  converges towards a limit $l(r_2,\dots r_n)\in [2\pi/b_1,T)$ as $r_1\to 0^+$, and
$$
\lim_{r_1\to 0^+} \sum_{i=2}^{n}{r_i}^2 \psi_i (\tau_1(r)) = \sum_{i=2}^{n}{r_i}^2 \psi_i (l(r_2,\dots r_n)) >0.
$$
Hence $  \lim_{r_1\to 0^+}\psi_{1}( \tau_1(r_1,\dots, r_n))=-\infty$, and by inverting $\psi_1$ we obtain 
$$
\lim\limits_{r_1\to 0^+} \tau_1(r_1,\dots, r_n)=2\pi/b_1.$$

Notice in particular that as $\delta t\rightarrow 0^+$, 
$
\psi_1(2\pi/b_1+\delta t)
\sim 
-\frac{ 8 \pi^2}{b_1^2 \delta t }.
$
Hence we get by inverting $\psi_1$
$$
\psi_1^{-1}\left(-\frac{1}{r_1^2}\sum_{i=2}^{n}r_i^2 \psi_i(\tau_1(r)))\right)-2\pi/b_1
\sim
\frac{8\pi^2}{b_1^2\sum_{i=2}^{n}r_i^2 \psi_i(2\pi/b_1))}r_1^2,
$$
hence  expansion~\eqref{E:equiv_tau_1}.
\end{proof}

\subsubsection{On the first domain}\label{A:On_the_first_domain}

\begin{lemma}\label{P:mini_det_1}
We have
$$
\Phi
\left(
2\pi /b_1+\eta \delta t,h,\eta
\right)
=
\eta^{2n+3}
K'
d
+O(\eta^{2n+4}),
$$
where
$$
K'=2^{2n-2}
\prod_{i=2}^{n}
\frac{
1
}{b_i^2}
 \sin^2\left(
\frac{\pi  b_i}{
b_1
}
\right)>0
$$
and
$$
d=\left|
\begin{matrix}
			\frac{\partial}{\partial h_1}
			\left(
			 F^{(2)}
			 +
			 \delta t
			 \frac{\partial F^{(1)}}{\partial \tau}
			 \right)_1
		 &
			 \frac{\partial}{\partial h_2}
			\left(
			 F^{(2)}
			 +
			 \delta t
			 \frac{\partial F^{(1)}}{\partial \tau}
			 \right)_1
		 &
		 -\tau\left(\frac{\partial F^{(1)}}{\partial \tau}\right)_{1}
	\\
			\frac{\partial}{\partial h_1}
			\left(
			 F^{(2)}
			 +
			 \delta t
			 \frac{\partial F^{(1)}}{\partial \tau}
			 \right)_2
		 &
			 \frac{\partial}{\partial h_2}
			\left(
			 F^{(2)}
			 +
			 \delta t
			 \frac{\partial F^{(1)}}{\partial \tau}
			 \right)_2
		 &
		 -\tau\left(\frac{\partial F^{(1)}}{\partial \tau}\right)_{2}
	\\
			\frac{\partial}{\partial h_1}
			\left(
			 F^{(2)}
			 \right)_{2n+1}
		 &
			 \frac{\partial}{\partial h_2}
			\left(
			 F^{(2)}
			 \right)_{2n+1}
		 &
		 0
\end{matrix}
\right|_{\tau=2\pi/b_1}.
$$
\end{lemma}

\begin{proof}

From Proposition~\ref{P:expansion}, we have that $F^{(1)}\left(\frac{2\pi}{b_1},h\right)=\left(\xhat\left(\frac{2\pi}{b_1},h\right),0\right)$, with $\xhat_1\left(\frac{2\pi}{b_1},h\right)=\xhat_2\left(\frac{2\pi}{b_1},h\right)=0$. 
Furthermore, observe that 
$$
\begin{pmatrix}
	\xhat_3(2\pi/b_1,h)
	\\
	\vdots
	\\
	\xhat_{2n}(2\pi/b_1,h)
	\end{pmatrix}
	=A
\begin{pmatrix}
	h_3
	\\
	\vdots
	\\
	h_{2n}
	\end{pmatrix}	
$$
where $A\in \mathcal{M}_{2n-2}(\R)$ is the block-diagonal matrix $\mathrm{diag}(A_2,\dots,A_n)$ of $2\times 2$ blocks
$$
A_i=
\dfrac{1}{b_i}
\begin{pmatrix}
\sin (\frac{2 b_i \pi}{b_1} ) &  1-\cos (\frac{2 b_i \pi}{b_1}) 
 \\
 \cos (\frac{2 b_i \pi}{b_1})-  &  \sin (\frac{2 b_i \pi}{b_1}) 
\end{pmatrix},
\qquad \forall i\in \ll2,n\rr.
$$

Thus  Equation~\eqref{E:order_1_F_delta_t} entails, by factorizing $\eta$,
\begin{multline*}
\Phi
\left(
2\pi /b_1+\eta \delta t,h,\eta
\right)
=\\\eta^{2n+3}
\left|
\begin{array}{cc|c|c}
			\frac{\partial}{\partial h_1}
			\left(
			 F^{(2)}
			 +
			 \delta t
			 \frac{\partial F^{(1)}}{\partial \tau}
			 \right)_1
		 &
			 \frac{\partial}{\partial h_2}
			\left(
			 F^{(2)}
			 +
			 \delta t
			 \frac{\partial F^{(1)}}{\partial \tau}
			 \right)_1
		 &
		 0\cdots 0
		 &
		 -\frac{2\pi}{b_1}\left(\frac{\partial F^{(1)}}{\partial \tau}\right)_{1}
	\\
			\frac{\partial}{\partial h_1}
			\left(
			 F^{(2)}
			 +
			 \delta t
			 \frac{\partial F^{(1)}}{\partial \tau}
			 \right)_2
		 &
			 \frac{\partial}{\partial h_2}
			\left(
			 F^{(2)}
			 +
			 \delta t
			 \frac{\partial F^{(1)}}{\partial \tau}
			 \right)_2
		 &
		 0 \cdots 0
		 &
		 -\frac{2\pi}{b_1}\left(\frac{\partial F^{(1)}}{\partial \tau}\right)_{2}
	\\
	\hline
		\:\:
	\begin{matrix}
	0
	\\
	\vdots
	\\
	0	
	\end{matrix}
	&
		\:\:
	\begin{matrix}
	0
	\\
	\vdots
	\\
	0	
	\end{matrix}
	&
	A
	&
		\:\:
	\begin{matrix}
	0
	\\
	\vdots
	\\
	0	
	\end{matrix}
	\\
	\hline
			\frac{\partial}{\partial h_1}
			\left(
			 F^{(2)}
			 \right)_{2n+1}
		 &
			 \frac{\partial}{\partial h_2}
			\left(
			 F^{(2)}
			 \right)_{2n+1}
		  &
		 0\cdots  0
		 &
		 0
\end{array}
\right|
+O(\eta^{2n+4}).
\end{multline*}
From Lemma~\ref{L:detA} in Appendix~\ref{A:computational_lemmas}, $\det(A)=K'$ and we have the stated result.
\end{proof}

\subsubsection{On the second domain}\label{A:second_domain}
To evaluate 
$$
\Phi
\left(
2\pi /b_1+\sqrt{\eta}\delta t ,  \sqrt{\eta}\L\h+ (I_{2n}-\L)\h  ,\eta
\right)
,$$ with  $\delta t \in \R$, $\bar{h}\in \R^{2n}$, notice that 
$$
\begin{aligned}
\frac{\partial F}{\partial{h_i}}=\frac{1}{\sqrt{\eta}} \frac{\partial G}{\partial{\h_i}},\quad  \forall i\in \ll 1,2 \rr,
\quad \text{ and } \quad 
\frac{\partial F}{\partial{h_i}}=\frac{\partial G}{\partial{\h_i}}
\quad 
\forall i \in \ll 3, 2n\rr.
\end{aligned}
$$
Then for all  $i\in \ll 1,2n\rr$, we set 
$C_i=\dfrac{\partial G}{\partial{\h_i}}$ and $C_{2n+1}=\eta \dfrac{\partial G}{\partial \eta}-\tau \dfrac{\partial G}{\partial \tau}$, evaluated at time $\tau=2\pi /b_1+\sqrt{\eta}\delta t$.
For all $i\in \ll 1,2n+1\rr$, the vector $C_i\in \R^{2n+1}$ also admits a power series  expansion in $\sqrt{\eta}$,
$$
C_i=\sum_{k=0}^{\infty}\eta^{k/2} C_i^{(k/2)}.
$$
Notice that by definition of $(C_i)_{ i\in \ll 1,2n+1\rr}$ we have 
$
C_i^{(0)}=C_i^{(1/2)}=C_i^{(1)}=0 
$
for all $i\in \ll 1, 2n \rr$.
As a consequence we can obtain an equation satisfied by a potential perturbation of order $1/2$ of the nilpotent conjugate time.
\begin{lemma}\label{L:deltat1/2}
Recall
$$
K= 
\sum_{i=2}^{n}(h_{2i-1}^2+h_{2i}^2)\left( 1 -\frac{b_i}{b_1} \pi \cot \frac{b_i \pi}{b_1}\right)>0,
\qquad
K'=2^{2n-2}
\prod_{i=2}^{n}
\frac{
1
}{b_i^2}
 \sin^2\left(
\frac{b_i \pi}{
b_1
}
\right)>0.
$$
We have
$$
\Phi
\left(
2\pi /b_1+\sqrt{\eta}\delta t,  \sqrt{\eta}\L\h+ (I_{2n}-\L)\h  ,\eta
\right)
=
-\frac{2\pi}{b_1^2} \eta^{2n+4}
KK' \delta t ^2
+o(\eta^{2n+4}).
$$
\end{lemma}
\begin{proof}
From Proposition~\ref{P:expansion_r1_eta}, we get that neither $G^{(1)}$ nor $G^{(2)}$ depend on $(h_1,h_2)$, hence from expression~\eqref{E:G_sqrt_eta} we deduce
$$
C_1=\eta^2\delta t_{1/2} \partial_{\h_1}  \partial_\tau G^{(3/2)}+O(\eta^{5/2}) ,	\qquad C_2=\delta t_{1/2} \partial_{\h_2} \partial_\tau G^{(3/2)}+O(\eta^{5/2}).
$$
Hence, evaluating $\Phi$ at $\left(
2\pi /b_1+\sqrt{\eta}\delta t,  \sqrt{\eta}\L\h+ (I_{2n}-\L)\h  ,\eta
\right)$ and eliminating higher order terms, there exist $V,W\in \mathcal{M}_{1\times 2n-2}(\R)$, $v\in \R$ such that 
$$
\Phi
=\eta^{2n+4}\left|
\begin{array}{cc|c|c}
			\delta t
		 &
			 0
		 &
		 0\cdots 0
		 &
		 0
	\\
			0
		 &
			 \delta t
		 &
		 0 \cdots 0
		 &
		 0
	\\
	\hline
	\:\:
	\begin{matrix}
	0
	\\
	\vdots
	\\
	0	
	\end{matrix}
	&
	\:\:
	\begin{matrix}
	0
	\\
	\vdots
	\\
	0	
	\end{matrix}
	&
	
	A
	
	&
	W
	\\
	\hline
			0
		 &
			 0
		  &
		  {}^t V
		 &
		 v
\end{array}
\right|_{\tau=2\pi/b_1}
+o(\eta^{2n+4}).
$$
Recall that $\det(A)=K'$ (see Lemma~\ref{L:detA}). To get the statement, let us show that 
$$
\left|
\begin{array}{c|c}	
	A	
	&
W
	\\
	\hline
		 ^{t}V
		 &
		 v
\end{array}
\right|=-\frac{2\pi}{b_1^2}
KK' .
$$
From Lemma~\ref{L:det(AW)} in Appendix~\ref{A:computational_lemmas}, we have
$$
\frac{1}{\det(A)}
\left|
\begin{array}{c|c}	
	A	
	&
W
	\\
	\hline
		 ^{t}V
		 &
		 v
\end{array}
\right|=
v+	
\\
\frac{1}{2}
\sum_{i=2}^{n}
b_i\left(
	V_{2i-1}W_{2i}-V_{2i}W_{2i-1}-\left(V_{2i-1}W_{2i-1}+V_{2i}W_{2i}\right)\cot \frac{b_{i}\pi}{b_1} 
	\right).
$$
In our case, for all $i\in \ll2,n\rr$,
$
(V_{2i-3},V_{2i-2})= (\h_{2i-1}, \h_{2i})\left(\frac{2\pi}{b_1}-\frac{1}{b_i}\sin \left(\frac{2 b_i\pi}{b_1}\right)\right)
$
and 
$$
\begin{pmatrix}
W_{2i-3}
\\
W_{2i-2}
\end{pmatrix}= 
\begin{pmatrix}
\frac{1}{b_i} \sin \frac{2 \pi  b_i}{b_1} -\frac{2 \pi  }{b_1}\cos  \frac{2 \pi  b_i}{b_1} 
&
\frac{1}{b_i}-\frac{2 \pi  }{b_1} \sin  \frac{2 \pi  b_i}{b_1} -\frac{1}{b_i} \cos \frac{2 \pi  b_i}{b_1} 
\\
\frac{2 \pi  }{b_1} \sin  \frac{2 \pi  b_i}{b_1} +\frac{1}{b_i} \cos  \frac{2 \pi  b_i}{b_1} -\frac{1}{b_i}
&
\frac{1}{b_i} \sin  \frac{2 \pi  b_i}{b_1} -\frac{2 \pi  }{b_1}\cos  \frac{2 \pi  b_i}{b_1} 
\end{pmatrix}
\begin{pmatrix}
\h_{2i-1}
\\
\h_{2i}
\end{pmatrix}.
$$
Finally,
$
v=\sum_{i=2}^{n}\left(\h_{2i-1}^2+\h_{2i}^2\right)\left( \frac{2 \pi  }{b_1} \cos ^2 \frac{\pi  b_i}{b_1}-\frac{1}{b_i} \sin  \frac{2 \pi  b_i}{b_1} \right)
$,
hence the result by summation.
\end{proof}

\begin{lemma}\label{P:mini_det_3}
We have
$$
\Phi
\left(
2\pi /b_1+\eta \delta t,h,\eta
\right)
=
\eta^{2n+5}
K'\left(d-\frac{2\pi}{b_1^2}Kd'\right)
+o(\eta^{2n+5}),
$$
where
$$
d=\left|
\begin{matrix}
			\frac{\partial}{\partial \h_1}
			\left(
			 G^{(5/2)}
			 +
			 \delta t
			 \frac{\partial G^{(3/2)}}{\partial \tau}
			 \right)_1
		 &
			 \frac{\partial}{\partial \h_2}
			\left(
			 G^{(5/2)}
			 +
			 \delta t
			 \frac{\partial G^{(3/2)}}{\partial \tau}
			 \right)_1
		 &
		 -\tau\left(\frac{\partial G^{(3/2)}}{\partial \tau}\right)_{1}
	\\
			\frac{\partial}{\partial \h_1}
			\left(
			 G^{(5/2)}
			 +
			 \delta t
			 \frac{\partial G^{(3/2)}}{\partial \tau}
			 \right)_2
		 &
			 \frac{\partial}{\partial \h_2}
			\left(
			 G^{(5/2)}
			 +
			 \delta t
			 \frac{\partial G^{(3/2)}}{\partial \tau}
			 \right)_2
		 &
		 -\tau\left(\frac{\partial G^{(3/2)}}{\partial \tau}\right)_{2}
	\\
			\frac{\partial}{\partial \h_1}
			\left(
			 G^{(3)}
			 \right)_{2n+1}
		 &
			 \frac{\partial}{\partial \h_2}
			\left(
			 G^{(3)}
			 \right)_{2n+1}
		 &
		 0
\end{matrix}
\right|_{\tau=2\pi/b_1}
$$
and
$$
d'=\left|
\begin{matrix}
			\frac{\partial}{\partial \h_1}
			\left(
			 G^{(2)}
			 +
			 \delta t
			 \frac{\partial G^{(1)}}{\partial \tau}
			 \right)_1
		 &
			 \frac{\partial}{\partial \h_2}
			\left(
			 G^{(2)}
			 +
			 \delta t
			 \frac{\partial G^{(1)}}{\partial \tau}
			 \right)_1
	\\
			\frac{\partial}{\partial \h_1}
			\left(
			 G^{(2)}
			 +
			 \delta t
			 \frac{\partial G^{(1)}}{\partial \tau}
			 \right)_2
		 &
			 \frac{\partial}{\partial \h_2}
			\left(
			 G^{(2)}
			 +
			 \delta t
			 \frac{\partial G^{(1)}}{\partial \tau}
			 \right)_2
\end{matrix}
\right|_{\tau=2\pi/b_1}.
$$
\end{lemma}
\begin{proof}
The proof is similar to that of Lemma~\ref{L:deltat1/2}. From Proposition~\ref{P:expansion_r1_eta} and expression~\eqref{E:G_sqrt_eta_bis} we deduce
$$
C_1=\eta^{5/2}\partial_{\h_1}  \left(G^{(5/2)}+ \delta t  \partial_\tau G^{(3/2)}\right)+O(\eta^{3}) ,
$$
$$
C_2=\eta^{5/2}\partial_{\h_2}  \left(G^{(5/2)}+ \delta t  \partial_\tau G^{(3/2)}\right)+O(\eta^{3}).
$$
Again, similarly to Lemma~\ref{L:deltat1/2}, evaluating $\Phi$ at $\left(
2\pi /b_1 +\eta \delta t,  \sqrt{\eta}\L\h+ (I_{2n}-\L)\h  ,\eta
\right)$ and eliminating higher order terms, there exist $V,W\in \mathcal{M}_{1\times 2n-2}(\R)$, $v\in \R$ such that at $\tau=2\pi/b_1$, the term of order $2n+5$ is given by
\begin{multline*}
\left|
\begin{array}{cc|c|c}
		\frac{\partial}{\partial \h_1}
			\left(
			 G^{(5/2)}
			 +
			 \delta t
			 \frac{\partial G^{(3/2)}}{\partial \tau}
			 \right)_1
		 &
			 \frac{\partial}{\partial \h_2}
			\left(
			 G^{(5/2)}
			 +
			 \delta t
			 \frac{\partial G^{(3/2)}}{\partial \tau}
			 \right)_1
		 &
		 0\cdots 0
		 &
		 0
	\\
\frac{\partial}{\partial \h_1}
			\left(
			 G^{(5/2)}
			 +
			 \delta t
			 \frac{\partial G^{(3/2)}}{\partial \tau}
			 \right)_2
		 &
			 \frac{\partial}{\partial \h_2}
			\left(
			 G^{(5/2)}
			 +
			 \delta t
			 \frac{\partial G^{(3/2)}}{\partial \tau}
			 \right)_2
		 &
		 0 \cdots 0
		 &
		 0
	\\
	\hline
	\:\:
	\begin{matrix}
	0
	\\
	\vdots
	\\
	0	
	\end{matrix}
	&
	\:\:
	\begin{matrix}
	0
	\\
	\vdots
	\\
	0	
	\end{matrix}
	&	
	A	
	&
	W
	\\
	\hline
		0
		 &
		0
		  &
		  {}^t V
		 &
		 v
\end{array}
\right|
+\\
\left|
\begin{array}{cc|c|c}
		\frac{\partial}{\partial \h_1}
			\left(
			 G^{(5/2)}
			 +
			 \delta t
			 \frac{\partial G^{(3/2)}}{\partial \tau}
			 \right)_1
		 &
			 \frac{\partial}{\partial \h_2}
			\left(
			 G^{(5/2)}
			 +
			 \delta t
			 \frac{\partial G^{(3/2)}}{\partial \tau}
			 \right)_1
		 &
		 0\cdots 0
		 &
		 -\tau\left(\frac{\partial G^{(3/2)}}{\partial \tau}\right)_{1}
	\\
\frac{\partial}{\partial \h_1}
			\left(
			 G^{(5/2)}
			 +
			 \delta t
			 \frac{\partial G^{(3/2)}}{\partial \tau}
			 \right)_2
		 &
			 \frac{\partial}{\partial \h_2}
			\left(
			 G^{(5/2)}
			 +
			 \delta t
			 \frac{\partial G^{(3/2)}}{\partial \tau}
			 \right)_2
		 &
		 0 \cdots 0
		 &
		 -\tau\left(\frac{\partial G^{(3/2)}}{\partial \tau}\right)_{2}
	\\
	\hline
	\:\:
	\begin{matrix}
	0
	\\
	\vdots
	\\
	0	
	\end{matrix}
	&
	\:\:
	\begin{matrix}
	0
	\\
	\vdots
	\\
	0	
	\end{matrix}
	&	
	A	
	&
	\:\:
	\begin{matrix}
	0
	\\
	\vdots
	\\
	0
	\end{matrix}
	\\
	\hline
\frac{\partial}{\partial \h_1}
			\left(
			 G^{(3)}
			 \right)_{2n+1}
		 &
			 \frac{\partial}{\partial \h_2}
			\left(
			 G^{(3)}
			 \right)_{2n+1}
		  &
		 0\cdots  0
		 &
		 0
\end{array}
\right|
\end{multline*}
Hence the result since $\det(A)=K'$ and   $
\left|
\begin{array}{c|c}	
	A	
	&
W
	\\
	\hline
		 ^{t}V
		 &
		 v
\end{array}
\right|=-\frac{2\pi}{b_1^2}K K'
$
(as showed in the proof of Lemma~\ref{L:deltat1/2}).
\end{proof}

\subsubsection{On the third domain}

Then for all  $ i\in \ll 1,2n\rr$, let $C_i$ and $C_{2n+1}$ be the respective  evaluations at time $\tau=2\pi/b_1+\eta \delta t_1+\eta^2\delta t_2$ of the vectors $\dfrac{\partial G}{\partial{\h_i}}$ and $\eta \dfrac{\partial G}{\partial \eta}-\tau \dfrac{\partial G}{\partial \tau}$.
For all  $ i\in \ll 1,2n+1\rr$, the vector $C_i\in \R^{2n+1}$ also admits a formal power series in $\eta$,
$
C_i=\sum_{k=1}^{\infty}\eta^k C_i^{(k)}.
$
Notice that by definition of $(C_i)_{ i\in \ll 1,2n+1\rr}$ we have 
$$
C_i^{(0)}=C_i^{(1)}=0 \qquad \forall  i\in \ll 1,2n\rr.
$$
Hence we have a priori
$
\Phi
\left(
2\pi /b_1,  \L\h+\eta (I_{2n}-\L)\h  ,\eta
\right)=O(\eta^{4n+1})
$.
We can use these elements to give the following refinement on Lemma~\ref{P:mini_det_1}.

\begin{lemma}\label{L:mini_det_4}
For all $\h\in \R^{2n}$, 
 $\delta t_1=\tau_c^{(1)}(\L \h)$ is the only solution to
 $$
 \Phi
\left(
2\pi /b_1+\eta \delta t_1+\eta^2 \delta t_2,  \L\h+\eta (I_{2n}-\L)\h  ,\eta
\right)
= O(\eta^{4n+2}).
 $$
Furthermore
$$
\Phi
\left(
2\pi /b_1+\eta \tau_c^{(1)}(\L \h)+\eta^2 \delta t_2,  \L\h+\eta (I_{2n}-\L)\h  ,\eta
\right)
=
\eta^{4n+2}K'
\left(\sum_{i=1}^{2n+1}d_i\right)
+O(\eta^{4n +3}),
$$
where
$
K'=2^{2n-2}
\prod_{i=2}^{n}
\frac{
1
}{b_i^2}
 \sin^2\left(
\frac{\pi  b_i}{
b_1
}
\right)>0
$ and
$$
d_1=
\left|
\begin{array}{llc}
			\left(C^{(3)}_1\right)_1
		 &
		 	\left(C^{(3)}_2\right)_1
		 &
		 \left(C^{(1)}_{2n+1}\right)_1
	\\
		\left(C^{(3)}_1\right)_2
		 &
		 \left(C^{(3)}_2\right)_2
		 &
		 \left(C^{(1)}_{2n+1}\right)_2
	\\
		\left(C^{(2)}_1\right)_{2n+1}
		 &
		 \left(C^{(2)}_2\right)_{2n+1}
		 &
0
\end{array}
\right|,
$$
$$
d_2=
\left|
\begin{array}{llc}
			\left(C^{(2)}_1\right)_1
		 &
		 	\left(C^{(2)}_2\right)_1
		 &
		 \left(C^{(1)}_{2n+1}\right)_1
	\\
		\left(C^{(2)}_1\right)_2
		 &
		 \left(C^{(2)}_2\right)_2
		 &
		 \left(C^{(1)}_{2n+1}\right)_2
	\\
		\left(C^{(3)}_1\right)_{2n+1}
		 &
		 \left(C^{(3)}_2\right)_{2n+1}
		 &
0
\end{array}
\right|,
$$
$$
d_{2n+1}=
\left|
\begin{array}{lll}
			\left(C^{(2)}_1\right)_1
		 &
		 	\left(C^{(2)}_2\right)_1
		 &
		 \left(C^{(2)}_{2n+1}\right)_1
	\\
		\left(C^{(2)}_1\right)_2
		 &
		 \left(C^{(2)}_2\right)_2
		 &
		 \left(C^{(2)}_{2n+1}\right)_2
	\\
		\left(C^{(2)}_1\right)_{2n+1}
		 &
		 \left(C^{(2)}_2\right)_{2n+1}
		 &
		\left(C^{(2)}_{2n+1}\right)_{2n+1}
\end{array}
\right|
$$
and
$$
d_{k}=
\frac{2\pi^2}{b_1^2}e_k \left(-h_2 \partial_{h_{k}}\left(G^{(3)}\right)_1+h_1 \partial_{h_{k}}\left(G^{(3)}\right)_2\right)\qquad \forall k\in\ll 3,2n\rr,
$$
where $e\in \R^{2n-2}$ is the vector such that $Ae$ is given by the components $3$ through $2n$ of the vector $\left(h_2 C_1^{(2)}-h_1 C_2^{(2)}\right)$, with $A\in \mathcal{M}_{2n-2}(\R)$  the matrix introduced in Lemma~\ref{P:mini_det_1}.
\end{lemma}
\begin{proof}
The first part of the statement is an application of Lemma~\ref{P:mini_det_1} in the case of an initial covector of the form  $h(0)=\L\h+\eta (I_{2n}-\L)\h$.
Indeed 
$$
\Phi
\left(
2\pi /b_1+\eta \tau_c^{(1)}+\eta^2 \delta t_2,  \L\h+\eta (I_{2n}-\L)\h  ,\eta
\right)
=
\eta^{4n+1}
\det
\left(
C_1^{(2)},
\dotsc,
C_{2n}^{(2)},
C_{2n+1}^{(1)}
\right)
+O(\eta^{4n+2}).
$$
The equation satisfied by $\tau_c^{(1)}$ comes down to 
$$
\eval{
\det
\left(
C_1^{(2)},
\dotsc,
C_{2n}^{(2)},C_{2n+1}^{(1)}
\right)}{\tau=2\pi /b_1+\eta \tau_c^{(1)}}=0
,$$
hence
\begin{multline*}
\Phi
\left(
2\pi /b_1+\eta \tau_c^{(1)}+\eta^2 \delta t_2,  \L\h+\eta (I_{2n}-\L)\h  ,\eta
\right)
=
\\
\eta^{4n+2}
\left[
\det
\left(
C_1^{(2)},
\dotsc,
C_{2n+1}^{(2)}
\right)
+
\sum_{k=1}^{2n}
\det
\left(
C_1^{(2)},
\dotsc,
C_{k-1}^{(2)},
C_{k}^{(3)},
C_{k+1}^{(2)},
\dotsc,
C_{2n}^{(2)},
C_{2n+1}^{(1)}
\right)\right]
\\+O(\eta^{4n+3}).
\end{multline*}

Setting 
$
d_k'=\det
\left(
C_1^{(2)},
\dotsc,
C_{k-1}^{(2)},
C_{k}^{(3)},
C_{k+1}^{(2)},
\dotsc,
C_{2n+1}^{(2)}
\right)$, for all $k\in \ll 3,2n\rr$, we first prove $d'_k=K' d_k$, for all $k\in \ll 3,2n\rr$.

We proceed to the following transformation on the columns $(C_i)_{i\in\ll1,2n+1\rr}$ of the Jacobian matrix. First, $C_1\leftarrow h_2C_1-h_1  C_2 $ and $C_2\leftarrow h_1 C_1 + h_2  C_2 $, then we transpose $C_k\leftrightarrow C_1$ and finally we cycle $C_{i+1}\leftarrow C_i$ for $i\in \ll3,2n\rr$ and $C_{3}\leftarrow C_{2n+1}$. This yields
$$
(h_1^2+h_2^2) d'_k
=
\det
\left(
C_{k}^{(3)},
h_1C_1^{(2)}+h_2C_2^{(2)},
C_{2n+1}^{(1)},
C_{3}^{(2)},
\dotsc,
C_{k-1}^{(2)},
h_2C_1^{(2)}-h_1C_2^{(2)},
C_{k+1}^{(2)},
\dotsc,
C_{2n}^{(2)}
\right) .
$$
Using Proposition~\ref{P:expansion_r2_eta}, $\left(C_{i}^{(2)}\right)_1=\left(C_{i}^{(2)}\right)_2=\left(C_{i}^{(2)}\right)_{2n+1}=0$, $i\in\ll3,2n\rr$.
All columns of this new matrix have zero $2n+1$ component except for $h_1C_1^{(2)}+h_2C_2^{(2)}$, and zero $1$ and $2$ component except for $C_{k}^{(3)}$,
$h_1C_1^{(2)}+h_2C_2^{(2)}$ and
$C_{2n+1}^{(1)}$. 
One can apply the Cramer rule for computing the $k$-th coefficient of $e=A^{-1}(h_2C_1^{(2)}-h_1C_2^{(2)})$ when computing
the determinant of the square submatrix of lines and columns 3 through $2n$.

Hence we have
$$
 d_k'
=
\frac{K' e_k}{h_1^2+h_2^2}
\det
\left(
\widetilde{C}_k^{(3)},
h_1\widetilde{C}_1^{(2)}+h_2\widetilde{C}_2^{(2)},
\widetilde{C}_{2n+1}^{(1)}
\right) 
$$
with $\widetilde{C}_i=\left({(C_i)}_1,{(C_i)}_2,{(C_i)}_{2n+1}\right)$, and we get the value of $d_k$ by computing the remaining determinant.

Similarly, we obtain the stated relation for $d_1$, $d_2$ and $d_{2n+1}$ by noticing that $C_{2n+1}^{(1)}=0$ and isolating the three $3\times 3$ matrices given by lines and columns $1$, $2$ and $2n+1$.
\end{proof}

The value of determinants $d_1$ through $d_{2n+1}$ can be explicitly stated in terms of second order invariants thanks to the computations in Appendix~\ref{A:third_order}.

\begin{lemma}\label{L:expression_d_i}
We have $d_{2n+1}=0$, $d_2=-\frac{2\pi}{ b_1}(\h_1^2+\h_2^2)(\beta \h_1-\alpha \h_2)^2$ and
$$
\begin{aligned}
d_1=&\frac{4 \pi^2}{b_1^2}
(\h_1^2+\h_2^2)\left(
\delta t_2+ 4b_1(\beta \h_1-\alpha \h_2)(\alpha \h_1 +\beta \h_2)
\right)
\\
&\frac{4 \pi^2}{b_1^2}\left(\h_1^2  \frac{\partial \left(
			 G^{(3)}
			 \right)_2}{\partial \h_2}			
+
\h_2^2  \frac{\partial \left(
			 G^{(3)}
			 \right)_1}{\partial \h_1}	
-
\h_1 \h_2
\left(  \frac{\partial \left(
			 G^{(3)}
			 \right)_1}{\partial \h_2}
			 +
			 \frac{\partial \left(
			 G^{(3)}
			 \right)_2}{\partial \h_1}
\right)\right).
\end{aligned}
$$
Furthermore, for all $i\in \ll2,n\rr$, we have 
$$
\begin{aligned}
d_{2i-1}=&
\frac{2 \pi^2 b_i }{b_1^2}\left(h_1 (h_1 \kappa^{1,{2i-1}}_2+h_2 \kappa^{2,{2i-1}}_2)-h_2 (h_1 \kappa^{1,{2i-1}}_1+h_2 \kappa^{2,{2i-1}}_1)\right) 
\\
&\qquad\left[
	\cot \left(\frac{\pi  b_i}{b_1}\right) 
	\left(
		\kappa^{1,2}_{2i-1}(h_2^2-h_1^2)+2 h_1 h_2 (\kappa^{1,1}_{2i-1}-\kappa^{2,2}_{2i-1})
	\right)
	\right.
	\\
	&
	\qquad\qquad \left. \vphantom{\left(\frac{\pi  b_i}{b_1}\right) }
	-\left(
		\kappa^{1,2}_{2i}(h_2^2-h_1^2)+2 h_1 h_2 (\kappa^{1,1}_{2i}-\kappa^{2,2}_{2i})
	\right)
\right],
\end{aligned}
$$
$$
\begin{aligned}
d_{2i}=&\frac{2 \pi ^2 b_i}{b_1^2} \left(   h_1 (h_1 \kappa^{1,2i}_2+h_2 \kappa^{2,2i}_2)-h_2 (h_1 \kappa^{1,2i}_1+h_2 \kappa^{2,2i}_1)   \right) 
\\
&\qquad\left[
	\cot \left(\frac{\pi  b_i}{b_1}\right) 
	\left(
		\kappa^{1,2}_{2i}(h_2^2-h_1^2)+2 h_1 h_2 (\kappa^{1,1}_{2i}-\kappa^{2,2}_{2i})
	\right)
	\right.
	\\
	&
	\qquad\qquad \left. \vphantom{\left(\frac{\pi  b_i}{b_1}\right) }+\left(
		\kappa^{1,2}_{2i-1}(h_2^2-h_1^2)+2 h_1 h_2 (\kappa^{1,1}_{2i-1}-\kappa^{2,2}_{2i-1})
	\right)\right].
\end{aligned}
$$
\end{lemma}

\begin{proof}
First, recall that 
$$
\begin{aligned}
x^{(2)}_1(2\pi/b_1,\Lambda \h)=\alpha   (3 \h_1^2+\h_2^2)+2 \beta \h_1 \h_2,
\\
x^{(2)}_2(2\pi/b_1,\Lambda \h)=2 \alpha \h_1 \h_2+\beta  (\h_1^2+3 \h_2^2 )
\end{aligned}
$$
and  $\tau_c^{(1)}(\L\h)=-2 ( \alpha \h_1+\beta  \h_2 )$.
Using Lemma~\ref{L:z3} from the Appendix, we have the value of $z^{(3)}\left(\frac{2\pi}{b_1},\L\h\right)$ and we can compute  the $3\times 3$ determinant $d_2$. (Remark that $F_{2n+1}^{(3)}=z^{(3)}+\tau_c^{(1)}\partial_\tau z^{(2)}(2\pi/b_1)$ and that $\partial_\tau z^{(2)}(2\pi/b_1)=0$.) Similarly we can compute $d_{2n+1}$ by noticing, for $i\in \ll1,2\rr$, at $\tau=2\pi/b_1+\eta \tau_c^{(1)}$
$$
\left(\eta \partial_\eta F_i-\tau \partial_\tau F_i\right)=
2x_i^{(2)}\left(\frac{2\pi}{b_1},\L\h\right) -\frac{2\pi}{b_1}\left(h_i^{(1)}\left(\frac{2\pi}{b_1},\L\h\right)+\tau_c^{(1)}(\L\h)(\Jbar\h)_i\right).
$$
Regarding $d_k$, $k\in \ll 3,2n\rr$, we obtain the result by explicitly computing the vector $e\in \R^{2n-2}$. First, 
since $\partial_{\h_{k}}z^{(3)}=0$
$$
\det\left(
\widetilde{C}_k^{(3)},
h_1\widetilde{C}_1^{(2)}+h_2\widetilde{C}_1^{(2)},
\widetilde{C}_{2n+1}^{(1)}
\right) 
=
-\frac{4\pi^2}{b_1^2}(h_1^2+h_2^2)
\left[
h_1h_2(\kappa^{1,k}_1 -\kappa^{2,k}_2)+h_2^2 \kappa^{2,k}_1 - h_1^2\kappa^{1,k}_2  )
\right].
$$
On the other hand, 
we have  $Ae=h_2C_1^{(2)}-h_1C_2^{(2)}$ and for all $3\leq i\leq 2n$
$$
\left(h_2C_1^{(2)}-h_1C_2^{(2)}\right)_i= \kappa^{1,2}_{i}(h_2^2-h_1^2)+2 h_1 h_2 (\kappa^{1,1}_i-\kappa^{2,2}_i).
$$
We then get the stated result since $A^{-1}$ is block diagonal with blocks in position $i-1$ being, for all $i\in\ll2,n\rr$,
$$
\frac{b_i}{2}
\begin{pmatrix}
\cot  \pi b_i /b_1 &-1
\\
1 & \cot  \pi b_i /b_1
\end{pmatrix}.
$$
\end{proof}

\section{Singularity classification}
\label{SS:Kernel_comp}

On each domain, the first step of the classification is to properly describe the Jacobian matrix of the exponential.
Recall that the rank is lower semi-continuous as a map from $\mathcal{M}_5(\R)$ to $\N$.
This implies  that the Jacobian matrix can have a kernel of  dimension   at most  $2$ at times near $2\pi/b_1$, as it is the case for the first order approximation $\widehat{\sre}$.

We decompose the matrix $\Jac_{p_0} \sre_{q_0}$ into the following sub matrices:
$$
\left(
\begin{array}{c|c|c}
A_1&A_2&C_1
\\
\hline
A_3&A_4&C_2
\\
\hline
L_1&L_2&E
\end{array}
\right)
$$
with $A_1,A_2,A_3,A_4 \in \mathcal{M}_{2\times 2}(\R)$, $L_1,L_2\in \mathcal{M}_{1\times 2}(\R)$, $C_1,C_2\in \mathcal{M}_{2\times 1}(\R)$ and $E\in \mathcal{M}_{1\times 1}(\R)$.

A vector $v$ in the kernel of $\Jac_{p_0} \sre_{q_0}$ must the satisfy the equations
\begin{equation}\label{E:ker12}
A_1 
\begin{pmatrix}
v_1 
\\
v_2 
\end{pmatrix}
+
A_2
\begin{pmatrix}
v_3 
\\
v_4 
\end{pmatrix}
+C_1v_5 =0,
\end{equation}
\begin{equation}\label{E:ker34}
A_3 
\begin{pmatrix}
v_1 
\\
v_2 
\end{pmatrix}
+
A_4
\begin{pmatrix}
v_3
\\
v_4
\end{pmatrix}
+C_2v_5 =0,
\end{equation}
\begin{equation}\label{E:ker5}
L_1
\begin{pmatrix}
v_1
\\
v_2
\end{pmatrix}
+
L_2
\begin{pmatrix}
v_3
\\
v_4
\end{pmatrix}
+
E
v_5
 =0.
\end{equation}
In the following three sections, we compute approximations of elements of the kernel with initial covectors of the form 
$$
\left(h_1,h_2,h_3,h_4,\eta^{-1}\right),\qquad \left(\sqrt{\eta}h_1,\sqrt{\eta}h_2,h_3,h_4,\eta^{-1}\right)\quad \text{ and } \quad \left(h_1,h_2,\eta h_3,\eta h_4,\eta^{-1}\right).
$$
All expansions as $\eta\to 0$ are assumed uniform under the condition $h_1^2+h_2^2+h_3^3+h_4^4<R$ for some arbitrary $R>0$.

\begin{remark}
The following computations make abundant use of explicit expressions of the approximations of the exponential map obtained in Section~\ref{S:conjugate_time}. Readers wishing to precisely follow the computations are referred to Propositions \ref{P:expansion}, \ref{P:expansion_r1_eta} and \ref{P:expansion_r2_eta} for a general expression of the approximation of the exponential map, and the results of Section~\ref{S:conjugate_time} and Appendix~\ref{A:Computation_invariants} for expressions in terms of  invariants.

\end{remark}

\subsection{First domain: initial covectors in $T^*_{q_0}M\setminus (S_1\cup S_2)$}
\label{A:dom1}
\subsubsection{Jacobian matrix}\label{SSS:in_cov_T_setminS1S2}

From computations of the conjugate time, we know that $\ker \Jac_{p_0} \sre_{q_0}\neq \{0\}$ at $t=t_c(p_0)$. Let us compute a first approximation of the set of solutions of the equation
$\Jac_{p_0} \sre_{q_0}(t_c(p_0)) \cdot v=0$ (thanks to our approximation of $F(\tau)=\sre(\eta\tau)$).

\begin{proposition}\label{L:kernel_r1r2>0}
The kernel of $\Jac_{p_0} \sre_{q_0}(t_c(p_0))$ is $1$-dimensional and there exists $\nu(h_1,h_2,h_3,h_4)$ such that it is generated by the vector 
$$
(-h_2,h_1,0,0,\nu)+O(\eta).
 $$
\end{proposition}

\begin{proof}
According to the computations carried  in Section \ref{SS:asymptotics_away_S}, we have 
$$
A_i=O(\eta^2), \quad A_4\neq O(\eta^2), \qquad i\in \ll1,3\rr ,
$$
$$
C_i=O(\eta^2), \quad L_i=O(\eta^2), \quad E=O(\eta^3),\qquad i\in \ll1,2\rr.
$$
Regarding $C_1,C_2,E$, this is in particular a consequence of $\partial_5 F=-\eta^2\partial_\eta F+\eta \tau \partial_\tau  F$. Then \eqref{E:ker34} implies $v_3=O(\eta)$ and $v_4=O(\eta)$ since $A_4^{(1)}$ is invertible, and from \eqref{E:ker5} we obtain
$$
L_1
\begin{pmatrix}
v_1
\\
v_2
\end{pmatrix}=O(\eta^3).
$$
That is $ h_1 v_1+h_2 v_2=O(\eta)$, hence there exists $\lambda\in \R$ such that $v_1=-\lambda h_2+O(\eta)$, $v_2=\lambda h_1+O(\eta)$.
Similarly, \eqref{E:ker12} yields
$$
A_1^{(2)}
\begin{pmatrix}
v_1^{(0)}
\\
v_2^{(0)}
\end{pmatrix}
+C_1^{(2)}v_5^{(0)} =0.
$$
Since $\tau_c^{(1)}$ corresponds to the fact that 
$
A_1^{(2)}
\begin{pmatrix}
-h_2 
\\
h_1 
\end{pmatrix}$
 is colinear to $C_1^{(2)}=\dfrac{2\pi}{b_1}\begin{pmatrix}
h_1
\\
h_2 
\end{pmatrix}
$, with $(v_1^{(0)},v_2^{(0)})= \lambda (-h_2,h_1)$,  $v_5^{(0)}$ is uniquely defined, linearly dependent on $\lambda$.
Similarly,
we compute 
$$
\begin{pmatrix}
v_3^{(1)}
\\
v_4^{(1)}
\end{pmatrix}=
-\left(A_4^{(1)}\right)^{-1}\left(v_5^{(0)} C_2^{(2)}+A_3^{(2)}
\begin{pmatrix}
v_1^{(0)}
\\
v_2^{(0)}
\end{pmatrix}\right).
$$
Hence the statement. The kernel of $\Jac_{p_0} \sre_{q_0}(t_c(p_0))$  is in particular $1$-dimensional as a consequence of the lower semi-continuity of the rank.
\end{proof}
Regarding the image space, we have can give a description as a consequence of Lemma~\ref{L:kernel_r1r2>0}.
\begin{lemma}\label{L:image_Jac_r1r2>0}
Let $p_0\in T^*_{q_0}M\setminus (S_1\cup S_2)$. The image of the Jacobian at $p_0$ of the exponential at the conjugate time admits the representation
$$
\im \Jac_{p_0} \sre_{q_0}(t_c(p_0))
=
\Span
\left\{
h_1 \partial_1F+h_2 \partial_2F,\partial_3F,\partial_4F,\partial_5F
\right\}.
$$
\end{lemma}
\begin{proof}
Let $v_1$ be such that $\ker \Jac_{p_0} \sre_{q_0}(t_c(p_0))=\Span(v_1)$. For any $4$ vectors $v_2,v_3,v_4,v_5$ such that $\mathrm{rk}(v_1,v_2,v_3,v_4,v_5)=5$,
we have the property that 
$$
\im \Jac_{p_0} \sre_{q_0}(t_c(p_0))=\Span\left( \sum_{k=1}^5 {\left(v_i\right)}_k \partial_k F    \right)_{i\in \ll2, 5\rr}.
$$

One possible choice is then $v_2=(h_1,h_2,0,0,0)$, $v_3=(0,0,1,0,0)$,    $v_4=(0, 0, 0, 1, 0)$, and $v_5=(0,0,0,0,1)$.
\end{proof}

\subsubsection{Classification}

We first introduce a computational lemma approximate the $\phi$ functions  from Proposition~\ref{P:second_domain_func}.
\begin{lemma}\label{L:reduced_cond_Case1}
For all $i\in\ll 1,5\rr$,  let $U_i : \R^4\to \R$ and
let 
$$\Psi(u_1,u_2,u_5)=-
u_5  \left({}^t  e_{\theta_1}A_1^{(2)}  e_{r_1}\right)+\frac{2\pi}{b_1}(h_1^2+h_2^2)\left(h_1 u_2-h_2 u_1\right).
$$
Then  we have for $p_0=(h,\eta^{-1})$, uniformly with respect to $h\in B_R$ as $\eta \to 0$,
$$
\det\left( U(h),h_1 \partial_1F+h_2 \partial_2 F,\partial_3F,\partial_4F,\partial_5F\right) = 
\eta^6 \frac{8 \pi}{b_1b_2^2}\sin^2\left( \frac{\pi b_2}{b_1}\right)\Psi(U_1(h),U_2(h),U_5(h))+o(\eta^6).
$$

\end{lemma}
\begin{proof}
We compute the dominant term of 
$\det(h_1 \partial_1F+h_2 \partial_2F,\partial_3F,\partial_4F,\partial_5F,U(h))$.
Using notations from Section~\ref{SSS:in_cov_T_setminS1S2} and a similar reasoning to what can be found in Section~\ref{S:order_2_approx}, we obtain
$$
\det(h_1 \partial_1F+h_2 \partial_2F,\partial_3F,\partial_4F,\partial_5F,U(h))
=
\left|
\begin{array}{c|c|c|c}
\eta^2A_1^{(2)}e_{r_1} &	\:\:\begin{matrix} 0 &0\\ 0 &0\end{matrix}& \eta^2\frac{2\pi}{b_1}e_{r_1}&	\:\:\begin{matrix} U_1(h) \\ U_2(h) \end{matrix}
\\
\hline
	\:\:\begin{matrix} 0 \\ 0 \end{matrix} & \eta A_4^{(1)}&	\:\:\begin{matrix} 0 \\ 0 \end{matrix}  &	\:\:\begin{matrix} 0 \\ 0 \end{matrix}
\\
\hline
\eta^2\frac{2\pi}{b_1}(h_1^2+h_2^2)&0\quad 0&0 &U_5(h)
\end{array}
\right|
+
o(\eta^6).
$$
We have the result once observed that $\det A_4^{(1)}=\frac{4}{b_2^2}\sin^2\left( \frac{\pi b_2}{b_1}\right)$ and
$$
\begin{aligned}
\Psi(U_1,U_2,U_5)
&=
\det\left(
\begin{array}{c|c|c}
A_1^{(2)}e_{r_1} &e_{r_1} &	\:\:\begin{matrix} U_1 \\ U_2 \end{matrix}
\\
\hline
\frac{2\pi}{b_1}(h_1^2+h_2^2)&0&U_5
\end{array}
\right)
\\&=-U_5  \left({}^t  e_{\theta_1}A_1^{(2)}  e_{r_1}\right) +\frac{2\pi}{b_1}(h_1^2+h_2^2)\left(h_1 U_2-h_2 U_1\right).
\end{aligned}
$$
\end{proof}

Let $p_0\in T_{q_0}^*M\setminus (S_1\cup S_2)$ and $v$ be as in the statement of Proposition~\ref{L:kernel_r1r2>0} so that $\ker \Jac_{p_0} \sre_{q_0}(t_c(p_0))=\Span(v)$. As explained in Remark~\ref{R:method_derivatives}, we choose the first
coordinate $\x_1:M\to \R$ such that  $\partial_{\x_1} = \sum_{i=1}^{5}{v}_i\partial_i$. Since $v_3,v_4=O(\eta)$, we have that $\partial_{x_1}^k F=O(\eta^2)$ for all integer $k\geq 2$.

If we denote $V':\R^4\to \R^5$ such that $\partial_{x_1}^2 F=\eta^2 V'(h)+o(\eta^2)$ then let $\Psi_2(h)=\Psi(V'_1,V'_2,V'_5)$.
Similarly, define $V'':\R^4\to \R^5$ such that $\partial_{x_1}^3 F=\eta^2 V''(h)+o(\eta^2)$ and $V''':\R^4\to \R^5$ such that $\partial_{x_1}^4 F=\eta^2 V'''(h)+o(\eta^2)$; and define
$\Psi_3(h)=\Psi(V''_1,V''_2,V''_5)$, $\Psi_4(h)=\Psi(V'''_1,V'''_2,V'''_5)$.

Since the length of expressions is still manageable in this case, we can give the explicit form of $\Psi_2$, $\Psi_3$ and $\Psi_4$ (up to multiplication by $2\pi(h_1^2 + h_2^2) /b_1$):
$$
\begin{aligned}
\Psi_2(h_1,h_2,h_3,h_4)=&-(h_3 (\kappa^{1,3}_2+2 \kappa^{2,3}_1)+h_4 (\kappa^{1,4}_2+2 \kappa^{2,4}_1)) h_1^2 
\\&+3 (h_3 (\kappa^{1,3}_1-\kappa^{2,3}_2)+h_4 (\kappa^{1,4}_1-\kappa^{2,4}_2))h_1 h_2
\\&+ (h_3 (2  \kappa^{1,3}_2+\kappa^{2,3}_1)+h_4 (2 \kappa^{1,4}_2+\kappa^{2,4}_1))h_2^2,
\\
\Psi_3(h_1,h_2,h_3,h_4)=
&
+(h_3 (\kappa^{1,3}_1- \kappa^{2,3}_2)+h_4 (\kappa^{1,4}_1-\kappa^{2,4}_2)) h_1^2 
\\&
+2 (h_3 (\kappa^{1,3}_2+\kappa^{2,3}_1)+h_4 (\kappa^{1,4}_2-\kappa^{2,4}_1))h_1 h_2
\\&
-(h_3 (\kappa^{1,3}_1- \kappa^{2,3}_2)+h_4 (\kappa^{1,4}_1-\kappa^{2,4}_2)) h_2^2 ,
\\
\Psi_4(h_1,h_2,h_3,h_4)=
&
- ( h_3 (3 \kappa^{1,3}_2+4 \kappa^{2,3}_1)+ h_4 (3 \kappa^{1,4}_2+4 \kappa^{2,4}_1))
h_1^2
\\
&
+
7( h_3 (\kappa^{1,3}_1-\kappa^{2,3}_2)+ h_4 (\kappa^{1,4}_1-\kappa^{2,4}_2))
h_1 h_2 
\\
&
+
 ( h_3 (4 \kappa^{1,3}_2+3 \kappa^{2,3}_1)+ h_4 (4 \kappa^{1,4}_2+3 \kappa^{2,4}_1))
h_2^2.
\end{aligned}
$$

As an application of Lemma~\ref{L:reduced_cond_Case1}, and the analysis of the Jacobian matrix of $\sre_{q_0}(t_c(p_0))$ of Section~\ref{SSS:in_cov_T_setminS1S2}, we immediately obtain that  for $\eta$ small enough
$$
\Psi_2(h)\neq 0 \Rightarrow \phi_{11}(p_0)\neq 0,
\;
\Psi_3(h)\neq 0 \Rightarrow \phi_{111}(p_0)\neq 0,
\;
\Psi_4(h)\neq 0 \Rightarrow \phi_{1111}(p_0)\neq 0.
$$

\subsection{Second domain: initial covectors near $S_1$}
\label{A:dom2}
\subsubsection{Jacobian matrix}\label{SS:covec_near_S1}

The idea is the same as before, now we consider initial covectors of the form 
$$
p_0=\left(\sqrt{\eta} h_1,\sqrt{\eta} h_2, h_3,h_4,\eta^{-1}\right).
$$

\begin{proposition}\label{L:kernel_r1=0}
If there exist a time near $2\pi\eta/b_1$ that is conjugate for $p_0$ then the kernel of $\Jac_{p_0} \sre_{q_0}(t_c(p_0))$ is either $1$ or $2$-dimensional. If 
 $(h_1,h_2)\neq (0,0)$  then
there exist two vectors
$$
v_{\theta_1}=(-h_2,h_1,0,0,0)+O(\eta)\quad \text{ and } \quad v_{r_1}=\left(h_1,h_2,0,0,  -\frac{ (h_1^2+h_2^2)}{  K}\right)+O(\eta)
$$
such that 
the kernel of $\Jac_{p_0} \sre_{q_0}(t_c(p_0))$ is 
either $\Span\left(\lambda_{\theta_1}v_{\theta_1}+\lambda_{r_1}v_{r_1}\right)$ or $\Span\left(v_{\theta_1},v_{r_1}\right)$.
\end{proposition}

\begin{proof}
From the computations in Section \ref{SS:asymptotics_near_S_1}, we have 
$$
A_i=O(\eta^{5/2}), \quad A_4\neq O(\eta^2), \qquad i\in \ll1,3\rr ,
$$
$$
C_1=O(\eta^{5/2}), \quad L_1=O(\eta^{3}),\quad C_2=O(\eta^{2}), \quad L_2=O(\eta^{2}), \quad E=O(\eta^3).
$$
As previously, \eqref{E:ker34} implies $v_3=O(\eta)$ and $v_4=O(\eta)$ and similarly to Section~\ref{SSS:in_cov_T_setminS1S2}, $\left(v_3^{(1)},v_4^{(1)}\right)$ can be computed as
$$
\begin{pmatrix}
v_3^{(1)}
\\
v_4^{(1)}
\end{pmatrix}
=
-v_5^{(0)}\left(A_4^{(1)}\right)^{-1}C_2^{(2)}.
$$

Hence the smallest non-vanishing order of the system \eqref{E:ker12}-\eqref{E:ker34}-\eqref{E:ker5} reduces  to the $3\times 3$ system 
\begin{equation}\label{E:reduced_r1->0}
A_1^{(5/2)}
\begin{pmatrix}
v_1^{(0)}
\\
v_2^{(0)}
\end{pmatrix}
+C_1^{(5/2)}v_5^{(0)} =0,
\end{equation}
\begin{equation}\label{E:reduced_r1->0bis}
L_1^{(3)}
\begin{pmatrix}
v_1^{(0)}
\\
v_2^{(0)}
\end{pmatrix}
+
\left(
E^{(3)}-L_2^{(2)}\left(A_4^{(1)}\right)^{-1}C_2^{(2)}
\right)
v_5^{(0)}
 =0.
\end{equation}
Now observe that $E^{(3)}-L_2^{(2)}\left(A_4^{(1)}\right)^{-1}C_2^{(2)}=-\frac{2 \pi}{b_1}K$, where $K$ is the constant introduced in Lemma~\ref{L:deltat1/2}. Furthermore from Propositions~\ref{P:mini_det_3} and \ref{P:t_c_bu_eta1/2}, we know that the first conjugate time is a perturbation of $2\pi\eta /b_1$ if
\begin{equation}\label{E:detK'=0}
\det 
\left(
\begin{array}{c|c}
A_1^{(5/2)}&C_1^{(5/2)}
\\
\hline
L_1^{(3)}&\frac{2 \pi}{b_1}K
\end{array}
\right)=0.
\end{equation}
When that is the case,  the set of solutions of \eqref{E:reduced_r1->0}-\eqref{E:reduced_r1->0bis} is at least $1$-dimensional, otherwise it is reduced to $\{0\}$.

Assume \eqref{E:detK'=0} holds and  that  $(h_1,h_2)\neq (0,0)$.  Let us denote $e_{r_1}=(h_1,h_2)$ and $e_{\theta_1}=(-h_2,h_1)$. There exist unique $\lambda_{r_1},\lambda_{\theta_1}\in \R$ such that $\left(v_1^{(0)},v_2^{(0)}\right)=\lambda_{r_1} e_{r_1}+\lambda_{\theta_1} e_{\theta_1}$. 
Since $\left(\partial_{h_1}F^{(3)}_5,\partial_{h_2}F^{(3)}_5\right)\in \Span (e_{r_1})$, we have from  \eqref{E:reduced_r1->0bis} that
 $$
v_5^{(0)}=-\lambda_{r_1} \frac{b_1  L_1^{(3)}e_{r_1}}{2 \pi K},
 $$
and from \eqref{E:reduced_r1->0} we get
$$
\lambda_{r_1} \left(A_1^{(5/2)} e_{r_1}- \frac{b_1 L_1^{(3)}e_{r_1}}{2 \pi K}C_1\right)+\lambda_{\theta_1} A_1^{(5/2)}e_{\theta_1}=0.
$$
Recall that $L_1^{(3)}=\frac{2\pi}{b_1}\begin{pmatrix}
h_1&h_2
\end{pmatrix}$, thus $\frac{b_1 L_1^{(3)}e_{r_1}}{2 \pi K}=\frac{h_1^2+h_2^2}{K}$.
Elements of the kernel must be linear combination of the vectors 
$$
v_{\theta_1}=(-h_2,h_1,0,0,0)+O(\eta)\quad \text{ and } \quad v_{r_1}=(h_1,h_2,0,0, -(h_1^2+h_2^2)/ K)+O(\eta).
$$
Assuming \eqref{E:detK'=0} holds, there are two cases: 
\begin{enumerate}
\item Either $A_1^{(5/2)} e_{r_1}+ \frac{h_1^2+h_2^2}{ K}C_1\neq 0$ or $A_1^{(5/2)}e_{\theta_1}\neq 0$, and  the kernel is a $1$-dimensional space generated by a linear combination of $v_{\theta_1}$ and $v_{r_1}$.

\item Both $A_1^{(5/2)} e_{r_1}+  \frac{h_1^2+h_2^2}{ K}C_1= 0$ and $A_1^{(5/2)}e_{\theta_1} = 0$, and the kernel is the  $2$-dimensional space $\Span(v_{\theta_1},v_{r_1})$.
\end{enumerate}
If $h_1=h_2=0$, assuming \eqref{E:detK'=0} holds implies that $v_5^{(0)}=0$ and the kernel is of the dimension of $\ker A_1^{(5/2)}$.
\end{proof}

\begin{remark}\label{R:2dim_ker}
Notice that
a $2$-dimensional kernel implies that the conjugate time is a zero of order 2, that is, $\Delta=0$. (The converse may not be true however.) Indeed, if $(h_1,h_2)\neq (0,0)$, $A_1^{(5/2)}e_{\theta_1}= 0$ implies  we must have for some $a,b\in \R$ 
$$
A_1^{(5/2)}=
\begin{pmatrix}
a h_1& a h_2
\\
b h_1 & b h_2
\end{pmatrix}.
$$
Then $  A_1^{(5/2)} e_{r_1}=  -\frac{h_1^2+h_2^2}{ K}C_1$ implies $a= -h_1 \frac{2\pi }{b_1 K}$, $b= -h_2 \frac{2\pi }{b_1 K}$. Under these conditions, one can check that the zero is of order 2.

If $(h_1,h_2)=(0,0)$ however, having a 2-dimensional kernel corresponds to $A_1^{(5/2)}=0$. However, in that case, using notations from Theorem~\ref{T:big_expansion_theorem}, this implies that $\gamma_{12}=\gamma_{21}=\gamma_{11}-\gamma_{22}=0$. From Proposition~\ref{P:def_S_2}, this is exactly stating that $q_0\in \mathfrak{S}_2$, hence the kernel of $\Jac_{p_0} \sre_{q_0}(t_c(p_0))$ for an initial covector $p_0$ in $S_1$ is of dimension at most $1$ at points of $M\setminus \mathfrak{S}_2$.

\end{remark}

Finally, let us give a useful description of the image set of the Jacobian matrix of $\sre_{q_0}(t_c(p_0))$ in the case of 1D kernel with initial covector such that $(h_1,h_2)\neq (0,0)$.

Let $\lambda_{r_1},\lambda_{\theta_1}$ be such that $\Span(\lambda_{r_1} v_{r_1}+\lambda_{\theta_1}v_{\theta_1})=\ker \Jac_{p_0} \sre_{q_0}(t_c(p_0))$, and let $V,W$ be two vectors in the image set of $\Jac_{p_0} \sre_{q_0}(t_c(p_0))$ such that
$$
W=\partial_5\bar{F}-\eta w_3\partial_3\bar{F}-\eta w_4\partial_4\bar{F},\text{ with }
\begin{pmatrix}
w_3
\\
w_4
\end{pmatrix}=
-\left(A_4^{(1)}\right)^{-1}C_2^{(2)}
$$
and
$$
V=-\lambda_{\theta_1} \left(h_1 \partial_1 \bar{F}+h_2 \partial_2 \bar{F}+\frac{(h_1^2+h_2^2)}{K} W\right)+\lambda_r \left(-h_2 \partial_1 \bar{F}+h_1 \partial_2 \bar{F} \right).
$$
They have been chosen to simplify low order terms in their expansions as $\eta\to 0$. Indeed 
$$
W_1=\eta^{5/2}2\pi/b_1 h_1+o(\eta^{5/2}),\quad W_2=\eta^{5/2}2\pi/b_1 h_2+o(\eta^{5/2}),
$$ 
$$
(W_3,W_4)=o(\eta^{5/2})\quad \text{ and } \quad W_5=-\eta^{3}\frac{2 \pi}{b_1}K+o(\eta^3).
$$
Likewise, $(V_1,V_2 )\neq o(\eta^{5/2})$ but $V_3,V_4=O(\eta^{5/2})$ and $V_5=o(\eta^3)$. (This observation is useful for the next section in particular.)

\begin{lemma}\label{L:image_Jac_r1=0}
Assume $p_0=\left(\sqrt{\eta} h_1,\sqrt{\eta} h_2, h_3,h_4,\eta^{-1}\right)$ is an initial covector such that $(h_1,h_2)\neq (0,0)$ and the kernel of $\Jac_{p_0} \sre_{q_0}(t_c(p_0))$ is of dimension $1$.
Then
$$
\im \Jac_{p_0} \sre_{q_0}(t_c(p_0))
=
\Span
\left\{
V,W,\partial_3F,\partial_4F
\right\}.
$$
\end{lemma}
\begin{proof}
The proof is analogous to the proof of Lemma~\ref{L:image_Jac_r1r2>0}. 
The kernel is spanned by $\lambda_{\theta_1}v_{\theta_1}+\lambda_{r_1}v_{r_1}$.

Let $v_3=(0,0,1,0,0)$, $v_4=(0,0,0,1,0)$, $w=(0,0,-\eta w_3,-\eta w_4,1)$ and 
$$
v=-\lambda_{\theta_1}(v_{r_1}-\eta w_3v_3-\eta w_4 v_4)+\lambda_{r_1} v_{\theta_1}.
$$
By construction,
$$
\mathrm{rk}(\lambda_{\theta_1}v_{\theta_1}+\lambda_{r_1}v_{r_1},v,w,v_3,v_4)=5,
$$
Hence the result since $V=\Jac_{p_0} \sre_{q_0}(t_c(p_0))\cdot v$ and $W=\Jac_{p_0} \sre_{q_0}(t_c(p_0))\cdot w$.
\end{proof}

\subsubsection{Classification}
Again, we introduce a lemma to help us approximate the $\phi$ functions.
\begin{lemma}\label{L:reduced_cond_Case2}
Let $V,W$  be as in the statement of Lemma~\ref{L:image_Jac_r1=0}.
For all $i\in\ll 1,5\rr$,  let $U_i : \R^4\to \R$ and
let 
$$\Phi(u_1,u_2,u_5)=
u_5  \left(V_1^{(5/2)} h_2-V_2^{(5/2)}h_1\right)+K \left(V_2^{(5/2)}u_1-V_1^{(5/2)} u_2\right).
$$
Let also $\mathfrak{d}_\eta:\R^5\to \R^5$ be such that
$
\mathfrak{d}_\eta(u)=(\eta^{5/2} u_1, \eta^{5/2} u_2, \eta^{5/2} u_3, \eta^{5/2} u_4, \eta^{3} u_5)
$.

With $p_0=\left(\sqrt{\eta}h_1,\sqrt{\eta}h_2,h_3,h_4,\eta^{-1}\right)$, uniformly with respect to $h\in B_R$ as $\eta \to 0$, we have at $p_0$
$$
\det\left( \mathfrak{d}_\eta(U(h)),V,W,\partial_3F,\partial_4F\right) = \eta^{10} \frac{8\pi}{b_1 b_2^2}\sin^2\left( \frac{\pi b_2}{b_1}\right)\Phi(U_1(h),U_2(h),U_5(h))+o\left(\eta^{10}\right).
$$
\end{lemma}

\begin{proof}
We compute the dominant term of  $\det\left(\mathfrak{d}_\eta(U(h)),V,W,\partial_3F,\partial_4F\right)$.
Similarly to what is done in the proof of Lemma~\ref{L:reduced_cond_Case1}, we get from the assumptions and the construction of $V$ and $W$ in Section~\ref{SS:covec_near_S1}
$$
\det\left( \mathfrak{d}_\eta(U(h)),V,W,\partial_3F,\partial_4F\right)
=
\left|
\begin{array}{c c c|c}
\:\:\begin{matrix} \eta^{5/2}U_1 \\ \eta^{5/2}U_2 \end{matrix}&	\:\:\begin{matrix} \eta^{5/2}V_1^{(5/2)} \\ \eta^{5/2}V_2^{(5/2)} \end{matrix} & \:\:\begin{matrix} \eta^{5/2}\frac{2\pi}{b_1} h_1 \\  \eta^{5/2}\frac{2\pi}{b_1}h_2 \end{matrix}&	\:\:\begin{matrix} 0 &0\\ 0 &0\end{matrix}
\\
\hline
	\:\:\begin{matrix} 0 \\ 0 \end{matrix} & 	\:\:\begin{matrix} 0 \\ 0 \end{matrix} & 	\:\:\begin{matrix} 0 \\ 0 \end{matrix} & \eta A_4^{(1)}
\\
\hline
\eta^{3}U_5&0&\eta^3 \frac{2 \pi}{b_1}K&0\quad 0
\end{array}
\right|
+o(\eta^{10}).
$$
Hence the statement since 
$
\Phi(U_1,U_2,U_5)=
\begin{vmatrix}
	\:\:\begin{matrix} U_1 \\ U_2 \end{matrix}&	\:\:\begin{matrix} V_1 \\ V_2 \end{matrix} & \:\:\begin{matrix} h_1 \\  h_2 \end{matrix}
\\
U_5&0&K
\end{vmatrix}
$
and 
$
\det A_4^{(1)}=\frac{4}{b_2^2}\sin^2\left( \frac{\pi b_2}{b_1}\right)$.
\end{proof}

Let $q_0\in M\setminus \S$ and $p_0=\left(\sqrt{\eta}h_1,\sqrt{\eta}h_2,h_3,h_4,\eta^{-1}\right)\in T^*_{q_0}M$.
We can separate cases depending on the dimension of $\ker\Jac_{p_0} \sre_{q_0}(t_c(p_0))$.

Let us first treat the case of a $2$-dimensional kernel.
Let $S^+$ be the subset of $T_{q_0}^*M$ on which $\dim \ker\Jac_{p_0} \sre_{q_0}(t_c(p_0))=2$.
Following the analysis in Remark~\ref{R:2dim_ker}, singular points with dimension 2 kernel on $M\setminus \S$ correspond to covectors such that $(h_1,h_2)\neq (0,0)$ and 
$$
\gamma_{12}=-\frac{2 \pi  h_1 h_2}{b_1  },\quad
\gamma_{21}=-\frac{2 \pi  h_1 h_2}{b_1 K},\quad
\gamma_{22}-\gamma_{11}=\frac{2 \pi  \left(h_1^2-h_2^2\right)}{b_1 K}.
$$

Furthermore, $\ker\Jac_{p_0} \sre_{q_0}(t_c(p_0))$ is generated by $v_{\theta_1},v_{r_1}$, hence we choose the coordinates $\mathrm{x}_1,\mathrm{x}_2$ such that $\Span(\partial_{\x_1}\mathrm{id},\partial_{\x_2}\mathrm{id})=\Span(v_{\theta_1},v_{r_1})$, and we can check that the singularity is always of type $\D_4^+$ at covectors of $S^+$.

Assume now that the kernel of $\Jac\sre_{q_0}(t_c(p_0))$ is $1$-dimensional. As a consequence of Proposition~\ref{L:kernel_r1r2>0}, assuming $(h_1,h_2)\neq (0,0)$,  the kernel is generatedby $v=\lambda_{\theta_1}v_{\theta_1}+\lambda_{r_1}v_{r_1}$. We choose the first 
coordinate $\x_1:M\to \R$ such that  $\partial_{\x_1} = \sum_{i=1}^{5}{v}_i\partial_i$. Since $v_3,v_4=O(\eta)$, we have that $\partial_{x_1}^k F=O(\eta^{5/2})$ and $\partial_{x_1}^k F_5=O(\eta^{3})$ for all integer $k\geq 2$.

If we denote $V':\R^4\to \R^5$ such that (coordinate-wise) $\partial_{x_1}^2 F=\mathfrak{d}_\eta(V'(h))+o(\mathfrak{d}_\eta(1))$ then let $\Phi_2(h)=\Phi(V'_1,V'_2,V'_5)$.
Similarly, define $V'':\R^4\to \R^5$ such that $\partial_{x_1}^3 F=\mathfrak{d}_\eta(V''(h))+o(\mathfrak{d}_\eta(1))$, $V''':\R^4\to \R^5$ such that $\partial_{x_1}^4 F=\mathfrak{d}_\eta(V'''(h))+o(\mathfrak{d}_\eta(1))$ and  $V'''':\R^4\to \R^5$ such that $\partial_{x_1}^5 F=\mathfrak{d}_\eta(V''''(h))+o(\mathfrak{d}_\eta(1))$; and define
$\Phi_3(h)=\Phi(V''_1,V''_2,V''_5)$, $\Phi_4(h)=\Phi(V'''_1,V'''_2,V'''_5)$, $\Phi_5(h)=\Phi(V''''_1,V''''_2,V''''_5)$.

We numerically check that singular values of the exponential corresponding to covectors $p_0$ such that $(h_1,h_2)=(0,0)$ are of type $\A_3$ (it is immediate by passing to the limit if the conjugate time at $p_0$ is not double)
As an application of Lemma~\ref{L:reduced_cond_Case2}, and the analysis of the Jacobian matrix of $\sre_{q_0}(t_c(p_0))$ of Section~\ref{SS:covec_near_S1}, we  obtain that  for $\eta$ small enough
$$
\Phi_2(h)\neq 0 \Rightarrow \phi_{11}(p_0)\neq 0,
\qquad
\Phi_3(h)\neq 0 \Rightarrow \phi_{111}(p_0)\neq 0,
$$
$$
\Phi_4(h)\neq 0 \Rightarrow \phi_{1111}(p_0)\neq 0,
\qquad
\Phi_5(h)\neq 0 \Rightarrow \phi_{11111}(p_0)\neq 0.
$$

\subsection{Third domain: initial covectors near $S_2$}
\label{A:dom3}
\subsubsection{Jacobian matrix}\label{SS:in_cov_near_S2}

We now consider initial covectors of the form 
$$
p_0=\left( h_1, h_2, \eta h_3, \eta h_4,\eta^{-1}\right).
$$
The approach here is similar to Section~\ref{SSS:in_cov_T_setminS1S2}, however we need two orders of approximation. 
For two matrices $A,B\in \mathcal{M}_n(\R)$, and two vectors $u,v\in \R^n$, having $(A+\eta B)(u+\eta v)=0$ yields $Au=0$ and $Av+Bu=0$.
This relates to the computation of the conjugate time in Section~\ref{SS:third_order_conjugate}, but we only proved $\det (A+\eta B)=o(\eta)$, hence the existence {\it a priori} of $u\in \R^n$ such that $Au=0$ but not of $v\in \R^n$ such that $Av+Bu=0$.
\begin{lemma}\label{L:rank_kernel}
Let $A,B\in \mathcal{M}_n(\R)$. If $\mathrm{rank}(A)=n-1$ and $\det (A+\eta B)=o(\eta)$ as $\eta\to 0$ then $B \cdot \ker A\subset \im A$.
\end{lemma}
\begin{proof}
Since $\mathrm{rank}(A)=n-1$, there exists $P,Q\in \mathrm{GL}_n(\R)$ such that $A=PA'Q$, with $A'$ the diagonal matrix with diagonal $(0,1,\dots,1)$. Let $u\in \ker A\setminus \{0\}$. Then $Qu$ is colinear to $e_1=(1,0,\dots ,0)$. Without loss of generality, we can assume $Qu=e_1$. Then, denoting $B'=P^{-1}BQ^{-1}$, $Bu\in \im A$ is equivalent to $B'e_1\in \im A'$, that is $B'_{11}=0$.

On the other hand $\det(A+\eta B)=o(\eta)$ implies $\det(A'+\eta B')=o(\eta)$, and developing the determinant with respect to $\eta$ yields
$\det(A'+\eta B')=\eta B'_{11}+o(\eta)$. Hence the result.
\end{proof}

\begin{proposition}\label{L:kernel_r2=0}
The kernel of $\Jac_{p_0} \sre_{q_0}(t_c(p_0))$ is $1$-dimensional and there exists $\nu(h_1,h_2)\in \R$, $\mu(h_1,h_2,h_3,h_4)\in \R$, such that $\ker\Jac_{p_0} \sre_{q_0}(t_c(p_0))$ is generated by the vector 
$$
\left(-h_2,h_1,*,*,\nu\right)
+
\eta 
\left[
\frac{-5 \nu}{4}
\left(h_1,h_2,*,*,0\right)
+
\mu
\left(-\nu h_2,\nu h_1,*,*,-\left( h_1^2+h_2^2\right)\right)
\right]
+O(\eta^2).
$$
\end{proposition}

\begin{proof}

From computations in Section \ref{SS:third_order}, we have 
$$
A_1=O(\eta^{2}), \quad A_3=O(\eta^{2}),\quad A_4=O(\eta^{2}), \quad \text{ and }\quad A_3= O(\eta^3),
$$
$$
C_1=O(\eta^{2}), \quad L_1=O(\eta^{2}),\quad C_2=O(\eta^{3}), \quad L_2=O(\eta^{3}), \quad E=O(\eta^3).
$$
Equation~\eqref{E:ker5} then implies 
$$
L_1
\begin{pmatrix}
v_1
\\
v_2
\end{pmatrix}=O(\eta^3).
$$ 
Hence, as previously, there exists $\lambda \in \R$ such that $(v_1 ,v_2 )=\lambda (-h_2,h_1)+O(\eta)$. Now however, since $A_4=O(\eta^{2})$ and $A_2=O(\eta^{3})$,
$$
A_3 
\begin{pmatrix}
v_1 
\\
v_2 
\end{pmatrix}
+
A_4
\begin{pmatrix}
v_3
\\
v_4
\end{pmatrix}
=O(\eta^3)
$$
and
$$
A_1
\begin{pmatrix}
v_1
\\
v_2
\end{pmatrix}
+ C_1 v_5=O(\eta^3).
$$
Hence we have 
$$
v_5^{(0)} C_1^{(2)}=-A_1^{(2)}
\begin{pmatrix}
v_1^{(0)}
\\
v_2^{(0)}
\end{pmatrix},
\quad
\text{ and }
\quad
\begin{pmatrix}
v_3^{(0)}
\\
v_4^{(0)}
\end{pmatrix}
=
-\left(A_4^{(2)}\right)^{-1}A_3^{(2)}
\begin{pmatrix}
v_1^{(0)}
\\
v_2^{(0)}
\end{pmatrix}.
$$
The lower semi-continuity of the rank  implies that the kernel is indeed $1$-dimensional. We can apply Lemma~\ref{L:rank_kernel} and compute $v^{(1)}\in \ker A^{\perp}$ such that (focusing on $v_1^{(1)},v_2^{(1)},v_5^{(1)}$)
\begin{equation}\label{E:order2r2=0}
A_1^{(2)}
\begin{pmatrix}
v_1^{(1)}
\\
v_2^{(1)}
\end{pmatrix}
+
\left(
A_1^{(3)}
-
A_2^{(3)}
\left(A_4^{(2)}\right)^{-1}A_3^{(2)}\right)
\begin{pmatrix}
v_1^{(0)}
\\
v_2^{(0)}
\end{pmatrix}
+
v_5^{(0)} C_1^{(3)}
+
v_5^{(1)} C_1^{(2)}
=0
\end{equation}
\begin{equation}\label{E:order2r2=0bis}
L_1^{(2)}
\begin{pmatrix}
v_1^{(1)}
\\
v_2^{(1)}
\end{pmatrix}
+
L_1^{(3)}
\begin{pmatrix}
v_1^{(0)}
\\
v_2^{(0)}
\end{pmatrix}
+
v_5^{(0)} E^{(3)}
=0.
\end{equation}

We can assume $(h_1,h_2)\neq (0,0)$, since we are considering covectors near $S_2$ but not $S_1$. Still focusing on $v_1^{(1)},v_2^{(1)},v_5^{(1)}$ and  looking for solutions in $\ker A^\perp$, we use a more suited basis of $\R^3$. We have $\nu$ such that 
$
\nu C_1^{(2)}=-A_1^{(2)}
\begin{pmatrix}
-h_2
\\
h_1
\end{pmatrix}$,
so that with
 $f_1=(-h_2,h_1,\nu)$, $\left(v_1^{(0)},v_2^{(0)},v_5^{(0)}\right)=\lambda f_1$. Then we set $f_2=(h_1,h_2,0)$ and $f_3=(-\nu h_2,\nu h_1,-(h_1^2+h_2^2))$, and  $ \left(v_1^{(1)},v_2^{(1)},v_5^{(1)}\right)=\mu_2 f_2+\mu_3 f_3$.

Then Equations~\eqref{E:order2r2=0}-\eqref{E:order2r2=0bis} yield
$$
\mu_2 A_1^{(2)}e_{r_1}
+
\mu_3 \nu  A_1^{(2)}e_{\theta_1}
+
\lambda
\left(
A_1^{(3)}
-
A_2^{(3)}
\left(A_4^{(2)}\right)^{-1}A_3^{(2)}\right)
e_{\theta_1}
+
\lambda \nu  C_1^{(3)}
-
\mu_3(h_1^2+h_2^2) C_1^{(2)}
=0,
$$
$$
\mu_2 L_1^{(2)}
e_{r_1}
+
\lambda
L_1^{(3)}
e_{\theta_1}
+
\lambda \nu E^{(3)}
=0.
$$
Then $\mu_2=-\frac{\lambda}{L_1^{(2)}e_{r_1}}\left( L_1^{(3)}e_{\theta_1}+ \nu E^{(3)}\right)=-5\lambda \nu/4$ (see the proof of Lemma \ref{L:expression_d_i} to find an explicit  expression of $L_1^{(3)}$ and $E^{(3)}$) and 
$$
-\lambda \frac{5}{4} \nu A_1^{(2)}e_{r_1}
+
\lambda
\left(
A_1^{(3)}
-
A_2^{(3)}
\left(A_4^{(2)}\right)^{-1}A_3^{(2)}\right)
e_{\theta_1}
+
\lambda \nu  C_1^{(3)}
=\mu_3(h_1^2+h_2^2+\nu^2) C_1^{(2)}.
$$
Hence the result with $\mu=\mu_3/\lambda$.
\end{proof}

Again, we end the section with a handy description of the image of $\Jac_{p_0} \sre_{q_0}(t_c(p_0))$.
Let 
$$
V'=h_1\partial_1\bar{F}+h_2\partial_2\bar{F}- w_3\partial_3\bar{F}- w_4\partial_4\bar{F},\text{ where }
\begin{pmatrix}
w_3
\\
w_4
\end{pmatrix}
=
\left(A_4^{(2)}\right)^{-1}A_3^{(2)}
\begin{pmatrix}
h_1
\\
h_2
\end{pmatrix},
$$
so that $(V'_3,V'_4)=O(\eta^3)$.

\begin{lemma}\label{L:image_Jac_r2=0}
Let $p_0=\left( h_1, h_2, \eta h_3, \eta h_4,\eta^{-1}\right)\in T^*_{q_0}M$.
The image of the Jacobian matrix at $p_0$ of the exponential at the conjugate time admits the representation
$$
\im \Jac_{p_0} \sre_{q_0}(t_c(p_0))
=
\Span
\left\{
V',\partial_3F,\partial_4F,\partial_5F
\right\}.
$$
\end{lemma}
\begin{proof}
The proof is again straightforward.  With $v$ generating $\ker \Jac_{p_0} \sre_{q_0}(t_c(p_0))$, as given by Proposition~\ref{L:kernel_r2=0}, 
 $v'=(h_1,h_2,w_3,w_4,0)$, $v_3=(0,0,1,0,0)$, $v_4=(0,0,0,1,0)$, $v_5=(0,0,0,0,1)$,
it is immediate that 
$$
\mathrm{rk}(v,v',v_3,v_4,v_5)=5.
$$
Hence the result, similarly to Lemma~\ref{L:image_Jac_r1r2>0}.
\end{proof}

\subsubsection{Classification}

We repeat the process one last time, except we now need two orders of approximation.

\begin{lemma}\label{L:reduced_cond_Case3}
Let $V'$  be as in the statement of Lemma~\ref{L:image_Jac_r2=0}.
For all $i\in\ll 1,5\rr$,  let $U,U': \R^4\to \R^5$ and
 for $u,u'\in \R^5$, let
$$
\Psi'(u)=
b_1 u_5  \left(\alpha h_2- \beta h_1\right)+ \pi \left(h_1 u_2-h_2 u_1\right)
$$
and 
$$
\begin{aligned}
\Gamma(u,u')=
&
	\Psi'(u') +
\frac{27b_1}{4}  (\alpha  h_{1}+\beta  h_{2}) (h_{2} u_{1}-h_{1} u_{2})-\frac{b_1}{\pi}  (\alpha  h_1+\beta  h_2) \Psi'(u)
\\&
+\frac{b_1 \,u_5}{2(h_1^2+h_2^2)}  \left(h_2 {V'_1}^{(3)}-h_1 {V'_2}^{(3)}\right)
\\
&
+\frac{b_2}{2} \Psi'(u)
\left[\vphantom{\frac{\pi  b_2}{b_1}}
h_1 (\kappa_{14}^3-\kappa_{13}^4)
+
h_2 (\kappa_{24}^3-\kappa_{23}^4)
+\right.
\\
&
\hphantom{+\frac{b_2}{2} \psi'(u)}\qquad \left. \cot \left(\frac{\pi  b_2}{b_1}\right)
\left(2 \tau_c^{(1)}(h)+
h_1 (\kappa_{13}^3+\kappa_{14}^4)
+
h_2 (\kappa_{23}^3+\kappa_{24}^4)
\right)
\right]
\\
&+
\frac{ b_2}{ 2} 
\left(U_4-U_3\cot \left(\frac{\pi  b_2}{b_1}\right)\right)
\left(
\kappa^{1,3}_2 h_1^2-\kappa^{2,3}_1 h_2^2+( \kappa^{2,3}_2-\kappa^{1,3}_1)h_1h_2
\right) 
\\&
+
\frac{ b_2}{ 2}
 \left(  U_3 + U_4\cot \left(\frac{\pi  b_2}{b_1}\right)\right)
 \left(
\kappa^{2,4}_1 h_2^2-\kappa^{1,4}_2 h_1^2+(\kappa^{1,4}_1- \kappa^{2,4}_2)h_1h_2
\right).
\end{aligned}
$$

With $p_0=\left( h_1, h_2, \eta h_3,\eta h_4,\eta^{-1}\right)$, uniformly with respect to $h\in B_R$ as $\eta \to 0$, we have at $p_0$
\begin{multline*}
\det\left( U(h)+\eta U'(h),V, \partial_3F,\partial_4F,\partial_5F \right) =\\
 \eta^{8} \frac{16\pi(h_1^2+h_2^2)}{b_1^2b_2^2}\sin^2\left( \frac{\pi b_2}{b_1}\right) \left[\vphantom{2^{2}}\Psi'(U(h))+\eta \Gamma (U(h),U'(h))\right]
 +o(\eta^9).
\end{multline*}
\end{lemma}

\begin{proof}
We compute the first two non-zero terms in  the expansion of 
$$
\det\left( U(h)+\eta U'(h),V, \partial_3F,\partial_4F,\partial_5F \right). 
$$
Observe that 
$$
V=\eta^2V^{(2)}+\eta^3V^{(3)}+o(\eta^3) 
\quad \text{ and} \quad 
\partial_iF=\eta^2\partial_iF^{(2)}+\eta^3\partial_iF^{(3)}+o(\eta^3)
\quad 
\forall i\in \ll3,5\rr.
$$
Notice that  $\det\left( U(h),V^{(2)}, \partial_3F^{(2)},\partial_4F^{(2)},\partial_5F^{(2)} \right) =\frac{4 \pi K'}{b_1^2}(h_1^2+h_2^2)\psi'(U(h))$ (recall $K'=\det\left(A_4^{(2)}\right)=\frac{4}{b_2^2}\sin^2\left( \frac{\pi b_2}{b_1}\right)$).
Then
$$
\begin{aligned}
\det\left( U(h) ,V, \partial_3F,\partial_4F,\partial_5F \right) =&
 \eta^{8}  \det\left( U(h),V^{(2)}, \partial_3F^{(2)},\partial_4F^{(2)},\partial_5F^{(2)} \right) 
 \\&
+ \eta^{9} K' (d_1+d_2+d_3+d_4+d_5)+o(\eta^9),
\end{aligned}
$$
 with
$$
K'd_1=  \det\left( U'(h),V^{(2)}, \partial_3F^{(2)},\partial_4F^{(2)},\partial_5F^{(2)} \right) =K'\left|
\begin{array}{c|c|c}
\:\:\begin{matrix} U_1' \\ U_2' \end{matrix} & A_1^{(2)}e_{r_1} & C_1^{(2)} 
\\
\hline
U_5'&\frac{2\pi}{b_1}(h_1^2+h_2^2)&0
\end{array}
\right|
$$
$$
K'd_2=\det\left( U(h),V^{(3)}, \partial_3F^{(2)},\partial_4F^{(2)},\partial_5F^{(2)} \right)=
K'
\left|
\begin{array}{c|c|c}
\:\:\begin{matrix} U_1 \\ U_2 \end{matrix} & A_1^{(3)}e_{r_1}-A_2^{(3)}\left(A_4^{(2)}\right)^{-1}A_3^{(2)}e_{r_1} & C_1^{(2)} 
\\
\hline
U_5&L_1^{(3)} e_{r_1}&0
\end{array}
\right|
$$

$$
K'd_5=\det\left( U(h),V^{(2)}, \partial_3F^{(2)},\partial_4F^{(2)},\partial_5F^{(3)} \right)=K'
\left|
\begin{array}{c|c|c}
\:\:\begin{matrix} U_1 \\ U_2 \end{matrix} & A_1^{(2)}e_{r_1} & C_1^{(3)} 
\\
\hline
U_5&\frac{2\pi}{b_1}(h_1^2+h_2^2)&E^{(3)}
\end{array}
\right|
$$
and
\begin{multline*}
K'd_3=\det\left( U(h),V^{(2)}, \partial_3F^{(3)},\partial_4F^{(2)},\partial_5F^{(2)} \right)=\frac{2\pi}{b_1}(h_1^2+h_2^2)\\
\left(
\left|
\begin{array}{cc}
U_3&\left(A_4^{(2)}\right)_{1,2} 
\\
U_4& \left(A_4^{(2)}\right)_{2,2} 
\end{array}
\right|
\left|
\begin{array}{cc}
\left(A_2^{(3)}\right)_{1,1} & \left(C_1^{(2)}\right)_1
\\
\left(A_2^{(3)}\right)_{2,1} & \left(C_1^{(2)}\right)_2
\end{array}
\right|
\right.
+
\left.
\frac{2\psi'(U)}{b_1} 
\left|
\begin{array}{cc}
\left(A_4^{(3)}\right)_{1,1} &\left(A_4^{(2)}\right)_{1,2} 
\\
\left(A_4^{(3)}\right)_{2,1} & \left(A_4^{(2)}\right)_{2,2} 
\end{array}
\right|
\right),
\end{multline*}
\begin{multline*}
K'd_4=\det\left( U(h),V^{(2)}, \partial_3F^{(2)},\partial_4F^{(3)},\partial_5F^{(2)} \right)=
-\frac{2\pi}{b_1}(h_1^2+h_2^2)
\\
\left(
\begin{vmatrix}
	U_3&\left(A_4^{(2)}\right)_{1,1} 
	\\
	U_4& \left(A_4^{(2)}\right)_{2,1} 
\end{vmatrix}
\begin{vmatrix}
	\left(A_2^{(3)}\right)_{1,2} & \left(C_1^{(2)}\right)_1
	\\
	\left(A_2^{(3)}\right)_{2,2} & \left(C_1^{(2)}\right)_2
\end{vmatrix}
\right.
+
\left.
\frac{2\psi'(U)}{b_1} 
\left|
\begin{array}{cc}
\left(A_4^{(3)}\right)_{1,2} &\left(A_4^{(2)}\right)_{1,1} 
\\
\left(A_4^{(3)}\right)_{2,2} & \left(A_4^{(2)}\right)_{2,1} 
\end{array}
\right|
\right).
\end{multline*}

Hence the statement by summation.
\end{proof}

Let $q_0\in M\setminus \S$ and $p_0=\left( h_1, h_2,\eta h_3,\eta h_4,\eta^{-1}\right)\in T^*_{q_0}M$.
Let $p_0\in T_{q_0}^*M$ and $v$ be as in the statement of Proposition~\ref{L:kernel_r2=0} so that $\ker \Jac_{p_0} \sre_{q_0}(t_c(p_0))=\Span(v)$. As explained in Remark~\ref{R:method_derivatives}, we choose the first
coordinate $\x_1:M\to \R$ such that  $\partial_{\x_1} = \sum_{i=1}^{5}{v}_i\partial_i$ and we have that $\partial_{x_1}^k F=O(\eta^2)$ for all integer $k\geq 2$.

If we denote $V',W':\R^4\to \R^5$ such that $\partial_{x_1}^2 F=\eta^2 V'(h)+\eta^3W'(h)+o(\eta^3)$ then let $\Psi'_2(h)=\Psi'(V')$ and $\Gamma_2(h)=\Gamma(V',W')$.
Similarly, define $V'',W'':\R^4\to \R^5$ such that $\partial_{x_1}^3 F=\eta^2 V''(h)+\eta^3W''(h)+o(\eta^2)$, $V''',W''':\R^4\to \R^5$ such that $\partial_{x_1}^4 F=\eta^2 V'''(h)+\eta^3 W'''(h)+o(\eta^2)$ and $V'''',W'''':\R^4\to \R^5$ such that $\partial_{x_1}^5 F=\eta^2 V''''(h)+\eta^3 W''''(h)+o(\eta^2)$; and define
$\Psi'_3(h)=\Psi'(V'')$, $\Gamma_3(h)=\Gamma(V'',W'')$,
$\Psi'_4(h)=\Psi'(V''')$, $\Gamma_4(h)=\Gamma(V''',W''')$, and 
$\Psi'_5(h)=\Psi'(V'''')$, $\Gamma_5(h)=\Gamma(V'''',W'''')$.

We would like to replicate what has been done in the previous two sections in regard of the functions $\Psi'_i$. However we can check that $\Psi'_i=0$ for $i\in \ll2,5\rr$ and we should instead focus on the functions $\Gamma_i$.
As an application of Lemma~\ref{L:reduced_cond_Case3}, and the analysis of the Jacobian matrix of $\sre_{q_0}(t_c(p_0))$ of Section~\ref{SS:in_cov_near_S2},
we immediately obtain that  for $\eta$ small enough
$$
\Gamma_2(h)\neq 0 \Rightarrow \phi_{11}(p_0)\neq 0,
\qquad
\Gamma_3(h)\neq 0 \Rightarrow \phi_{111}(p_0)\neq 0,
$$
$$
\Gamma_4(h)\neq 0 \Rightarrow \phi_{1111}(p_0)\neq 0,
\qquad
\Gamma_5(h)\neq 0 \Rightarrow \phi_{11111}(p_0)\neq 0.
$$

\newpage


\bibliographystyle{abbrv}
\bibliography{biblio}

\begin{thebibliography}{10}

\bibitem{ABB_2018}
A.~Agrachev, D.~Barilari, and U.~Boscain.
\newblock {\em A Comprehensive Introduction to Sub-Riemannian geometry}.
\newblock Cambridge University Press, 2018.

\bibitem{agrachev_barilari_rizzi_contact_2016}
A.~Agrachev, D.~Barilari, and L.~Rizzi.
\newblock {Sub-Riemannian curvature in contact geometry}.
\newblock {\em {Journal of Geometric Analysis}}, 2016.

\bibitem{Gauthier_2001_SR_metrics_and_isoperimetric_problems}
A.~Agrachev and J.~Gauthier.
\newblock Sub-{R}iemannian metrics and isoperimetric problems in the contact
  case.
\newblock {\em Journal of Mathematical Sciences}, 103(6):639--663, 2001.

\bibitem{agrachev_1996_exponential}
A.~A. Agrachev.
\newblock Exponential mappings for contact sub-{R}iemannian structures.
\newblock {\em J. Dynam. Control Systems}, 2(3):321--358, 1996.

\bibitem{AgrachevCharlotGauthierZakalyukin}
A.~A. Agrachev, G.~Charlot, J.~P.~A. Gauthier, and V.~M. Zakalyukin.
\newblock On sub-{R}iemannian caustics and wave fronts for contact
  distributions in the three-space.
\newblock {\em J. Dynam. Control Systems}, 6(3):365--395, 2000.

\bibitem{Arnold_singularities_diff_maps1}
V.~I. Arnold, S.~M. Guse\u\i~n Zade, and A.~N. Varchenko.
\newblock {\em Singularities of differentiable maps. {V}ol. {I}}, volume~82 of
  {\em Monographs in Mathematics}.
\newblock Birkh\"auser Boston, Inc., Boston, MA, 1985.
\newblock The classification of critical points, caustics and wave fronts,
  Translated from the Russian by Ian Porteous and Mark Reynolds.

\bibitem{Barilari2013}
D.~Barilari.
\newblock Trace heat kernel asymptotics in 3d contact sub-riemannian geometry.
\newblock {\em Journal of Mathematical Sciences}, 195(3):391--411, Dec 2013.

\bibitem{barilari2018VolumeOfBalls}
D.~Barilari, I.~Beschastnyi, and A.~Lerario.
\newblock Volume of small balls and sub-riemannian curvature in 3d contact
  manifolds, 2018.

\bibitem{barilari2016heat}
D.~Barilari, U.~Boscain, G.~Charlot, and R.~W. Neel.
\newblock On the heat diffusion for generic riemannian and sub-riemannian
  structures.
\newblock {\em International Mathematics Research Notices},
  2017(15):4639--4672, 2016.

\bibitem{barilari2012_small_time_heat_kernel}
D.~Barilari, U.~Boscain, and R.~W. Neel.
\newblock Small-time heat kernel asymptotics at the sub-riemannian cut locus.
\newblock {\em J. Differential Geom.}, 92(3):373--416, 11 2012.

\bibitem{IHP-Vol}
D.~Barilari, U.~Boscain, and M.~Sigalotti, editors.
\newblock {\em Geometry, analysis and dynamics on sub-{R}iemannian manifolds.
  {V}ol. 1 \& 2}.
\newblock EMS Series of Lectures in Mathematics. European Mathematical Society
  (EMS), Z\"urich, 2016.
\newblock Lecture notes from the IHP Trimester held at the Institut Henri
  Poincar\'e, Paris and from the CIRM Summer School ``Sub-Riemannian Manifolds:
  From Geodesics to Hypoelliptic Diffusion'' held in Luminy, Fall 2014.

\bibitem{bellaiche1996TagentSpace}
A.~Bella\"iche.
\newblock The tangent space in sub-{R}iemannian geometry.
\newblock In {\em Sub-{R}iemannian geometry}, volume 144 of {\em Progr. Math.},
  pages 1--78. Birkh\"auser, Basel, 1996.

\bibitem{Bennequin_caustiques_mystiques}
D.~Bennequin.
\newblock Caustique mystique.
\newblock {\em S\'eminaire Bourbaki}, 27:19--56, 1984-1985.

\bibitem{Biggs2016}
R.~Biggs and P.~T. Nagy.
\newblock On sub-{R}iemannian and {R}iemannian structures on the heisenberg
  groups.
\newblock {\em Journal of Dynamical and Control Systems}, 22(3):563--594, Jul
  2016.

\bibitem{boscain2015intrinsic}
U.~Boscain, R.~Neel, and L.~Rizzi.
\newblock Intrinsic random walks and sub-laplacians in sub-{R}iemannian
  geometry.
\newblock {\em Advances in Mathematics}, 314(Supplement C):124 -- 184, 2017.

\bibitem{charlot_2002_quasi_contact}
G.~Charlot.
\newblock Quasi-contact {S}-{R} metrics: normal form in {$\mathbf{R}^{2n}$},
  wave front and caustic in {$\mathbf{R}^4$}.
\newblock {\em Acta Appl. Math.}, 74(3):217--263, 2002.

\bibitem{Diniz_veloso2009}
M.~M. Diniz and J.~M.~M. Veloso.
\newblock Regions where the exponential map at regular points of
  sub-{R}iemannian manifolds is a local diffeomorphism.
\newblock {\em J. Dyn. Control Syst.}, 15(1):133--156, 2009.

\bibitem{Gauthier_1996_small_SR_balls}
E.-H.~C. El-Alaoui, J.~Gauthier, and I.~Kupka.
\newblock Small sub-{R}iemannian balls on {$\mathbf{R}^3$}.
\newblock {\em J. Dynam. Control Systems}, 2(3):359--421, 1996.

\bibitem{gaveau1977principe}
B.~Gaveau.
\newblock Principe de moindre action, propagation de la chaleur et estim{\'e}es
  sous elliptiques sur certains groupes nilpotents.
\newblock {\em Acta mathematica}, 139(1):95--153, 1977.

\bibitem{HughenThesis}
W.~K. Hughen.
\newblock {\em The sub-{R}iemannian geometry of three-manifolds}.
\newblock ProQuest LLC, Ann Arbor, MI, 1995.
\newblock Thesis (Ph.D.)--Duke University.

\bibitem{izumiya2016differential}
S.~Izumiya, M.~d. C.~R. Fuster, M.~A.~S. Ruas, F.~Tari, et~al.
\newblock {\em Differential geometry from a singularity theory viewpoint}.
\newblock World Scientific, 2016.

\bibitem{lerario2017many}
A.~Lerario and L.~Rizzi.
\newblock How many geodesics join two points on a contact sub-riemannian
  manifold?
\newblock {\em Journal of Symplectic Geometry}, 15(1):247--305, 2017.

\bibitem{montgomery2006tour}
R.~Montgomery.
\newblock {\em A tour of subriemannian geometries, their geodesics and
  applications}.
\newblock Number~91. American Mathematical Soc., 2006.

\end{thebibliography}

\end{document}